\documentclass[12pt]{article} 
\usepackage{verbatim, latexsym, amssymb, amsmath,color, amsthm}
\usepackage{enumitem}
\usepackage{epsfig}
\usepackage{xcolor}
\usepackage{mathrsfs}
\usepackage{stmaryrd}

\usepackage{blindtext}
\usepackage{sectsty}
\sectionfont{\centering}

\usepackage{geometry}
\usepackage{hyperref}

\DeclareMathOperator{\Ric}{Ric}

\DeclareMathOperator{\spt}{spt}
\DeclareMathOperator{\Id}{Id}

\DeclareMathOperator{\Vol}{Vol}

\DeclareMathOperator{\dist}{dist}

\DeclareMathOperator{\Sing}{Sing}

\DeclareMathOperator{\minvol}{MinVol}
\DeclareMathOperator{\minent}{MinEnt}

\DeclareMathOperator{\spherevol}{SphereVol}

\DeclareMathOperator{\Reg}{\mathrm{Reg}}

\DeclareMathOperator{\Induced}{Ind}

\DeclareMathOperator{\End}{End}

\newtheorem{theo}{Theorem}[section]
\newtheorem{prop}[theo]{Proposition}
\newtheorem{lemme}[theo]{Lemma}
\newtheorem{definition}[theo]{Definition}
\newtheorem{coro}[theo]{Corollary}

\newtheorem{conj}[theo]{Conjecture}

\theoremstyle{definition}
\newtheorem{remarque}[theo]{Remark}

\begin{document}
\title{Hyperbolic groups and spherical minimal surfaces}

\author{Antoine Song}

\AtEndDocument{\bigskip{\footnotesize%
  \textsc{\\California Institute of Technology\ 177 Linde Hall, \\ \#1200 E. California Blvd., Pasadena, CA 91125}\par 
  \textit{\\E-mail address}: \texttt{aysong@caltech.edu}
  }}

\date{}
\maketitle

\begin{abstract} 
Let $M$ be a closed, oriented, negatively curved, $n$-dimensional manifold with fundamental group $\Gamma$.
Let $S^\infty$ be the unit sphere in $\ell^2(\Gamma)$, on which $\Gamma$ acts by the regular representation.
The spherical volume of $M$ is a topological invariant introduced by Besson-Courtois-Gallot.
We show that it is equal to the area of an $n$-dimensional area-minimizing minimal surface inside the ultralimit of $S^\infty/\Gamma$, in the sense of Ambrosio-Kirchheim. Our proof combines the theory of metric currents with a study  of limits of the regular representation of torsion-free hyperbolic groups. 
\end{abstract}

\tableofcontents

\section{Introduction}

The \emph{spherical volume} of a closed oriented manifold $M$ is an invariant, introduced by Besson-Courtois-Gallot in \cite[Section 3, I]{BCG91},  which measures the topological complexity of $M$.
It generalizes the hyperbolic volume of low dimensional manifolds \cite[Exemples 3.3 (d)]{BCG91} \cite[Equation (1.2)]{BCG95} \cite{BCG96} \cite[Theorem 5.1]{Antoine23a}, and
it has a rich interplay with a host of other invariants, as shown in the chain of inequalities (\ref{ineq chain}), 
Historically, it played a crucial role in the celebrated proof of the entropy rigidity conjecture of Katok-Gromov by Besson-Courtois-Gallot \cite{BCG96} (see also \cite[Subsection 4.2]{Antoine23a} for a review), which in turn resulted in an elementary proof of Mostow's rigidity theorem. Nevertheless, many basic questions about this invariant remain unanswered \cite{Kotschick11} \cite{Antoine23a}. The spherical volume of $M$ is nonzero only when $M$ is ``large'', for instance when $M$ admits a Riemannian metric of negative curvature \cite[Subsection 0.3]{Gromov82} \cite[Th\'{e}or\`{e}me 3.16]{BCG91}.  By combining tools from geometric measure theory and geometric group theory, we find that 
the spherical volume of a negatively curved $M$ is naturally realized by an $n$-dimensional ``mass-minimizing cycle'', a generalized notion of minimal surface which is central in geometric measure theory \cite[Introduction]{AK00}.

Let $M$ be a closed, oriented, negatively curved, Riemannian $n$-manifold. 
Set $\Gamma:=\pi_1(M)$, the fundamental group of $M$, let $\tilde{M}$ be the universal cover of $M$, on which $\Gamma$ acts properly, freely, cocompactly, and let $D_M$ be a Borel fundamental domain in $\tilde{M}$. 
Let $(S^\infty,\mathbf{g}_{\mathrm{Hil}})$ be the real separable Hilbert unit sphere of infinite dimension, with its standard round metric. 
We can identify it with the unit sphere in $\ell^2(\Gamma)$. 
Then $\Gamma$ acts isometrically on $S^\infty$ by translation, namely by the regular representation $\lambda_\Gamma$.
Consider the family $\mathscr{H}_{M}$ of $\Gamma$-equivariant smooth maps from $\tilde{M}$ to $S^\infty$.
Then the spherical volume of $M$ is defined as
\begin{equation}\label{first def sv}
\spherevol(M):= \inf\{\Vol(D_M,\phi^*\mathbf{g}_{\mathrm{Hil}});\quad \phi\in \mathscr{H}_M\}.
\end{equation}
Let $S^\infty/\Gamma$ be the quotient of $S^\infty$ by $\Gamma$ via the action $\lambda_\Gamma$. Denote by  $(S^\infty/\Gamma)_\omega$ the ultralimit of the constant sequence $S^\infty/\Gamma$ with respect to a non-principal ultrafilter $\omega$ \cite[Chapter 10]{DK18}. This ultralimit is still a quotient of a Hilbert unit sphere by an isometric group action, but it has complicated singularities. Our main result is the following:
\begin{theo} \label{theorem:plateau solution}
 Let $M$ be a closed, oriented, negatively curved $n$-manifold, with fundamental group $\Gamma$. Then, there is a mass-minimizing $n$-cycle $T$ inside the ultralimit $(S^\infty/\Gamma)_\omega$, whose mass is equal to $\spherevol(M)$.
Moreover, the intersection of the support of $T$ and the singular part of $(S^\infty/\Gamma)_\omega$ is a discrete set.
\end{theo}

In this statement, ``mass-minimizing $n$-cycle'' is taken in the sense of
the theory of metric currents, developed by Ambrosio-Kirchheim \cite{AK00}, Lang \cite{Lang11}, Wenger \cite{Wenger11} and others, This theory provides a powerful framework to study \emph{integral currents}, a kind of generalized notion of submanifolds, in any complete metric space $(E,d)$. An $n$-dimensional integral current $T$ in $(E,d)$ has well-defined mass $\mathbf{M}(T)$ and boundary $\partial T$. If $T$ has no boundary, namely $\partial T=0$, then $T$ is called an $n$-cycle. An integral current $T$ it is called \emph{mass-minimizing} \cite[Definition 9.1]{FF60} if for all $(n+1)$-dimensional integral current $R$ in $(E,d)$, the following mass inequality is true:
$$\mathbf{M}(T)\leq \mathbf{M}(T + \partial R).$$

In finite dimensions, a fundamental theorem due to Federer-Fleming \cite[Corollary 9.6]{FF60} states that in a closed Riemannian manifold, any homology class is represented by a mass-minimizing cycle, in the sense above. In infinite-dimensional Riemannian spaces, this existence theorem does not hold:  surprisingly, the quotient $S^\infty/\Gamma$ never contains nonzero mass-minimizing cycles \cite[Remark 3.8]{Antoine23a}. 
Thus, Theorem \ref{theorem:plateau solution} can be viewed as a corrected Federer-Fleming theorem for $S^\infty/\Gamma$.

The mass-minimizing $n$-cycle $T$ in Theorem \ref{theorem:plateau solution} is obtained as a limit of a minimizing sequence of cycles in $S^\infty/\Gamma$.
A unexpected feature is that $T$ is almost entirely supported in the smooth part of $(S^\infty/\Gamma)_\omega$, called $\Reg(S^\infty/\Gamma)_\omega $, which is an open domain locally isometric to a Hilbert unit sphere. The partial regularity result of Ambrosio-De Lellis-Schmidt \cite[Theorem 1.2]{ADLS18} is stated for mass-minimizing integral currents in Hilbert spaces, 
but a suitable modification of the argument allows to handle Hilbert spheres \cite{DL24}. Hence, the support of the cycle $T$ in Theorem \ref{theorem:plateau solution} is regular on an open dense subset. Conjecturally,  Almgren's big regularity theorem \cite[Theorem 0.1]{Ambrosio16} also holds in this infinite-dimensional setting, meaning that our $T$ should be smooth outside of a codimension 2 subset.
Minimal surfaces of dimension $n$ inside domains of Hilbert unit spheres have special analytic properties since they are isometrically embedded by Laplace eigenfunctions with eigenvalue $n$ \cite{Takahashi66}.



We will prove a version of Theorem \ref{theorem:plateau solution} which holds more generally for any homology class of $S^\infty/\Gamma$ whenever $\Gamma$ is assumed to be torsion-free \emph{hyperbolic} \cite{Gromov87}\cite{GDLH90}, a well-known family of groups in  geometric group theory.
There are many known constructions of torsion-free hyperbolic $\Gamma$ such that $S^\infty/\Gamma$ has nontrivial homology, see Section \ref{hyp groups}. The fact that $\spherevol(M)>0$ for negatively curved closed manifolds $M$ was known, by combining \cite[Subsection 0.3]{Gromov82} and \cite[Th\'{e}or\`{e}me 3.16]{BCG91}. Our arguments provide a different proof, see Corollary  \ref{positive spherevol}.

It is instructive to see what happens when $M$ is a hyperbolic $n$-manifold or more generally negatively curved locally symmetric. In that case, it essentially  follows from Besson-Courtois-Gallot \cite{BCG95}\cite{BCG96}\cite[Section 4.2]{Antoine23a} that $T$ in Theorem \ref{theorem:plateau solution} can be chosen to be given by the image $\mathscr{P}(M)$ of a special minimal embedding $\mathscr{P}$ from $M$ to a smooth spherical quotient $S^\infty/\rho_B(\Gamma)$ where $\rho_B$ is a certain  isometric, proper, free action of $\Gamma$ on $S^\infty$, and $S^\infty/\rho_B(\Gamma)$ is isometrically embedded inside $(S^\infty/\Gamma)_\omega$. The uniqueness of such $T$ is conjectured when $n\geq 3$ \cite[Question 4.5]{Antoine23a}.

In Theorem \ref{theorem:plateau solution}, can $(S^\infty/\Gamma)_\omega$ be replaced by some smooth spherical quotient $S^\infty/\rho(\Gamma)$, 
as for negatively curved locally symmetric manifolds?  
In that direction, we propose the conjecture below, which we discuss in more details in Remark \ref{plausible}.
\begin{conj}\label{conjecture}
Let $M$ be a closed oriented negatively curved Riemannian $n$-manifold with fundamental group $\Gamma$. 
Then, there exist an orthogonal representation $\rho: \Gamma\to \End(H)$ weakly equivalent to the regular representation $\lambda_\Gamma$, and a $\Gamma$-equivariant, indecomposable, $n$-dimensional minimal surface $\Sigma$ in  the unit sphere of $H$. 
\end{conj}

\subsection{Motivations and related works}

\textbf{The spherical volume}

Our first motivation is to better understand the spherical volume of closed oriented $n$-manifolds $M$. While its definition is geometric, Brunnbauer proved that it is a homotopy invariant \cite[Theorem 1.1]{Brunnbauer08}. Its value has been computed for all rank one locally symmetric spaces by Besson-Courtois-Gallot \cite[Equation (1.2)]{BCG95}, and for all closed oriented surfaces \cite[Exemples 3.3 (d)]{BCG91} and $3$-manifolds \cite[Theorem 5.1]{Antoine23a} (in low dimensions, it coincides with the hyperbolic volume up to a scaling factor). The spherical volume of a closed oriented manifold vanishes if the fundamental group is amenable \cite[Subsection 3.3.2]{Antoine23a}.  A vague heuristic based on low-dimensional manifolds suggests that ``most'' closed oriented manifolds have positive spherical volume. Besides, there is a rich interplay between the spherical volume and other invariants in topology, dynamical systems and Riemmanian geometry \cite[Equation (1)]{Kotschick11}:
\begin{equation}\label{ineq chain}
\frac{1}{2^n n!}\|M\| \leq \spherevol(M) \leq \frac{1}{2^n n^{n/2}} \minent(M)^n 
\leq \frac{(n-1)^n}{2^n n^{n/2}} \minvol_{\Ric}(M)
\end{equation}
where  $\|M\|$ is Gromov's simplicial volume \cite{Gromov82}, $\minent(M)$ is the infimum of the normalized volume entropy, and   $\minvol_{\Ric}(M)$ is the infimum of the volume under the Ricci curvature bound $\Ric\geq -(n-1)$.  
The spherical volume of $M$ might always be naturally realized by an Ambrosio-Kirchheim $n$-dimensional minimal surface inside a spherical quotient \cite[Question 3.4]{Antoine23a}.
This possible realization by a nice geometric object is a property that distinguishes the spherical volume. 
The hypothetical minimal surface should be the solution of a variational problem associated with the spherical volume,  called \emph{spherical Plateau problem} \cite[Section 3]{Antoine23a}. One version of it can be described as follows. 
Given a closed oriented $n$-manifold $M$, $D_M\subset \tilde{M}$, $\Gamma:=\pi_1(M)$, $S^\infty \subset \ell^2(\Gamma)$, $\lambda_\Gamma$ and $\mathscr{H}_M$ as above, consider a sequence of  maps  $\phi_i\in \mathscr{H}_M$, which is minimizing in the sense that 
$$\lim_{i\to \infty} \Vol(D_M,\phi_i^*\mathbf{g}_{\mathrm{Hil}}) = \spherevol(M).$$
By a compactness theorem of Wenger \cite{Wenger11}\cite[Theorem 4.19]{SW11}, subsequentially, the sequence of compact quotients $\phi_i(M)/\Gamma \subset S^\infty/\Gamma$ converges to an ``\emph{integral current space}'' $C_\infty$ called \emph{spherical Plateau solution}, which lives in a quotient of a Hilbert unit sphere. In fact, the minimal surface of Theorem \ref{theorem:plateau solution} will be a spherical Plateau solution, and Theorem \ref{theorem:plateau solution} answers Questions 3.3 and 3.4 in \cite{Antoine23a} for negatively curved manifolds.
(There is also an equivariant version of the spherical Plateau problem in the unit sphere $S^\infty$, instead of the quotient $S^\infty/\Gamma$.)

This variational problem has attractive properties \cite{BCG95}\cite{Kotschick11}\cite{Antoine23a}\cite{Antoine23b}\cite{CN23}. For example, when $(M,g_0)$ is negatively curved locally symmetric of dimension at least $3$, then $C_\infty$ recovers $(M,g_0)$ up to a dimensional factor. When $M$ is a closed oriented $3$-manifold, $C_\infty$ recovers the hyperbolic part $(M_{\mathrm{hyp}},g_0)$ of $M$ given by the Geometrization theorem. 
A similar classification of the possible limits $C_\infty$ arising from closed surfaces seems possible.

\vspace{1em}

\textbf{Geometrization problem for manifolds and orthogonal representations}

The classical geometrization problem for manifolds is the following question of Hopf-Thom-Yau  \cite[Section 11.1]{Berger03} \cite[Problem 1]{Yau93}: given a closed manifold $M$, does it admit a ``best'' metric? 
Taken too literally, this question does not have a satisfying answer \cite{Nabutovsky95}\cite{NW00}\cite[Chapter III, Remarks and references (e)]{Gromov00}. 
Under a more relaxed interpretation of the problem,  
Conjecture \ref{conjecture} attempts to ``geometrize'' negatively curved manifolds using the following kind of geometric structure: 
$
(X,\Sigma,\Gamma)
$
where $X$ is a (possibly infinite-dimensional) homogeneous space, $\Gamma$ is a discrete subgroup of the isometry group of $X$, and $\Sigma$ is a $\Gamma$-invariant $n$-dimensional minimal surface in $X$.
Theorem \ref{theorem:plateau solution} and its proof suggest that if $M$ is negatively curved, the minimal surface in Conjecture \ref{conjecture} could be obtained as a solution of the equivariant version of the spherical Plateau problem.

An extension of the spherical Plateau problem also provides a framework to investigate the geometry of general orthogonal (or unitary) representations, not just the regular representation. 
In the context of  representations, the geometrization problem can be recast as the problem of
understanding triples $(\Gamma,\rho,\Sigma)$ where $\Gamma$ is a group, $(\rho,V)$ is a unitary representation of $\Gamma$, and $\Sigma$ is a $\Gamma$-equivariant $m$-dimensional minimal surface in the unit sphere $S_V$ of $V$. 
This perspective  leads to many questions. 
Theorem \ref{theorem:plateau solution} and Conjecture \ref{conjecture} focus on the regular representation of negatively curved fundamental groups.
Spherical volumes can be interpreted as a kind of higher dimensional version of Kazhdan constants for unitary representations.
In \cite{Antoine24a}, we give a geometric counterpart of a strong asymptotic freeness theorem due Bordenave-Collins \cite[Theorem 3]{BC19}: equivariant $2$-dimensional minimal surfaces in spheres, corresponding to certain random representations of the free group $F_2$, are almost hyperbolic. 

\vspace{1em}

\textbf{Minimal surfaces in infinite-dimensional spaces}

Another motivation was to find new ways of constructing $n$-dimensional minimal surfaces (without boundary) in infinite-dimensional Riemannian spaces.  There are by now many existence theories for minimal surfaces in finite-dimensional manifolds, including explicit constructions, area minimization in homotopy or homology classes, and min-max theory.  
Recently, in the framework of metric currents of Ambrosio-Kirchheim, numerous papers have examined the metric Plateau problem, namely the problem of constructing and using minimal surfaces with a given non-empty boundary inside metrics spaces, \cite[Section 10]{AK00}\cite{Wenger14}\cite{AS13}\cite{Wenger14} \cite{LW17}\cite{LW18}\cite{KL20} etc. 
Before that, following the work of Korevaar-Schoen \cite{KS93}\cite{KS97}, many results about harmonic maps from domains with or without boundary into nonpositively curved metric spaces were established.
In contrast, aside from some beautiful but isolated examples, there is so far no general existence theory for minimal surfaces without boundary in infinite-dimensional spaces. Theorem \ref{theorem:plateau solution} is a step in that direction.

\vspace{1em}


\textbf{Limit of cycles, limit of representations, limit of groups}

An underlying theme of this paper is the problem of taking limits of geometric structures.
Concretely,  we are interested in taking limits of minimizing sequences of cycles in spaces of the form $S^\infty/\Gamma$, and those limits live inside the ultralimit $(S^\infty/\Gamma)_\omega$. The space $(S^\infty/\Gamma)_\omega$ is of independent interest, because in some way, it encodes some properties of limits of the regular representation $\lambda_\Gamma$ in the Fell topology \cite[Section F.2]{BDLHV08} and limit groups of $\Gamma$ in the sense of Sela \cite{Sela01}, see Subsection \ref{prelim3}. 
The study of $(S^\infty/\Gamma)_\omega$ involves a description of an ``ultralimit representation'' $\lambda_\omega$ of the ultraproduct $\Gamma_\omega$ of $\Gamma$, see Subsections \ref{subsection:nota} and \ref{subsection:lambdaomega}. 
This has applications to the geometry of $S^\infty/\Gamma$, which is an ``infinite-dimensional locally symmetric space'' of the simplest kind. 
When $\Gamma$ is  non-elementary torsion-free hyperbolic, we will see that the singularities of $(S^\infty/\Gamma)_\omega$ are in fact tractable. With this observation, we show that a Margulis type lemma, akin to the classical result for hyperbolic manifolds, holds for the $\delta$-thin part of $S^\infty/\Gamma$ when $\delta$ is small.
We will also construct special $1$-Lipschitz deformation maps from separable subsets of $(S^\infty/\Gamma)_\omega$ to $(S^\infty/\Gamma)_\omega$, which are essential for our purpose. 

\subsection{Strategy of proof}

\textbf{Preliminaries.} 
Consider a non-elementary hyperbolic group $\Gamma$. 
Given a nonzero homology class $h\in H_n(S^\infty/\Gamma;\mathbb{Z})$, 
let $\mathscr{C}(h)$ be the set of compactly supported $n$-cycles in $S^\infty/\Gamma$ representing $h$, in the sense of Ambrosio-Kirchheim \cite{AK00}. Set
$$\spherevol(h):= \inf\{\mathbf{M}(C); \quad C\in \mathscr{C}(h)\}.$$
It can be shown \cite[Section 3.5]{Antoine23a} that the spherical volume of a closed, oriented, negatively curved manifold $M$ with fundamental group $\Gamma$ (as defined in (\ref{first def sv})) is equal to $\spherevol(h)$ for a well-defined homology class $h$. 

Let $C_i\in \mathscr{C}(h)$ be a minimizing sequence, i.e. such that 
$$\lim_{i\to \infty} \mathbf{M}(C_i) = \spherevol(h).$$
A simple observation is that by Wenger's compactness theorem \cite{Wenger11}, $C_i$ subsequentially converges in the intrinsic flat topology \cite{SW11} to an integral current space $C_\infty$, called a spherical Plateau solution, which is isomorphic to an $n$-cycle $S_\infty$ inside the ultralimit $(S^\infty/\Gamma)_\omega$. Note that $S_\infty$ is not yet known to be mass-minimizing. In fact, showing this is one of the main challenges.
By lower semicontinuity of the mass under intrinsic flat convergence, we at least know that 
\begin{equation} \label{lowwer}
\mathbf{M}(S_\infty)\leq \spherevol(M).
\end{equation} 


The ultralimit  $(S^\infty/\Gamma)_\omega$ is itself a quotient of a Hilbert unit sphere, but it is not separable and not smooth.
Let $\Reg(S^\infty/\Gamma)_\omega $ and $\Sing(S^\infty/\Gamma)_\omega $ denote respectively its smooth part and its singular part.
Roughly speaking, the points in $\Reg(S^\infty/\Gamma)_\omega $ correspond to sequences of points in $S^\infty/\Gamma$ with injectivity radii uniformly bounded away  from $0$.

\textbf{Main difficulties and the deformation property.}
 After those standard preliminaries, there are three a priori independent major difficulties, due to the infinite-dimensionality of the problem and the non-smoothness of $(S^\infty/\Gamma)_\omega $: 
\begin{enumerate} [label=(\alph*)]
\item to show that equality actually holds in (\ref{lowwer}),
\item to show that $S_\infty$ is mass-minimizing in $(S^\infty/\Gamma)_\omega$, even across $\Sing (S^\infty/\Gamma)_\omega$,
\item to show that $S_\infty$ is almost entirely contained in $\Reg(S^\infty/\Gamma)_\omega $.
\end{enumerate} 
The key to prove all those points is a \emph{deformation property}:
any integral current  in $(S^\infty/\Gamma)_\omega$ can be continuously deformed to an integral current almost entirely contained in $\Reg(S^\infty/\Gamma)_\omega $ via $1$-Lipschitz maps which send $\Sing(S^\infty/\Gamma)_\omega $ to itself.

Once (a), (b), (c) are established, the main theorem is mostly proved.
 Here ``almost entirely contained'' means that the restriction of the cycle to  $\Sing(S^\infty/\Gamma)_\omega $ has zero mass.
Interestingly, we do not know how to show either one of the properties (a) (b) (c) without essentially proving the other properties at the same time. Our construction of the $1$-Lipschitz deformation maps relies strongly on the fact that $\Gamma$ is hyperbolic, as discussed below.

\textbf{The $1$-Lipschitz deformation maps.}
Let us describe the \emph{deformation maps}, which are the key new tool of our proof, see Subsection \ref{subsection:projection}. The construction is general in that it does not involve metric currents. For any separable subset $\mathbf{A}\subset (S^\infty/\Gamma)_\omega$, we can construct a continuous family of $1$-Lipschitz maps $\{\mathscr{D}_\theta\}_{\theta\in [0,1]}$ depending on $\mathbf{A}$,
$$\mathscr{D}_\theta : \mathbf{A} \to (S^\infty/\Gamma)_\omega$$
such that 
$\mathscr{D}_0=\Id$, $\mathscr{D}_1(\mathbf{A})\cap \Sing(S^\infty/\Gamma)_\omega$  is contained in a  discrete subset $\underline{\underline{\mathcal{D}}}$, and if $q\in \mathbf{A}$ is $d$-close to $\Sing (S^\infty/\Gamma)_\omega$, then $\mathscr{D}_1(q)$ is  $d$-close to $\underline{\underline{\mathcal{D}}}$ for any $d>0$.
 
In fact, $\mathscr{D}_1$ is the composition of two $1$-Lipschitz maps: $\mathscr{D}_1 = \mathscr{J}\circ \overline{\mathscr{D}}_1$ with
\begin{equation}\label{Q def}
\overline{\mathscr{D}}_1: \mathbf{A} \to \mathcal{Q}\quad \text{and} \quad \mathscr{J}: \mathcal{Q}\to (S^\infty/\Gamma)_\omega,
\end{equation}
where $\mathcal{Q}$ is an intermediate Hilbert spherical quotient such that $\mathcal{Q}\setminus \underline{\mathcal{D}}$ is smooth for some discrete subset $\underline{\mathcal{D}}\subset \mathcal{Q}$.
The map $\overline{\mathscr{D}}_1$ sends $\mathbf{A}\cap\Sing(S^\infty/\Gamma)_\omega$ to $\underline{\mathcal{D}}$, and $\mathscr{J}$ sends $\mathcal{Q}\setminus \underline{\mathcal{D}}$ (resp. $\underline{\mathcal{D}}$) to $\Reg(S^\infty/\Gamma)_\omega $ (resp. the discrete subset $\underline{\underline{\mathcal{D}}}$).


\textbf{Outline of proof for (a), (b), (c).}
On a heuristic level, the reason why the deformation property is useful can be explained as follows. The support of any integral current is separable. The $n$-cycle $S_\infty$ can thus be deformed via $1$-Lipschitz maps $\{\mathscr{D}_\theta\}_{\theta\in [0,1]}$ corresponding to $\mathbf{A}:=\spt S_\infty$, to an $n$-cycle $(\mathscr{D}_1)_\sharp S_{\infty}$ almost entirely contained in $\Reg(S^\infty/\Gamma)_\omega $. Take a thin neighborhood $N$ of $\Sing(S^\infty/\Gamma)_\omega $, let $N^c$ be its complement and consider the restriction $S_{\infty} \llcorner N^c$ of $S_{\infty} $ to $N^c$. 
A technical lemma ensures that $(\mathscr{D}_\theta)_\sharp (S_{\infty} \llcorner N^c)$ can be assumed to be contained in $\Reg(S^\infty/\Gamma)_\omega $ for any $\theta\in [0,1]$.
Since points of $\Reg(S^\infty/\Gamma)_\omega $ are limits of points in $S^\infty/\Gamma$ with bounded geometry, it is possible to approximate the pushforward current $(\mathscr{D}_1)_\sharp (S_{\infty} \llcorner N^c)$ by some compactly supported integral current $C'$ with boundary, inside $S^\infty/\Gamma$, with almost the same mass as $(\mathscr{D}_1)_\sharp (S_{\infty} \llcorner N^c)$. 
Crucially, since the deformation map 
$\mathscr{D}_1$  sends points of $\spt S^\infty\cap N$ to points close to a discrete subset $\{z_1,z_2,...\}$, 
it sends the boundary of $(\mathscr{D}_1)_\sharp S_{\infty} \llcorner N^c$ to small neighborhoods of $z_1,z_2,...$
We can then squeeze the boundary of $(\mathscr{D}_1)_\sharp S_{\infty} \llcorner N^c$ towards those points $z_j$ using conical contractions, and arbitrarily reduce its mass, without increasing too much the mass of $(\mathscr{D}_1)_\sharp S_{\infty} \llcorner N^c$ (Lemma \ref{conical retraction}).
We can thus assume without loss of generality that the boundary of $(\mathscr{D}_1)_\sharp S_{\infty} \llcorner N^c$ is a cycle in  $\Reg(S^\infty/\Gamma)_\omega $ with small mass. 
Hence the boundary $\partial C'$ also can be chosen to have small mass, and can be capped-off in such a way that we obtain an $n$-cycle $C''$ in $S^\infty/\Gamma$ without changing too much the mass (Lemma \ref{controlled filling}).
That controlled filling lemma relies on some knowledge about the topology of the thin part of $S^\infty/\Gamma$, a Margulis type lemma.
With some additional work, we can show that $C'' \in \mathscr{C}(h)$ because $\{\mathscr{D}_\theta\}_{\theta\in [0,1]}$ is a continuous family of maps. From this approximation argument, we get 
$$\mathbf{M}((\mathscr{D}_1)_\sharp S_{\infty}) \geq\spherevol(h).$$
But since $\mathscr{D}_1$ is $1$-Lipschitz, (\ref{lowwer}) implies (a), namely
$$\mathbf{M}(S_\infty) = \mathbf{M}((\mathscr{D}_1)_\sharp S_{\infty}) = \spherevol(h).$$
Since $(\mathscr{D}_1)_\sharp S_{\infty}$ is almost entirely contained in $\Reg(S^\infty/\Gamma)_\omega $ and the $1$-Lipschitz map   $\mathscr{D}_1$ sends singular points to singular points, this also shows (c), namely that 
$$\mathbf{M}(S_\infty) = \mathbf{M}\big(S_\infty \llcorner \Reg(S^\infty/\Gamma)_\omega \big).$$
Next, if $S_\infty=U+\partial V$, then applying the deformation map $\mathscr{D}_1$ corresponding to $\mathbf{A}:=\spt S_\infty \cup \spt V$, we get new currents
$$(\mathscr{D}_1)_\sharp S_{\infty} = (\mathscr{D}_1)_\sharp U + (\mathscr{D}_1)_\sharp \partial V.$$
By similar arguments, we can show that $(\mathscr{D}_1)_\sharp U$ is approximated by $n$-cycles in $S^\infty/\Gamma$, belonging to $\mathscr{C}(h)$, and with almost the same mass as $(\mathscr{D}_1)_\sharp U$. Thus 
$$\mathbf{M}((\mathscr{D}_1)_\sharp U) \geq \spherevol(h),$$
which is enough to conclude (b) since $\mathbf{M}((\mathscr{D}_1)_\sharp U)  \leq \mathbf{M}(U) $.
This is roughly how the deformation maps are used, although there is an added technical difficulty due to the fact that integral currents do not in general have compact supports. 
The arguments above are described in Section \ref{section:final}.


\textbf{Geometry of $(S^\infty/\Gamma)_\omega$ and construction of the deformation maps.}
The construction of the deformation maps relies on a structural result for $(S^\infty/\Gamma)_\omega$, in particular for its singular set. The ultralimit $(S^\infty/\Gamma)_\omega$ is equal to the quotient $S^\infty_\omega/\lambda_\omega(\Gamma_\omega)$ where $S^\infty_\omega$ is a non-separable Hilbert unit sphere of a real Hilbert space $H_\omega$, $\Gamma_\omega$ is the ultraproduct of $\Gamma$ with respect to $\omega$, and $\lambda_\omega$ is the ``ultralimit'' of the regular representation:
 this is an orthogonal representation of $\Gamma_\omega$ which contains all possible ``limits'' of the regular regular representation of the hyperbolic group $\Gamma$. Understanding $(S^\infty/\Gamma)_\omega$ essentially amounts to understanding the representation $\lambda_\omega$, see Subsection \ref{subsection:lambdaomega}. 

Let $\mathbf{A}\subset (S^\infty/\Gamma)_\omega$ be a separable subset. It is not hard to see that $\mathbf{A}$ can be assumed to be equal to a spherical quotient $$S^\infty_{(1)}/\lambda_\omega(\Gamma_{(1)})$$
where $S^\infty_{(1)} \subset S^\infty_\omega$ is the unit sphere of a separable subspace $H_{(1)}\subset H_\omega$, $\Gamma_{(1)}\subset \Gamma_\omega$ is a countable subgroup, and the restriction $\lambda_\omega|_{\Gamma_{(1)}}$ of  $\lambda_\omega$ to $\Gamma_{(1)}$ leaves $H_{(1)} $ invariant. 
Relying  on the result of Delzant for free subgroups of hyperbolic groups \cite[Th\'{e}or\`{e}me I]{Delzant96}, we will show that the orthogonal representation $(H_{(1)},\lambda_\omega|_{\Gamma_{(1)}})$ decomposes as follows:
$$H_{(1)} = H_{\mathrm{pf}} \oplus \bigoplus_j  H_j$$
$$\lambda_\omega|_{\Gamma_{(1)}} = \eta_{\mathrm{pf}} \oplus \bigoplus_j \lambda_j,$$
where $\eta_{\mathrm{pf}}(\Gamma_{(1)})$ acts freely properly on the unit sphere of $H_{\mathrm{pf}}$, and each $(H_j, \lambda_j)$ is the induced representation from an orthogonal representation of an abelian subgroup $G_j$ of $\Gamma_{(1)}$, see \cite[Section 1.F]{BDLH19} or Section \ref{rep theory} for the definition of induced representations. 

We will then find a subrepresentation $(\bigoplus_j \hat{H}_j,\bigoplus_j \hat{\lambda}_j)$ of $(H_\omega,\lambda_\omega|_{\Gamma_{(1)}})$, such that 
$$\bigoplus_j \hat{H}_j\quad \text{is orthogonal to $H_{(1)}$}$$
and each $(\hat{H}_j, \hat{\lambda}_j)$ is the induced representation from the trivial one-dimensional representation of the abelian subgroup $G_j$ of $\Gamma_{(1)}$. This is where we need $\mathbf{A}$ to be separable.

Next, there is a natural $1$-Lipschitz equivariant map from  $H_{(1)}=H_{\mathrm{pf}} \oplus \bigoplus_j  H_j$ to $H_{(2)}:= H_{\mathrm{pf}} \oplus  \bigoplus_j \hat{H}_j$, which is the identity on $H_{\mathrm{pf}}$, and is the ``obvious'' map from $H_j$ to $\hat{H}_j$ for each $j$.
This map descends to a $1$-Lipschitz map between the quotients of the respective unit spheres
$$\overline{\mathscr{D}}_1: S^\infty_{(1)}/\lambda_\omega(\Gamma_{(1)}) \to S^\infty_{(2)}/\lambda_\omega(\Gamma_{(1)}).$$
The space $S^\infty_{(2)}/\lambda_\omega(\Gamma_{(1)})$ is the intermediate spherical quotient $\mathcal{Q}$ mentioned in (\ref{Q def}). It is smooth outside a discrete subset $\underline{\mathcal{D}}$.
There is also a natural quotient map 
$$\mathscr{J}: S^\infty_{(2)}/\lambda_\omega(\Gamma_{(1)}) \to (S^\infty/\Gamma)_\omega.$$
We finally set
$$\mathscr{D}_1 = \mathscr{J}\circ \overline{\mathscr{D}}_1.$$
We can interpolate from $\Id$ to $\mathscr{D}_1 $ via $1$-Lipschitz maps $\{\mathscr{D}_\theta\}_{\theta\in [0,1]}$ using a similar construction. 
With a good choice $S^\infty_{(1)}$ and $\Gamma_{(1)}$, all the desired properties for $\mathscr{D}_\theta$ described earlier hold.
This concludes the sketch of the definition of the deformation maps.

\textbf{A Margulis type lemma.}
Incidentally, a corollary of the structural result for $(S^\infty/\Gamma)_\omega$ mentioned above is a Margulis type lemma for  $S^\infty/\Gamma$, which says that for some $\bar{\delta}>0$ depending on  $\Gamma$, each component of the $\bar{\delta}$-thin part (the part where the injectivity radius is at most $\bar{\delta}$) has cyclic fundamental group inside $S^\infty/\Gamma$,  see Corollary \ref{margulis}. This implies that any $n$-cycle in $S^\infty/\Gamma$ representing a nontrivial homology class $h$ of dimension $n\geq 2$, necessarily intersects the  $\bar{\delta}$-thick part of $S^\infty/\Gamma$. That in turn provides a new  and direct proof that $\spherevol(h)>0$, see Theorem \ref{positive spherevol}. This non-vanishing theorem ensures that our mass-minimizing cycle is non-empty.


\subsection{Organization}
\textbf{Section \ref{section:plateau problem}:}
We introduce the spherical Plateau problem, and study it for general groups and group homology classes. We show that a minimizing sequence of cycles can be replaced by a pulled-tight sequence with better convergence properties. We find that, in general,  spherical Plateau solutions are divided into a non-collapsed part, which is mass-minimizing, and a collapsed part. This section mostly only relies on the theory of metric currents.

\textbf{Section \ref{section:noncollapsing}:} 
We prove general properties for the structure of ``limits'' of the regular representation of a torsion-free hyperbolic group $\Gamma$. This is applied to show a Margulis lemma for $S^\infty/\Gamma$, and to define $1$-Lipschitz deformation maps on $(S^\infty/\Gamma)_\omega$ which are crucial for us. This section mostly only relies on geometric group theory and representation theory.

\textbf{Section \ref{section:final}:} Combining the previous sections, we employ the deformation maps to show that any spherical Plateau solution for a homology class $h$ of a  torsion-free hyperbolic group $\Gamma$ realizes the spherical volume of $h$,  is globally mass-minimizing in the ultralimit $(S^\infty/\Gamma)_\omega$, and is almost entirely supported in the smooth part $\Reg(S^\infty/\Gamma)_\omega $. 

\textbf{Appendix:} We briefly review some basic properties of metric currents, hyperbolic groups, regular representations, ultralimits and representation theory.

\subsection*{Acknowledgements}
I am grateful to G\'{e}rard Besson and Gilles Courtois for insightful conversations about the spherical volume, Bachir Bekka for explanations about unitary representations, Daniel Groves for exchanges about limit groups, John Lott for discussions about the simplicial volume, and Luca Spolaor for clarifications about the theory of currents. I would also like to thank Camillo De Lellis, Song Sun, Stefan Wenger, Alain Valette, Yves de Cornulier, Jingyin Huang, David Fisher, Jason Manning, Ian Agol, Richard Bamler, Guido De Philippis, St\'{e}phane Sabourau, Kevin Schreve, Alexander Nabutovsky, Tamunonye Cheetham-West, Alexander Nolte, Shi Wang, Xin Zhou, Ben Lowe for stimulating discussions during the writing of this article.

A.S. was partially supported by NSF grant DMS-2104254. This research was conducted during the period A.S. served as a Clay Research Fellow.

\section{The spherical Plateau problem} \label{section:plateau problem}
\label{definition of spherical plateau problem}

\subsection{Setup} \label{subsection:setup}

All Hilbert spaces will be real unless otherwise noted.  
Let $\Gamma$ be an infinite countable group.
Let $S^\infty$ be the unit sphere of $\ell^2(\Gamma)$, the space of real-valued $\ell^2$ functions on $\Gamma$. 
The $\ell^2$-norm induces the standard Hilbert-Riemannian metric $\mathbf{g}_{\mathrm{Hil}}$ on the unit sphere $S^\infty$. 
The group $\Gamma$ acts isometrically on $S^\infty$ by the left regular representation, see (\ref{reg rep}).
Denote by $S^\infty/\Gamma$
the corresponding Hilbert spherical quotient, implicitly endowed with the standard quotient metric.
The action of $\Gamma$ on $S^\infty$ is proper, and it is non-free exactly when $\Gamma$ has torsion elements. When $\Gamma$ is torsion-free and nontrivial, then it is not hard to check that $\Gamma$ acts properly freely on $S^\infty$ by the regular representation and the quotient space $S^\infty/\Gamma$ is in fact a classifying space for $\Gamma$, namely a $K(\Gamma,1)$ space (recall that the infinite-dimensional sphere $ S^\infty$ is contractible), see \cite[Section 3.2]{Antoine23a}.

The basic notions of the theory of metric currents, like integral currents \cite{AK00} or intrinsic flat convergence \cite{SW11} are reviewed in \cite[Section 1]{Antoine23a}.
Let $n\geq 0$ be an integer. Consider the following homology groups defined using integral currents:

\begin{align*}
\mathcal{Z}_n(S^\infty/\Gamma)  := \{T; \quad \text{$T$ is an integral $n$-current with} \\
 \text{compact support in $S^\infty/\Gamma$}\}, 
 \end{align*}
\begin{align*}
\mathcal{B}_n(S^\infty/\Gamma)  := \{\partial D; \quad \text{$D$ is an integral $(n+1)$-current with}\\
  \text{compact support in $S^\infty/\Gamma$}\},
 \end{align*}
 $$
\mathbf{H}_n(S^\infty/\Gamma)  := \mathcal{Z}_n(S^\infty/\Gamma)/\mathcal{B}_n(S^\infty/\Gamma).
$$
There is a natural morphism 
$$
\pi: H_*(\Gamma;\mathbb{Z}) \to \mathbf{H}_*(S^\infty/\Gamma)
$$
where $H_*(\Gamma;\mathbb{Z})$ are the singular homology groups of the group $\Gamma$ with coefficients in $\mathbb{Z}$ (\emph{group homology} for short). Recall that one of the equivalent definitions of the group homology $H_*(\Gamma;\mathbb{Z})$ is that it is equal to the homology $H_*(\mathcal{K};\mathbb{Z})$ of any classifying space $\mathcal{K}$ for $\Gamma$.
Given a group homology class $h\in H_n(\Gamma;\mathbb{Z})$, consider the space 
$$\mathscr{C}(h)$$ of boundaryless $n$-dimensional integral currents with compact supports inside $S^\infty/\Gamma$ which represent the homology class $\pi(h)\in \mathbf{H}_n(S^\infty/\Gamma)$, see \cite[Section 3.2]{Antoine23a} for a precise definition.


\begin{definition} [Spherical volume] \label{seconde def}
Let $h\in H_n(\Gamma;\mathbb{Z})$. The \emph{spherical volume} of $h$ is defined as
$$\spherevol(h) = \inf\{\mathbf{M}(C); \quad C\in \mathscr{C}(h)\}.$$
\end{definition} 
In case $\Gamma$ is the fundamental group of a closed, negatively curved, $n$-manifold, and $h_M$ is the fundamental class of $M$, it is known that $\spherevol(h_M)$ is equal to the spherical volume of $M$ defined in (\ref{first def sv}) \cite[Section 3.5]{Antoine23a}.
\begin{definition}[Spherical Plateau solution] 
We call spherical Plateau solution for $h$ any $n$-dimensional integral current space $C_\infty$ which is the limit in the intrinsic flat topology of a sequence $\{C_i\}\subset \mathscr{C}(h)$ such that
$$\lim_{i \to \infty} \mathbf{M}(C_i) = \spherevol(h).$$
\end{definition}
Spherical Plateau solutions $C_\infty$ for a group homology class $h$ always exist by the compactness theorem of Wenger \cite{Wenger11}\cite[Theorem 4.19]{SW11}. We will use the notation:
$$C_\infty = (X_\infty,d_\infty,S_\infty)$$
where $(\overline{X}_\infty,d_\infty)$ is a complete metric space containing an integral current $S_\infty$, and $X_\infty$ is the canonical set of $S_\infty$. Often, we will identify $C_\infty$ with $S_\infty$. By definition, $\mathbf{M}(C_\infty) := \mathbf{M}(S_\infty)$, and by lower semicontinuity of mass, $\mathbf{M}(C_\infty) \leq \spherevol(h)$.
Note that, while we focus on the regular representation only, all those definitions could be extended to general orthogonal or unitary representations.

\subsection{Spherical Plateau solutions and ultralimits} \label{subsection:nota}

Let us talk about $S^\infty/\Gamma$, its ultralimit and how it relates to spherical Plateau solutions.
Let $\Gamma$ be an infinite countable group, acting on $S^\infty\subset \ell^2(\Gamma)$ by the regular representation, and let $1$ denote its unit element. 
Let $\dist$ denote the standard length metric in the Hilbert unit sphere $S^\infty$. 
Let 
$$\Pi:S^\infty\to S^\infty/\Gamma$$ be the natural projection. For a threshold $\delta \in[0,1)$, define the \emph{$\delta$-thick part} and the \emph{$\delta$-thin part} of  $S^\infty/\Gamma$ as follows:
$$[S^\infty/\Gamma]^{>\delta} := \Pi(\{x \in S^\infty ; \quad \dist(x,\gamma.x) > \delta \text{ for all $\gamma \in \Gamma \setminus\{1\}$}\}),$$
$$[S^\infty/\Gamma]^{\leq\delta} := \Pi(\{x \in S^\infty; \quad \dist({x},\gamma.{x}) \leq  \delta \text{ for some $\gamma \in \Gamma\setminus\{1\}$}\}).$$
The points in $[S^\infty/\Gamma]^{>\delta}$ have injectivity radius at least $\delta$.
The following open subset
\begin{equation}\label{sstar}
[S^\infty/\Gamma]^{>0}= \bigcup_{\delta>0} [S^\infty/\Gamma]^{>\delta}.
\end{equation} 
is smooth and locally isometric to $(S^\infty,\mathbf{g}_{\mathrm{Hil}})$. 
We often denote by $\mathbf{g}_{\mathrm{Hil}}$ the standard round Hilbert-Riemannian metric on $[S^\infty/\Gamma]^{>0}$.
A useful fact to keep in mind is that given $\delta>0$ and $\epsilon\in (0,\frac{\delta}{3})$, the $\epsilon$-neighborhood of $ [S^\infty/\Gamma]^{>\delta}$ is contained in $ [S^\infty/\Gamma]^{>\delta-2\epsilon}$.
If $C$ is an integral current in $S^\infty/\Gamma$, then set
\begin{equation}\label{c>delta}
C^{>\delta}:=C \llcorner [S^\infty/\Gamma]^{>\delta}.
\end{equation}

As explained in Subsection \ref{prelim3} of the Appendix, given a non-principal ultrafilter $\omega$,  the ultralimit of the constant sequence $S^\infty$  with respect to $\omega$ is the unit sphere 
$S^\infty_\omega$ of a non-separable Hilbert space $H_\omega$.
The group $\Gamma$ and its action on $S^\infty$ give rise to a discrete uncountable group $\Gamma_\omega$, the ultraproduct of $\Gamma$ with respect to $\omega$, which acts (non-freely) by isometries on $S^\infty_\omega$. The action of $\Gamma_\omega$ on $S^\infty_\omega$ defines an orthogonal representation, which we will sometimes call  the ``ultralimit representation'': 
\begin{equation}\label{ultralimit representation}
\mathbf{\lambda}_\omega :\Gamma_\omega \to O(H_\omega).
\end{equation}
One checks directly that $S^\infty_\omega/\lambda_\omega(\Gamma_\omega)$, endowed with the quotient metric, is equal to the ultralimit $(S^\infty/\Gamma)_\omega$ of the constant sequence $S^\infty/\Gamma$ (or equivalently of the constant sequence $[S^\infty/\Gamma]^{>0}$):
 $$S^\infty_\omega/\lambda_\omega(\Gamma_\omega) = (S^\infty/\Gamma)_\omega.$$
In Section \ref{section:noncollapsing}, we will study in detail the ultralimit representation $\lambda_\omega$, which will be useful to describe $(S^\infty/\Gamma)_\omega$ and its singularities.
Let $\mathbf{\dist}_\omega$ denote the metric on the ultralimit $(S^\infty/\Gamma)_\omega$.

Let $F_\omega$ be the subset of points in $S^\infty_\omega$ fixed by at least one element of $\Gamma_\omega$ which is not the identity element $1\in \Gamma_\omega$, and let $\Omega_\omega$ be the complementary set:
\begin{equation} \label{Fomega}
\begin{split}
F_\omega & :=\{x\in S^\infty_\omega;\quad \mathbf{\lambda}_\omega(\mathbf{\gamma}).x = x \quad  \text{for some $\mathbf{\gamma}\in \Gamma_\omega \setminus \{1\}$}\}, \\
\Omega_\omega & := \{x\in S^\infty_\omega;\quad \mathbf{\lambda}_\omega(\mathbf{\gamma}).x \neq x \quad \text{for any $\mathbf{\gamma}\in \Gamma_\omega \setminus \{1\}$}\}.
\end{split}
\end{equation}
By properties of ultralimits and (\ref{sstar}), $\Omega_\omega$ is an open subset of $S^\infty_\omega$, the action of $\Gamma_\omega$ on ${\Omega}_\omega$ is proper free
and $\Omega_\omega/\lambda_\omega(\Gamma_\omega)$ is an open subset which is exactly the smooth (or regular) part of $(S^\infty/\Gamma)_\omega$,
while the complement $F_\omega/\lambda_\omega(\Gamma_\omega)$ is the singular part of  $(S^\infty/\Gamma)_\omega$:
$$\Reg(S^\infty/\Gamma)_\omega := \Omega_\omega/\lambda_\omega(\Gamma_\omega),$$
$$\Sing(S^\infty/\Gamma)_\omega := F_\omega/\lambda_\omega(\Gamma_\omega).$$
The open subset $ \Reg(S^\infty/\Gamma)_\omega$ is locally isometric to the unit Hilbert sphere $S^\infty_\omega$. We often denote by $\mathbf{g}_{\mathrm{Hil}}$ the standard round Hilbert-Riemannian metric on $\Reg(S^\infty/\Gamma)_\omega$.

If $U^{>\delta}$ denotes the ultralimit of the constant sequence $[S^\infty/\Gamma]^{>\delta}$ where $\delta\in(0,1)$, then $U^{>\delta}$ is a nested sequence of complete metric spaces such that 
\begin{equation} \label{omegaU}
\Reg(S^\infty/\Gamma)_\omega = \bigcup_{\delta\in(0,1)} U^{>\delta}.
\end{equation}
 At points of $U^{>\delta}$, the injectivity radius  is bounded from below by $\delta$.

Given $d>0$, if $\dist_\omega$ denotes the standard quotient metric on $(S^\infty/\Gamma)_\omega$, set
$$N_\omega(d) :=  \{p\in (S^\infty/\Gamma)_\omega;\quad \mathbf{\dist}_\omega(p, \Sing(S^\infty/\Gamma)_\omega ) \leq d\}, $$
$$N^c_\omega(d) := (S^\infty/\Gamma)_\omega \setminus N_\omega(d).$$

The main reason for introducing ultralimits is that it naturally contains any spherical Plateau solution. The lemma below follows from \cite[Proposition 2.2]{Wenger11}:

\begin{lemme} \label{ultraembedding1}
Let $h\in H_n(\Gamma;\mathbb{Z})$ and  let $C_\infty=(X_\infty,d_\infty,S_\infty)$ be a spherical Plateau solution for $h$. There is an isometric embedding 
$$\iota_{u} :(\spt S_\infty,d_\infty) \hookrightarrow  (S^\infty/\Gamma)_\omega.$$
It induces an embedding of $S_\infty$ as an integral current without  boundary in $(S^\infty/\Gamma)_\omega$. 
\end{lemme}
 
The embeddings in the above lemma depend on the minimizing sequence $\{C_i\}$ converging to $C_\infty$ (and the choice of non-principal ultrafilter $\omega$). Nevertheless, we will from now on always  view $S_\infty$ as an integral current inside $(S^\infty/\Gamma)_\omega$.

\begin{definition} \label{definition:noncollapsed}
Let $h\in H_n(\Gamma;\mathbb{Z})$ and  let $C_\infty=(X_\infty,d_\infty,S_\infty)$ be a spherical Plateau solution for $h$. Set
$$Y_\infty := \spt S_\infty\cap \Reg(S^\infty/\Gamma)_\omega ,$$
$$T_\infty := S_\infty \llcorner \Reg(S^\infty/\Gamma)_\omega .$$
The triple $C_\infty^{>0}:=(Y_\infty,d_\infty,T_\infty)$ is called a noncollapsed part of $C_\infty$.
\end{definition}
Subsequently, for simplicity we will talk about ``the noncollapsed part'' $C_\infty^{>0}$ of a Plateau solution $C_\infty$, even though a noncollapsed part $C_\infty^{>0}$ depends not only on $C_\infty$ but also on the embedding in Lemma \ref{ultraembedding1}. 
(In general, the noncollapsed part $C_\infty^{>0}$ may not be ``completely settled'' in the sense of \cite[Definition 2.10]{SW11} so strictly speaking $C_\infty^{>0}$ may not be an integral current space. We do not know if the current $T_\infty$ can have nonzero boundary.)

\subsection{Approximation by currents with compact supports}

Metric currents have in general non-compact supports, but it is often easier to work with compactly supported currents. The goal of this technical subsection is to check some useful and rather general approximation results.

Let us first review the Ekeland variational principle \cite{Ekeland74}, which we will apply to pull-tight currents in a way similar to \cite[Theorem 10.6]{AK00}. 
Recall the statement of that variational principle: if $(X,d)$ is a complete metric space and if $f:X\to \mathbb{R}$ is a nonnegative continuous function, for any $x_0\in X$ and any $\varepsilon>0$, there exists some $y\in X$ such that 
$$f(y)\leq f(x_0) -\varepsilon d(x_0,y)$$
and for every $y'$ different from $y$,
$$f(y)< f(y')+ \varepsilon d(y,y').$$

\begin{lemme} \label{1st approx}
Consider some $n$-dimensional integral currents $S_0$, $U_0$ in $S^\infty/\Gamma$ and an $(n+1)$-dimensional integral current $V_0$ in  in $S^\infty/\Gamma$ such that
$$S_0=U_0+\partial V_0.$$
Suppose that $\spt S_0$ is compact and suppose that $\spt S_0$, $\spt U_0$, $\spt V_0$ are contained in $[S^\infty/\Gamma]^{>\delta_0}$ for some $\delta_0>0$. Then there are integrals current $U_1$, $V_1$ such that $\spt U_1$, $\spt V_1$ are compact, contained in $[S^\infty/\Gamma]^{>\delta_0/2}$, 
$$S_0 = U_1+ \partial V_1,$$
$$\mathbf{M}(U_1) \leq \mathbf{M}(U_0) -\frac{1}{10}\mathbf{M}(U_0-U_1),$$
and for any $\delta\in (\delta_0,1)$, for any $x\in [S^\infty/\Gamma]^{>\delta} \cap \spt U_1$, for all $\rho\in (0,\frac{\delta}{100})$,
\begin{equation}\label{monot ineq}
\mathbf{M}(U_1\llcorner B(x,\rho)) \geq c_{n} \rho^n
\end{equation}
where $c_n$ is a positive constant depending only on $n$.

\end{lemme}

\begin{proof}
To get around the issue that $S^\infty/\Gamma$ does not have bounded geometry, the trick is to rescale the metric conformally so that the geometry becomes bounded. 
More concretely, let $F:[S^\infty/\Gamma]^{>3\delta_0/5}\to [1,\infty)$ be a smooth function such that $F=1$ on $[S^\infty/\Gamma]^{>4\delta_0/5}$, and for all $x$ close enough to the boundary of $[S^\infty/\Gamma]^{>3\delta_0/5}$,
$$\frac{1}{\dist(x, [S^\infty/\Gamma]^{\leq 3\delta_0/5})^2} \leq F(x) \leq \frac{100}{\dist(x, [S^\infty/\Gamma]^{\leq 3\delta_0/5})^2}.$$ 
Consider the complete rescaled metric on $[S^\infty/\Gamma]^{>3\delta_0/5}$:
$$\mathbf{g}_F := F. \mathbf{g}_{\mathrm{Hil}}.$$
Integral currents for the rescaled  metric are endowed with the rescaled mass $\mathbf{M}_{\mathbf{g}_F}$. They  are integral currents for the original metric, and conversely  integral currents for the original metric are integral currents for the new one if they are supported in $[S^\infty/\Gamma]^{>\delta_0}$ (but not in general).
It is not hard to check that the new metric space $([S^\infty/\Gamma]^{>3\delta_0/5},\mathbf{g}_F)$ admits a local Euclidean type isoperimetric inequality in $S^\infty/\Gamma$ in the following sense: for any $\delta\in (3\delta_0/5,1)$, for any $x\in [S^\infty/\Gamma]^{>\delta}$ and any $\mathbf{g}_F$-radius $r\in (0,\delta/100)$, for any $k$-dimensional integral current $T$ with support in the $\mathbf{g}_F$-geodesic ball $B_{\mathbf{g}_F}(x,r)$, there is a $(k+1)$-dimensional integral current $R$  supported in $B_{\mathbf{g}_F}(x,r)$, such that $\partial R = T$ and $\mathbf{M}_{\mathbf{g}_F}(R)\leq \tilde{c}_k \mathbf{M}_{\mathbf{g}_F}(T)^{\frac{n}{n-1}}$ where $\tilde{c}_k$ only depends on the dimension $k$.
Additionally, $([S^\infty/\Gamma]^{>3\delta_0/5},\mathbf{g}_F)$ is $c'$-quasiconvex and  admits $c'$-cone type inequalities in the sense of \cite[]{Wenger07} for some $c'>0$.

Let $\mathscr{C}(S_0)$ denote the set of integral currents $A$ in $([S^\infty/\Gamma]^{>3\delta_0/5},\mathbf{g}_F)$ such that 
$$S_0 = A+ \partial R$$ for some $(n+1)$-dimensional integral current $R$ in $([S^\infty/\Gamma]^{>3\delta_0/5},\mathbf{g}_F)$.
By \cite[]{Wenger07} (this is where we need $c'$-quasiconvexity and $c'$-cone type inequalities), $\mathscr{C}(S_0)$  is a  complete metric space when endowed with the new mass topology induced by $\mathbf{M}_{\mathbf{g}_F}$.
We apply the Ekeland variational principle \cite{Ekeland74} to $\mathscr{C}(S_0)$ with $\varepsilon:=\frac{1}{10}$, 
so that we find $U_1\in \mathscr{C}(S_0)$ with support in $[S^\infty/\Gamma]^{>3\delta_0/5}$ such that
$$ \mathbf{M}_{\mathbf{g}_F}(U_1)\leq \mathbf{M}_{\mathbf{g}_F}(U_0) - \frac{1}{10} \mathbf{M}_{\mathbf{g}_F}(U_0-U_1),$$
and
$$T\mapsto \mathbf{M}_{\mathbf{g}_F}(T) +\frac{1}{10} \mathbf{M}_{\mathbf{g}_F}(T-U_1)$$
achieves a minimum at $U_1$. Note that the first inequality above implies
$$\mathbf{M}(U_1) \leq \mathbf{M}(U_0) -\frac{1}{10}\mathbf{M}(U_0-U_1).$$

As in the proof of \cite[Theorem 10.6]{AK00}, the local Euclidean type isoperimetric inequality implies that for any $\delta \in (0,1)$, for any $x\in [S^\infty/\Gamma]^{>\delta} \cap \spt U_1$, for all $\rho\in (0,\frac{\delta}{100})$,
$$\mathbf{M}_{\mathbf{g}_F}(U_1\llcorner B(x,\rho)) \geq c_{n} \rho^n$$
where $c_n$ is a positive constant depending only on the dimension.
In particular, $\spt U_1$ is compact. Of course, if $\delta\geq\delta_0$, $\mathbf{M}_{\mathbf{g}_F}$ can be replaced by $\mathbf{M}$.

Since $U_1\in  \mathscr{C}(S_0)$, there is an $(n+1)$-dimensional integral current $V_1$ in $[S^\infty/\Gamma]^{>3\delta_0/5}$  such that
$$S_0=U_1+\partial V_1.$$ 
Another application of the Ekeland variational principle as above ensures that $V_1$ can be chosen to be compactly supported  in $[S^\infty/\Gamma]^{>\delta_0/2}$.


\end{proof}

\begin{lemme} \label{apply coarea}
Let  $d>0$ and $\eta>0$. Consider a $k$-dimensional integral current $W$  possibly with boundary in $(S^\infty/\Gamma)_\omega$, with support in $N^c_\omega(d)$. Suppose that $\spt W$ is locally compact, then there is  an open subset $O_W\subset N^c_\omega(d/2)$ such that $W\llcorner O_W$ is an integral current with compact support, and
$$\mathbf{M}(W-W\llcorner O_W) \leq \eta,$$
 $$\mathbf{M}(\partial(W-W\llcorner O_W))\leq \eta.$$
 Besides, given a compact subset $A\subset \spt W$, the open subset $O_W$ can be chosen to contain $A$.

\end{lemme}

\begin{proof}
Since $\spt W$ is locally compact by assumption and always $\sigma$-compact, it admits an exhaustion by compact subsets. This in turn implies that we can 
find a sequence of open embedded balls $B_1,B_2... \subset \Reg(S^\infty/\Gamma)_\omega $ centered at points $x_1,x_2,...$ such that 
\begin{itemize}
\item we have $\spt W\subset  \bigcup_{j\geq1}B_j$,
\item the radius of each ball $B_j$ is smaller than $\frac{1}{10}$ times the injectivity radius of $\Reg(S^\infty/\Gamma)_\omega $ at $x_j$, 
\item the closure of each intersection $\spt W\cap 2B_j$ is compact, where $2B_j$ is the ball of center $x_j$ and radius twice that of $B_j$
\item each ball $2B_j$ only intersects finitely many balls of the form $2B_l$.

\end{itemize}
Set
$$\mathcal{B}:= \bigcup_{j\geq1}B_j.$$ 
We can also impose  that 
$\mathcal{B} \subset  N^c_\omega(\frac{d}{2}).$

Denote by $d_{\mathcal{B}}$ the metric on $\spt W$ defined as follows. Given  two points $a,b\in \spt W$, $d_{\mathcal{B}}(a,b)$ is the infimum of the lengths of piecewise geodesic curves inside $\mathcal{B}$ joining $a$ to $b$, such that each geodesic segment composing such curves has endpoints $u,v$ satisfying:
$u,v\in \spt W$ and there are $j,j'$ such that $B_j\cap B_{j'}\neq \varnothing$, $u\in B_j$, $v\in B_{j'}$. The distance $d_{\mathcal{B}}$ coincides with the original distance $\dist$ for pairs of points on $\spt W$ which are close enough, and is allowed to equal $\infty$.

After reordering the balls $B_j$ if necessary, we can find a large integer $J$ such that 
$$\mathbf{M}(W-W\llcorner \bigcup_{j=1}^J B_j) \leq \eta/10,$$
$$\mathbf{M}(\partial W - \partial W\llcorner \bigcup_{j=1}^J B_j)\leq \eta/10.$$
Given any compact subset $A\subset \spt W$, we can also ensure that $A\subset  \spt W\cap \bigcup_{j=1}^J B_j.$
(We cannot conclude yet, since in general $\partial W - \partial W\llcorner \bigcup_{j=1}^J B_j \neq \partial (W - W\llcorner \bigcup_{j=1}^J B_j)$.)

By the coarea formula and the slicing theorem \cite{AK00}\cite[Sections 1.4 and 1.5]{Antoine23a} applied to the following $1$-Lipschitz function (with respect to $d_{\mathcal{B}}$)
$$\varphi:\spt W\to \mathbb{R}$$
$$\varphi(x):= d_{\mathcal{B}}(x,\{x_1,...,x_J\}),$$
we find some large $R>0$ and an open subset $O_W$ of $\Reg(S^\infty/\Gamma)_\omega $ contained in  $N^c_\omega(\frac{d}{2})$, such that
the restrictions $W\llcorner O_W$ and $W\llcorner \varphi^{-1}([0,R))$ are equal and are integral currents, and 
$$A\subset \spt W\cap  \bigcup_{j=1}^J B_j\subset O_W,$$
$$\mathbf{M}(W-W\llcorner O_W) \leq \eta,$$
$$\mathbf{M}(\partial(W-W\llcorner O_W))\leq \eta.$$

To finish the proof, let us check that $\spt (W\llcorner O_W)$ is indeed compact. It is enough to show that $\spt W \cap \varphi^{-1}([0,R'])$ is compact for any $R'>0$. This can be shown by a continuity argument. In fact, using that $d_{\mathcal{B}} = \dist$ locally, and that each intersection $\spt W\cap 2B_j$ is compact and each $2B_j$ intersects finitely many other $2B_{j'}$, it is an exercise to check that the set of $R'>0$ such that $\spt W \cap \varphi^{-1}([0,R'])$ is compact is non-empty, open and closed, and thus equal to the whole interval $[0,\infty)$.

\end{proof}

\begin{lemme} \label{reduce to compact}
Consider some $n$-dimensional integral currents $S_0$, $U_0$  without boundary, and an $(n+1)$-dimensional integral current $V_0$ in $(S^\infty/\Gamma)_\omega$ such that
$$S_0=U_0+\partial V_0.$$
Let $d>0$ and $\eta>0$. 
Suppose that the closure of $\spt S_0 \cap U^{>\delta}$ is compact for any $\delta>0$, and that $\spt V_0 \subset N^c_\omega(d)$. 
Then there exists integral currents $U_2$, $V_2$ in $(S^\infty/\Gamma)_\omega$ such that 
$\spt V_2$ is compact and
$$S_0=U_2+\partial V_2,$$
$$\spt V_2 \subset N^c_\omega(d/2),$$
$$\mathbf{M}(U_2 )\leq \mathbf{M}(U_0) +\eta.$$

\end{lemme}
\begin{proof}

Similarly to  the proof of Lemma \ref{1st approx}, consider the conformally rescaled Riemannian metric on $N^c_\omega(d/2)$
$$\mathbf{g}_F:=F.\mathbf{g}_{\mathrm{Hil}},$$
where $F\geq 1$ everywhere, $F=1$ on $N^c_\omega(d)$, 
and for all $x$ close enough to the boundary of $N^c_\omega(d/2)$,
$$\frac{1}{\dist(x, N_\omega(d/2))^2} \leq F(x) \leq \frac{100}{\dist(x, N_\omega(d/2))^2}.$$
Consider the rescaled mass $\mathbf{M}_{\mathbf{g}_F}$.  The space $(N^c_\omega(d/2),\mathbf{M}_{\mathbf{g}_F})$ admits a local Euclidean type isoperimetric inequality.

Compared to Lemma \ref{1st approx}, we do not need to prove a uniform inequality of the type (\ref{monot ineq}). On the other hand, the currents $U_0, V_0$  are not assumed to be supported in  $U^{>\delta}$ for some $\delta$. For those reasons, it is convenient to apply the Ekeland principle to a different function. Let $\mathscr{S}(S_0)$ be the set of pairs of integral currents $(A,R)$ in $(N^c_\omega(d/2),\mathbf{g}_F)$ such that $S_0= A+\partial R$. 
Note that  $\mathscr{S}(S_0)$ is clearly closed with respect to the product of rescaled mass topologies, hence complete. 
By the Ekeland variational principle applied to $\varepsilon:=\frac{1}{10}$ and to the function 
$$(A,R) \mapsto \mathbf{M}_{\mathbf{g}_F}(A)+\frac{\eta  \mathbf{M}_{\mathbf{g}_F}(R)}{2(1+\mathbf{M}(V_0))},$$ we get  $(U_1,V_1)\in \mathscr{S}(S_0)$ such that
$$\spt U_1 \subset N^c_\omega(d/2), \quad \spt V_1 \subset N^c_\omega(d/2),$$
\begin{equation}\label{u1u0}
\mathbf{M}_{\mathbf{g}_F}(U_1)+ \frac{\eta  \mathbf{M}_{\mathbf{g}_F}(V_1)}{2(1+\mathbf{M}(V_0))} \leq \mathbf{M}_{\mathbf{g}_F}(U_0)+ \frac{\eta \mathbf{M}_{\mathbf{g}_F}(V_0) }{2(1+\mathbf{M}(V_0))} - \frac{1}{10}\big(\mathbf{M}_{\mathbf{g}_F}(U_0-U_1)+\mathbf{M}_{\mathbf{g}_F}(V_0-V_1)\big)
\end{equation}
and 
\begin{equation}\label{4 janv}
(A,R)\mapsto \mathbf{M}_{\mathbf{g}_F}(A)+  \frac{\eta  \mathbf{M}_{\mathbf{g}_F}(R) }{2(1+\mathbf{M}(V_0))} +\frac{1}{10}\big(\mathbf{M}_{\mathbf{g}_F}(A-U_1)+\mathbf{M}_{\mathbf{g}_F}(R-V_1)\big)
\end{equation}
is minimized at $(U_1,V_1)$. In (\ref{u1u0}), $\mathbf{M}_{\mathbf{g}_F}$ can be replaced by $\mathbf{M}$. In particular,
\begin{equation}\label{5 janv}
\mathbf{M}(U_1) \leq \mathbf{M}(U_0) +\frac{\eta}{2}.
\end{equation}

The local Euclidean type isoperimetric inequality implies that $\spt U_1$ is locally compact by the standard comparison argument. Roughly, for $x\in \spt U_1$ and any generic ball $b$ centered at a point $x'$ close enough to $x$, with small enough radius, 
the mass $\mathbf{M}_{\mathbf{g}_F}(U_1\llcorner b)$ can be made arbitrarily small, and $\partial (U_1\llcorner b)$ bounds an $n$-dimensional integral current $T$ in $b$ such that 
$$\mathbf{M}_{\mathbf{g}_F}(T) \leq \min\{\mathbf{M}_{\mathbf{g}_F}(U_1\llcorner b), \tilde{c}_{n-1} \mathbf{M}_{\mathbf{g}_F}(\partial(U_1\llcorner b))^{\frac{n}{n-1}}\}.$$ 
Moreover, $T- U_1\llcorner b$ bounds an  $(n+1)$-dimensional integral current $Q$ in $b$ such that 
$$\mathbf{M}_{\mathbf{g}_F}(Q) \leq \tilde{c}_{n} \mathbf{M}_{\mathbf{g}_F}(T- U_1\llcorner b)^{\frac{n+1}{n}}\leq \tilde{c}_{n}(2\mathbf{M}_{\mathbf{g}_F}(U_1\llcorner b))^{\frac{n+1}{n}}$$
which in turn is smaller than $\frac{1}{100} \mathbf{M}_{\mathbf{g}_F}(U_1\llcorner b)$ whenever $b$ is small enough depending on $U_1$ and the distance between $x'$ and $x$. Then, using the minimization property of $(U_1,V_1)$, and comparing with $(U_1- U_1\llcorner b+T,V_1+Q)$ in (\ref{4 janv}), in a similar way to \cite[Theorem 10.6]{AK00}, we get for some $c''>0$ and $\rho_0>0$:
$$\forall \rho\leq \rho_0,\quad \mathbf{M}_{\mathbf{g}_F}(U_1\llcorner B(x',\rho)) \geq c'' \rho^n$$
where $c''$, $\rho_0$ only depend on the (small enough) distance between $x'$ and $x$. 
This implies that $\spt U_1$ is locally compact. A similar argument then shows that $\spt V_1$ is also locally compact.

To finish the proof, we apply Lemma \ref{apply coarea} to $V_1$ and $\frac{\eta}{2}$, and find the corresponding open subset $O_{V_1}$, then we set
$$V_2:= V_1\llcorner O_{V_1}, \quad U_2 = S_0-\partial V_2.$$
By (\ref{5 janv}) and Lemma \ref{apply coarea}, we get 
$$\mathbf{M}(U_2) =\mathbf{M}( U_1+\partial V_1-\partial V_2) \leq \mathbf{M}(U_1) +\frac{\eta}{2}\leq \mathbf{M}(U_0) +\eta.$$

\end{proof}


\subsection{Pulled-tight sequence and no mass cancellation} \label{thickdefinition}

The next two sections pertain to the general structure of spherical Plateau solutions. Here, we first improve the convergence of a minimizing sequence to a spherical Plateau solution.


Given $\Gamma$ and $h\in H_n(\Gamma;\mathbb{Z})$, let $C'_i \in \mathscr{C}(h)$ be a minimizing sequence converging to  a spherical Plateau solution $$C_\infty = (X_\infty,d_\infty,S_\infty)$$  in the intrinsic flat topology.
The following result improves the convergence of $C'_i$: it states that by pulling tight the minimizing sequence, the supports of the ``noncollapsed parts'' converge in a Gromov-Hausdorff sense and without cancellation in the sense of currents. 
Recall that ${C}_i^{>\frac{1}{k}}$ is defined in (\ref{c>delta}).

\begin{theo} \label{compactness}
There is a minimizing sequence $\{{C}_i\} \subset \mathscr{C}(h)$ converging to $C_\infty$ in the intrinsic flat topology,  
which is pulled-tight in the following sense: 
there is a Banach space $\mathbf{Z}$ containing isometrically $\spt(S_\infty)$, and there are isometric embeddings  $j_i : \spt(C_i) \hookrightarrow \mathbf{Z}$ such that 
\begin{enumerate}
\item
$(j_i)_\#(C_i) $ converges to $S_\infty$ in the flat topology inside $\mathbf{Z}$,
\item
for each positive integer $k$, $j_i(\spt({C}_i^{>\frac{1}{k}}))$ converges in the Hausdorff topology to some compact subset $X_{k,\infty}$ of $\spt(S_\infty)$ inside $\mathbf{Z}$.
\end{enumerate}

\end{theo}

\begin{proof}


As a first step, we  prove that there is another minimizing sequence $\{C_p\}\subset \mathscr{C}(h)$ converging to $C_\infty$ in the intrinsic flat topology, such that for each positive integer $k$, $\spt({C}_p^{>\frac{1}{k}})$ converges in the Gromov-Hausdorff topology to some compact metric space $X_{k,\infty}$.

For each $p$, apply Lemma \ref{1st approx} to $S_0=U_0= C'_p$, and let $U_1$, $V_1$ be as in the statement of the lemma. Set $C_p:= U_1$, then since 
$C'_p = C_p +\partial V_1$ and $V_1$ has compact support, we have
$$C_p\in \mathscr{C}(h).$$
In particular, $\mathbf{M}(C_p) \geq \spherevol(h)$.
Moreover, Lemma \ref{1st approx} also gives 
$$\frac{1}{10}\mathbf{M}(C'_p-C_p) \leq \mathbf{M}(C'_p) - \mathbf{M}(C_p).$$
By assumption, $\lim_{p\to \infty}\mathbf{M}(C'_p)  =\spherevol(h)$. The above inequality therefore implies that
$$\lim_{p\to \infty} \mathbf{M}({C}_p) = \lim_{p\to \infty} \mathbf{M}({C}'_p)=\spherevol(h)$$
and that
$C_p$ converges to $C_\infty$ too in the intrinsic flat topology, as $p\to \infty$.

Inequality (\ref{monot ineq}) in Lemma \ref{1st approx} also implies the following. Given $k>0$, for any $p$, for all 
$x\in \spt({C}_p) \cap [S^\infty/\Gamma]^{\frac{1}{2k}}$
and $\rho\in(0,\frac{1}{200k})$, 
\begin{equation} \label{equicomp}
\mathbf{M}({C}_p \llcorner B(x,\rho)) > c_{n} \rho^n
\end{equation}
for some strictly positive dimensional constant $c_{n}$. 
This inequality shows that the sequence of sets $ \spt {C}_p \cap [S^\infty/\Gamma]^{\frac{1}{2k}} $ is equi-compact  as $p\to \infty$ (see \cite[Definition 6.5]{AK00}), and of course those sets have uniformly bounded diameter.
By Gromov's compactness theorem,  subsequentially $ \spt(C_p^{>\frac{1}{k}})$ converges in the Gromov-Hausdorff topology to a limit  
$X_{k,\infty}$. We can choose the subsequence to be of full measure with respect to the non-principal ultrafilter $\omega$.
The first step is completed.

By Lemma \ref{in a common Banach} in the Appendix, there exist a Banach space $\mathbf{Z}$, and isometric embeddings
$$j_p : \spt(C_p) \hookrightarrow \mathbf{Z}, \quad \spt(S_\infty)\hookrightarrow \mathbf{Z}, \quad X_{k,\infty} \hookrightarrow \mathbf{Z},$$
such that $(j_p)_\#(C_p)$ converges in the flat topology to $S_\infty$  inside $\mathbf{Z}$, and for each $k>0$, $\{(j_p)_\#(\spt(C_p^{>\frac{1}{k}}))\}$ converges in the Hausdorff topology to $X_{k,\infty}$.

In the second step, we want to show that for each $k$, the Hausdorff limit $X_{k,\infty}$ of $(j_p)_\#(\spt(C_p^{>\frac{1}{k}}))$ is contained inside $\spt(S_\infty) \subset \mathbf{Z}$, in other words that there is no mass cancellation. 

By construction, for any $k\leq k'$, 
$$X_{k,\infty} \subset X_{k',\infty}.$$
Fix $k$ and let $x\in  X_{k,\infty}$ be a limit of points $j_p(x_p) \in j_p(\spt(C_p^{>\frac{1}{k}}))$. 
Suppose by contradiction that $x$ is not contained in $\spt(S_\infty)$. Then  there exists some $r\in (0,\frac{1}{10k})$ so that
\begin{equation} \label{tinfinityb'}
\mathbf{M}(S_\infty \llcorner B'(x,r))  =0 ,
\end{equation}
where $B'(x,r)$ is the $r$-ball  in $\mathbf{Z}$ centered at $x$. Note that for any $p$, the ball $B(x_p,r)$  is entirely contained inside $[S^\infty/\rho(\Gamma)]^{>\frac{1}{2k}}$. Hence by (\ref{tinfinityb'}), by the slicing theorem \cite[Theorem 5.6]{AK00} and by compactness \cite[Theorem 4.6]{SW11}, for some choice of $r_p\in (r/2,r)$, 
$ C_p \llcorner B(x_p,r_p)$ is an integral current with compact support, such that $\mathbf{M}( C_p \llcorner B(x_p,r_p))+\mathbf{M}(\partial  (C_p \llcorner B(x_p,r_p)))$ is uniformly bounded as $p\to \infty$, and which converges in the intrinsic flat topology to the zero current.
By \cite{Wenger07}, for any $\epsilon$, whenever $p$ is large enough there are integral currents  $U,V$ in $B(x_p,r)$ depending on $p$, such that
$$C_p \llcorner B(x_p,r_p) = U+\partial V,$$
$$\mathbf{M}(U)\leq \epsilon.$$
By Lemma \ref{1st approx} applied to $S_0 =C_p \llcorner B(x_p,r_p)$, $U_1=U$, $V_1=V$, $\delta_0=\frac{1}{2k}$,
the integral currents $U$ and $V$ can be chosen to have compact supports without loss of generality.
Then, set 
$$W_p := C_p - C_p \llcorner B(x_p,r_p) + U = C_p - \partial V.$$
Since $U$, $V$ have compact supports, $W_p \in \mathscr{C}(h)$. 
By the uniform mass lower bound (\ref{equicomp}) for ${C}_p$,
\begin{align*}
\mathbf{M}(W_p) & \leq \mathbf{M}(C_p) - \mathbf{M}(C_p \llcorner B(x_p,r_p))  +\epsilon\\
& \leq \mathbf{M}({C}_p) - c'_{k,n} r_p^n +\epsilon  \\
& \leq \mathbf{M}({C}_p) - (\frac{c'_{k,n}r^n}{2^n}-\epsilon).
\end{align*}
Choosing $\epsilon$ so small that $\frac{c'_{k,n}r^n}{2^n}-\epsilon> 0$, we obtain from the above inequality that
$$\limsup_{p\to \infty}\mathbf{M}(W_p)< \lim_{p\to \infty} \mathbf{M}({C}_p) = \lim_{p\to \infty} \mathbf{M}({C}'_p)=\spherevol(h)$$ which is a contradiction with $W_p \in \mathscr{C}(h)$. We conclude that in fact $x\in \spt(S_\infty)$. This finishes the second step and the proof.

\end{proof}


Let $C_\infty=(X_\infty,d_\infty,S_\infty)$ be a spherical Plateau solution for $h$.
Let $\{C_i\}\subset \mathscr{C}(h)$ be a minimizing sequence converging to $C_\infty$ in the intrinsic flat topology. We suppose that 
$$\text{$\{C_i\}$ is pulled-tight in the sense of Theorem \ref{compactness}}.$$
Let $X_{k,\infty} \subset \spt(S_\infty)$ be the compact sets as in Theorem \ref{compactness}. Consider the non-collapsed part  $C_\infty^{>0} = (Y_\infty,d_\infty,T_\infty)$ as in Definition \ref{definition:noncollapsed}.  Recall the notation $U^{>\delta}$ from (\ref{omegaU}). The lemma below follows from definitions.

\begin{lemme} \label{lemma:yinfinity}
We have
$$Y_\infty=\bigcup_{k>0}X_{k,\infty},\quad T_\infty = S_\infty \llcorner Y_\infty,$$ 
$$X_{k,\infty} \subset \text{closure of $\spt S_\infty \cap U^{>\frac{1}{k}}$} \subset X_{k+1,\infty}.$$
In particular, for every $\delta>0$, the closure of $\spt S_\infty \cap U^{>\delta}$ is compact.
\end{lemme}


\subsection{Minimality of the non-collapsed part}
\label{hilbert spherical quotients}

In this section, we show that the non-collapsed part of a spherical Plateau solution is always mass-minimizing.
We continue to use the notations introduced earlier.
Let $C_i$ be a minimizing sequence converging to a Plateau solution 
$$C_\infty = (X_\infty,d_\infty,S_\infty)$$ in the intrinsic flat topology. Consider the boundaryless integral current $S_\infty$, which is  an integral current inside $(S^\infty/\Gamma)_\omega$, see Lemma \ref{ultraembedding1}. Let 
$$C_\infty^{>0}=(Y_\infty,d_\infty,T_\infty)$$
be its noncollapsed part as defined in Definition \ref{definition:noncollapsed}.

\begin{theo} \label{theorem:locally mass minimizing}
The current $T_\infty$ is mass-minimizing inside $\Reg(S^\infty/\Gamma)_\omega $ in the following sense: for any
$(n+1)$-dimensional integral current  $Q$ in $\Reg(S^\infty/\Gamma)_\omega $ with
$$\spt Q\subset N^c_\omega(d)$$
for some $d>0$, we have
$$\mathbf{M}(T_\infty)\leq \mathbf{M}(T_\infty +\partial Q).$$
\end{theo}

\begin{proof}

Let $\{{C}_i\}$ be a pulled-tight sequence converging to $C_\infty$ as in Theorem \ref{compactness}, and let 
$$j_i : \spt(C_i) \hookrightarrow \mathbf{Z}, \quad \spt(S_\infty) \hookrightarrow \mathbf{Z},$$
be as in Theorem \ref{compactness}.
Let $Q$ be an $(n+1)$-dimensional integral current in $\Reg(S^\infty/\Gamma)_\omega $ such that
for some $d>0$, 
$$\spt Q\subset N^c_\omega(d).$$
Suppose towards a contradiction that for some $\alpha>0$,
$$\mathbf{M}(T_\infty +\partial Q) \leq \mathbf{M}(T_\infty) - \alpha$$
or equivalently
$$\mathbf{M}(S_\infty +\partial Q) \leq \mathbf{M}(S_\infty) - \alpha.$$

By Lemma \ref{lemma:yinfinity}, the closure of $\spt S_\infty \cap U^{>\delta}$ is compact for any $\delta>0$.
By Lemma \ref{reduce to compact} applied to $S_0 = S_\infty$, $V_0=-Q$, $U_0=S_\infty +\partial Q$ and $\eta\in (0,\alpha/2)$, we can assume that $Q$ has compact support after reducing $d$ if necessary.

Let $\eta_i\in (0,\alpha/2)$ be a sequence converging to $0$. 
By the approximation lemma,  Lemma \ref{approx before limit}, which is stated right after this proof, there exists an integral current $C'_i\in \mathscr{C}(h)$ in $S^\infty/\Gamma$ such that 
\begin{align*}
\mathbf{M}(C'_i)& \leq \spherevol(h) - \mathbf{M}(S_\infty) +\mathbf{M}(S_\infty + \partial Q) +\eta_i \\
& \leq \spherevol(h) - \alpha +\eta_i.
\end{align*}
Letting $i\to \infty$, we get a contradiction. This ends the proof of the theorem.

\end{proof}

In the proof of the previous theorem, we needed the approximation lemma below in order to approximate limits in $(S^\infty/\Gamma)_\omega$ by cycles in $S^\infty/\Gamma$. 
Recall that $S_\infty$ is the integral current in $(S^\infty/\Gamma)_\omega$ coming from a spherical Plateau solution $C_\infty=(X_\infty,d_\infty,S_\infty)$.
\begin{lemme} \label{approx before limit}
Consider an $n$-dimensional integral current $U_0$ in $(S^\infty/\Gamma)_\omega$ without boundary and an $(n+1)$-dimensional integral current $V_0$ in $(S^\infty/\Gamma)_\omega$ such that
$$S_\infty=U_0+\partial V_0.$$
Suppose that $\spt V_0$ is compact and contained inside $\Reg(S^\infty/\Gamma)_\omega $.
Let $\eta>0$. Then there exists integral currents $C\in  \mathscr{C}(h)$ in $S^\infty/\Gamma$ such that 
$$\mathbf{M}(C) \leq \spherevol(h) - \mathbf{M}(S_\infty) + \mathbf{M}(U_0) +\eta.$$

\end{lemme}

\begin{proof}

 Choose a small $\delta>0$ and let $\mathcal{O} \subset  U^{>\delta}$ be an open subset containing $\spt V_0$; such an $\mathcal{O}$  exists when $\delta$ is chosen small enough by compactness of $\spt V_0$. Recall by Lemma \ref{lemma:yinfinity} that $\spt S_\infty \cap U^{>\delta}$ is compact. By the slicing theorem, one can choose $\mathcal{O}$ so that both $S_\infty\llcorner \mathcal{O}$ and  $U_0 \llcorner \mathcal{O}$ are integral currents with compact supports.
Given $\epsilon>0$, by a standard polyhedral chain approximation explained in \cite[Lemma 1.6]{Antoine23a}, there are polyhedral chains $P_1, P_2, R$ supported in $\mathcal{O}$ such that 
\begin{itemize}
\item
$\partial R = P_2-P_1$,
\item
$P_1$ (resp. $P_2$, $R$) is $\epsilon$-close to $S_\infty\llcorner \mathcal{O}$ (resp. $U_0 \llcorner \mathcal{O}$, $V_0$) in the flat topology,  
\item
$\mathbf{M}(P_1) \leq \mathbf{M}(S_\infty\llcorner \mathcal{O}) +\epsilon$,
\item
$\mathbf{M}(P_2) \leq \mathbf{M}(U_0\llcorner \mathcal{O}) +\epsilon$.
\end{itemize}

By definition of polyhedral chains, there are finite simplicial complexes $K_1,K_2, L$ with an orientation, such that $K_1\cup K_2 \subset L$, $\partial L = K_2-K_1$, and a Lipschitz map 
$$\phi:L \to \mathcal{O}$$ 
such that
$$\phi_\# \llbracket 1_{K_1} \rrbracket = P_1, \quad \phi_\#\llbracket 1_{K_2} \rrbracket = P_2, \quad \phi_\#\llbracket 1_{L} \rrbracket = R.$$
By refining the simplicial complexes and redefining the map $\phi$,
we can ensure that  there is no cancellation in the sense of currents so that 
\begin{equation} \label{pas d'annulat}
\Vol(K_1, \phi^*\mathbf{g}_{\mathrm{Hil}}) = \mathbf{M}(P_1) \text{ and }
\Vol(K_2, \phi^*\mathbf{g}_{\mathrm{Hil}}) = \mathbf{M}(P_2),
\end{equation}
where $\phi^*\mathbf{g}_{\mathrm{Hil}}$ denotes the pull-back Riemannian metric.
We can always suppose the triangulation of $L$ fine enough so that any $k$-simplex inside $L$ is sent by $\phi$ to a totally geodesic $k$-simplex embedded inside a small convex embedded geodesic ball of $\mathcal{O} \subset  \Reg(S^\infty/\Gamma)_\omega $.

Let $C_i\in \mathscr{C}(h)$ be a pulled-tight sequence converging to $C_\infty=(X_\infty;d_\infty,S_\infty)$, with $\lim_{i\to \infty}\mathbf{M}(C_i) = \spherevol(h)$, as supposed at the beginning of this subsection. 
By Theorem \ref{compactness}, by the slicing theorem and Wenger's compactness \cite{Wenger11}, by changing a bit $\mathcal{O}$ if necessary, we have a sequence of open sets $\mathcal{O}_i \subset [S^\infty/\Gamma]^{>\delta/2}$, such that 
\begin{itemize}
\item
 $ \spt(C_i \llcorner \mathcal{O}_i)$ converges in the Gromov-Hausdorff topology to $\spt(S_\infty\llcorner \mathcal{O})$,
 \item $C_i \llcorner \mathcal{O}_i$  is an integral current  whose boundary has uniformly bounded mass and which converges in the intrinsic flat topology to $S_\infty\llcorner \mathcal{O}$.
 \end{itemize}

By the definition of ultralimits, there exist Lipschitz maps 
$$\phi_i : L \to [S^\infty/\Gamma]^{>\delta/2}$$
such that if $\{p_1,...,p_m\}$ denotes the vertices of $L$, the metric spaces 
$$\{\phi_i(p_1),...,\phi_i(p_m)\}, \quad \phi_i(L), \quad \phi_i(L)\cup \spt(C_i\llcorner \mathcal{O}_i)$$
with the metric induced by $S^\infty/\Gamma$ converge subsequentially in the Gromov-Hausdorff topology respectively to 
$$\{\phi(p_1),...,\phi(p_m)\}, \quad \phi(L), \quad \phi(L) \cup \spt(S_\infty\llcorner \mathcal{O})$$ 
with the metric induced by $(S^\infty/\Gamma)_\omega$ and for any $k$-simplex $\Delta\subset L$, $\phi_i(\Delta)$ is the totally geodesic embedded $k$-simplex determined by its vertices.
By Lemma \ref{in a common Banach} in the Appendix, after taking a renumbered subsequence, we get new isometric embeddings
$$j'_i: \spt(C_i) \cup \phi_i(L) \hookrightarrow \mathbf{Z}', \quad \spt(S_\infty) \cup  \phi(L) \subset \mathbf{Z}',$$
into a Banach space $\mathbf{Z}'$ such that the conclusion of Theorem \ref{compactness} is still true, and moreover:
\begin{itemize}
\item $j'_i(\phi_i(p_j))$ converges to $\phi(p_j)$ for each $j\in \{1,...,m\}$, 
and $j'_i(\phi_i(L)) \cup j'_i(\spt(C_i\llcorner \mathcal{O}_i))$ converges in the Hausdorff topology to $\phi(L) \cup \spt(S_\infty\llcorner \mathcal{O})$,
\item $(j'_i)_\#(C_i\llcorner \mathcal{O}_i)$ and  $(j'_i\circ \phi_i)_\#\llbracket 1_{K_1} \rrbracket $ converge in the flat topology respectively to $S_\infty\llcorner \mathcal{O}$ and $P_1= \phi_\#\llbracket 1_{K_1} \rrbracket$.
\item by (\ref{pas d'annulat}), we have
\begin{equation} \label{mass images}
 \lim_{i\to \infty} \mathbf{M}( (\phi_i)_\#\llbracket 1_{K_2} \rrbracket ) = \mathbf{M}(P_2).
 \end{equation}
\end{itemize}
Since the flat distance between $S_\infty\llcorner \mathcal{O}$ and $P_1= \phi_\#\llbracket 1_{K_1} \rrbracket$ is less than $\epsilon$, we have that  for $i$ large, 
\begin{equation}\label{leq é epsilon}
\mathbf{d}_\mathcal{F}(0, (\phi_i)_\#\llbracket 1_{K_1} \rrbracket - C_i\llcorner\mathcal{O}_i) \leq 2\epsilon
\end{equation}
where $\mathbf{d}_\mathcal{F}$ is the intrinsic flat distance.
Then for any $\epsilon_1>0$, if $\epsilon$ is chosen small enough, we claim that there are integral currents $U,V$ in  $[S^\infty/\Gamma]^{>\delta/2}$ so that
$$(\phi_i)_\#\llbracket 1_{K_1} \rrbracket - C_i\llcorner \mathcal{O}_i  = U+ \partial V,$$
\begin{equation}\label{mumveta}
\mathbf{M}(U) < \epsilon_1.
\end{equation}
The argument for checking this claim is quite standard: note that by Gromov-Hausdorff convergence, $$\spt (C_i\llcorner\mathcal{O}_i) \cup \phi_i(K_1),$$ can be covered by $J$ convex balls $b_{i,1},...,b_{i,J}$ of radius $\frac{\delta}{100}$, where $J$ is a number that does not depend on $i$.   For each $j=1,...,J$, after maybe moving the balls $b_{i,j}$ a bit and using the slicing theorem, there are domains $d_1,...d_{J'}$ with piecewise smooth boundary for some number $J'$ depending only on $J$, such that $\bigcup_{j=1}^{J'} d_j = \bigcup_{j=1}^{J} b_j$, the interiors of $d_j$ are pairwise disjoint, and 
$$((\phi_i)_\#\llbracket 1_{K_1} \rrbracket - C_i\llcorner \mathcal{O}_i) \llcorner d_{i,j}$$ 
are integral currents whose sum over $j=1,...,J'$ is equal to $(\phi_i)_\#\llbracket 1_{K_1} \rrbracket - C_i\llcorner \mathcal{O}_i$. Moreover,
by (\ref{leq é epsilon}), for all large $i$, for some $\varphi(\epsilon)>0$ converging to $0$ as $\epsilon \to 0$,
$$\mathbf{d}_\mathcal{F}(0, ((\phi_i)_\#\llbracket 1_{K_1} \rrbracket - C_i\llcorner \mathcal{O}_i) \llcorner d_{i,j}) \leq \varphi(\epsilon).$$
By \cite[Theorem 1.4]{Wenger07} applied to a convex ball containing $d_{i,j}$, taking $\epsilon$ small enough, for each $j$ and all large $i$,  
$$((\phi_i)_\#\llbracket 1_{K_1} \rrbracket - C_i\llcorner \mathcal{O}_i) \llcorner d_{i,j} =U_{i,j}+ \partial V_{i,j} \quad \text{with}\quad  \mathbf{M}(U_{i,j}) \leq \frac{\epsilon_1}{J}.$$ 
 Now, just set $U=\sum_{j=1}^{J'} U_{i,j}$, $V=\sum_{j=1}^{J'} V_{i,j}$ and the claim is checked.

By Lemma \ref{1st approx}, $U,V$ can be chosen to have compact supports.
Set 
$$D:= C_i - C_i\llcorner \mathcal{O}_i -U + (\phi_i)_\#\llbracket 1_{K_1} \rrbracket  = C_i+\partial V,$$
and 
$$\hat{D} := D +\partial (\phi_i)_\#\llbracket 1_{L} \rrbracket  = D- (\phi_i)_\#\llbracket 1_{K_1} \rrbracket  + (\phi_i)_\#\llbracket 1_{K_2} \rrbracket .$$
From those equalities, clearly $\hat{D} \in \mathscr{C}(h)$.
Since $C_i \llcorner \mathcal{O}_i$ converges in the intrinsic  flat topology to $S_\infty\llcorner \mathcal{O}$, lower semicontinuity of the mass tells us that 
\begin{equation}\label{oio...}
\liminf_{i\to \infty} \mathbf{M}(C_i \llcorner \mathcal{O}_i) \geq \mathbf{M}(S_\infty \llcorner \mathcal{O}).
\end{equation}
Given $\epsilon_1>0$, if $i$ is large enough and $\epsilon$ small enough, (\ref{mass images}), (\ref{oio...}) and (\ref{mumveta}) finally yield
\begin{align*}
\mathbf{M}(\hat{D}) & = \mathbf{M}(C_i - C_i\llcorner \mathcal{O}_i  -U + (\phi_i)_\#\llbracket 1_{K_2} \rrbracket) \\
& \leq \mathbf{M}(C_i)- \mathbf{M}(C_i\llcorner \mathcal{O}_i) + \mathbf{M}(U) + \mathbf{M}((\phi_i)_\#\llbracket 1_{K_2} \rrbracket) \\
& \leq \mathbf{M}(C_i)- \mathbf{M}(C_i\llcorner \mathcal{O}_i) + \epsilon_1 + \mathbf{M}(P_2)\\
& \leq \mathbf{M}(C_i)- \mathbf{M}(C_i\llcorner \mathcal{O}_i)  + \mathbf{M}(U_0\llcorner \mathcal{O}) + 2\epsilon_1\\
& \leq \spherevol(h) - \mathbf{M}(S_\infty\llcorner \mathcal{O})  + \mathbf{M}(U_0\llcorner \mathcal{O}) + 3\epsilon_1\\
& = \spherevol(h) - \mathbf{M}(S_\infty)  + \mathbf{M}(U_0) + 3\epsilon_1
\end{align*}
where the last equality is due to the fact that $S_\infty$ and $U_0$ coincide as currents outside of $\mathcal{O}$. The lemma is proved by taking $C:=\hat{D}$, $\eta:=3\epsilon_1$.

\end{proof}

\section{Hyperbolic groups and non-collapsing} \label{section:noncollapsing}

\subsection{Limits  of the regular representation} \label{subsection:lambdaomega}
In this section, we prove general structure results for limits of the regular representation of torsion-free hyperbolic groups.
Let $\Gamma$ be a given non-elementary torsion-free hyperbolic group, with identity element $1$.
Let 
$\lambda_{\Gamma}$ be the regular representation of $\Gamma$ on $\ell^2(\Gamma)$. As usual, $\langle.,.\rangle$ denotes the scalar product on $\ell^2(\Gamma)$. 
The following lemma seems to be a new observation. It will be our basic tool for describing the ultralimit representation $\lambda_\omega:\Gamma_\omega\to O(H_\omega)$ of (\ref{ultralimit representation}).

\begin{lemme} \label{degen of reg rep}
Let $\{g_j\}\subset \Gamma\setminus\{1\}$, $\{h_j\}\subset \Gamma\setminus\{1\}$ satisfy $g_jh_j\neq h_j g_j$ for all $j$.
\begin{enumerate}
\item If a sequence $\{v_j\}\subset S^\infty\subset \ell^2(\Gamma)$ is such that
$$\lim_{j\to \infty} \langle \lambda_{\Gamma}(g_j) .v_j,v_j \rangle = 1,$$
then
$$\lim_{j\to \infty} \langle \lambda_{\Gamma}(h_j) .v_j,v_j \rangle = 0.$$
\item If two sequences 
$\{v_j\}, \{w_j\}\subset S^\infty\subset \ell^2(\Gamma)$ are such that
$$\lim_{j\to \infty} \langle \lambda_{\Gamma}(g_j) .v_j,v_j \rangle = 1,\quad \lim_{j\to \infty} \langle \lambda_{\Gamma}(h_j) .w_j,w_j \rangle = 1,$$
then
$$\lim_{j\to \infty} \langle v_j,w_j \rangle = 0.$$
\end{enumerate}
\end{lemme}

\begin{proof}
We begin by showing property (1).
First let us check that this is true if $\Gamma$ is the free group of rank two, $F_2=\langle a,b \rangle$, generated by two elements $a$ and $b$, and if $g_j=a$, $h_j=b$. Indeed let $v_j \in  S^\infty\subset \ell^2(F_2)$ be such that 
$$\lim_{j\to \infty} \langle \lambda_{F_2}(a) .v_j,v_j \rangle = 1.$$
For preliminaries on the boundary and compactification of hyperbolic groups, see Subsection \ref{hyp groups} in the Appendix.
The sequence of functions $v_j^2$ viewed as probability measures on the compactification 
$$\overline{F_2} := F_2 \cup \partial F_2$$
weakly converges subsequentially to a probability measure $\mu$ on   $\overline{F_2}$ such that
 $$\lambda_{F_2}(a)_* \mu = \mu$$
 where $\lambda_{F_2}(a)$ denotes the extended isometry induced by $a$ on $\overline{F_2}$. From the elementary facts about the group action on the boundary recalled in Subsection \ref{hyp groups}, this implies that $\mu$ is supported on at most two points $q_1,q_2\in \partial F_2$, where $q_1,q_2$ are the two fixed points of $a$. But
 \begin{equation}\label{disjoint support mu}
\{ \lambda_{F_2}(b).q_1, \lambda_{F_2}(b).q_2\}\cap \{q_1,q_2\} =0,
 \end{equation}
in other words $\lambda_{F_2}(b)_* \mu$ and $\mu$ are probability measures with disjoint supports. Let us check that this implies the desired claim:
  \begin{equation}\label{desired claim f2}
 \lim_{j\to \infty} \langle \lambda_{F_2}(b) .v_j,v_j \rangle =0.
 \end{equation}
 Consider neighborhoods $U_1, U_2$ of $\{q_1,q_2\},  \{ \lambda_{F_2}(b).q_1, \lambda_{F_2}(b).q_2\}$ respectively, such that $U_1\cup U_2= \overline{F_2}$, $U_1\cap U_2=\varnothing$ and
   \begin{equation}\label{lim 0 f2}
 \lim_{j\to \infty} \sum_{x\in U_1} v_j^2(x) = \lim_{j\to \infty} \sum_{x\in U_2}  (\lambda_{F_2}(b).v_j(x))^2 =0.
 \end{equation}
 Then by Peter-Paul inequality, for any $\hat{\epsilon}>0$, for any $j$,
\begin{align*}
\langle \lambda_{F_2}(b) .v_j,v_j \rangle   = & \sum_{x\in U_1} (\lambda_{F_2}(b).v_j)(x) v_j(x) + \sum_{x\in U_2} (\lambda_{F_2}(b).v_j)(x) v_j(x) \\
 \leq &\quad   \frac{\hat{\epsilon}}{2} \sum_{x\in U_1}  (\lambda_{F_2}(b).v_j(x))^2  +\frac{1}{2\hat{\epsilon}} \sum_{x\in U_1} v_j^2(x)
\\
& +
\frac{\hat{\epsilon}}{2} \sum_{x\in U_2} v_j^2(x) +\frac{1}{2\hat{\epsilon}} \sum_{x\in U_2}  (\lambda_{F_2}(b).v_j)^2(x).
\end{align*}
By (\ref{lim 0 f2}), the above expression is smaller than say $3\hat{\epsilon}$ when $j$ is large, and so (\ref{desired claim f2}) is proved.

This fact extends directly to the following statement: given $F_2=\langle a,b \rangle$ as above, embedded as a subgroup of some groups $\Gamma_j$, 
 let $\{v_j\}\subset S^\infty\subset \ell^2(\Gamma_j)$ be a sequence such that 
$$\lim_{j\to \infty} \langle \lambda_{\Gamma_j}(a) .v_j,v_j \rangle = 1,$$
then
$$\lim_{j\to \infty} \langle \lambda_{\Gamma_j}(b) .v_j,v_j \rangle = 0.$$

Now consider the general case and let $v_j$, $g_j$, $h_j$ be sequences  satisfying the assumptions of  the statement.
By Delzant's result \cite[Th\'{e}or\`{e}me I]{Delzant96}, see Theorem \ref{theorem:delzant} in the Appendix, for some integer $k=k(\Gamma)$ and  each $j$, the subgroup generated by 
$$a_j:=g_j^k \quad \text{and} \quad b_j:=h_j^{-1}g^k_j h_j$$ is a free group of rank 2. 
By assumption, $v_j$ is almost fixed by $g_j$, so 
$$\lim_{j\to \infty} \langle \lambda_{\Gamma}(a_j).v_j,v_j \rangle = 1.$$
By the conclusion of the first paragraph of this proof, we already obtain that for any positive integer $N$,
$$\lim_{j\to \infty} \langle \lambda_{\Gamma}(b_j^N). v_j , v_j \rangle =0.$$

The end of the argument is based on basic spherical geometry. 
First, for simplicity let us assume that $r$, $s$ are isometric transformations of the separable Hilbert unit sphere $S^\infty$ such that for some $v\in S^\infty$, and for any positive integer $N$,
\begin{equation}\label{spherical geometry}
r(v)=v\quad \text{and}\quad \langle s^{-1}r^N s (v) , v\rangle = 0.
\end{equation}
\textbf{Claim:} We have $\langle  s (v) , v\rangle=0$.

Indeed, fix $\epsilon_1>0$. 
Let us first check that there are a positive integer $l_1=l_1(\epsilon_1)$ and a positive constant  $\eta=\eta(\epsilon_1)>0$ independent of $r$, $s$, $v$ such that if $\alpha:=\langle  s (v) , v\rangle \geq \epsilon_1$, then for some $\tilde{l}\in \{1,...,l_1\}$:
\begin{equation}\label{6jann}
\langle s (v), r^{\tilde{l}} s (v) \rangle \geq \frac{\epsilon_1^2}{2}.
\end{equation}
Since $v$ is fixed by $r$, we can write $r^l s (v) = \alpha v+\beta_l u_l$ where $|\alpha|^2+|\beta_l|^2=1$, $\|u_l\|=1$, $\beta_l\in [0,1]$, $u_l\perp v$. 
For any $l\neq l'$,
\begin{equation}\label{rvs}
\langle r^l s (v), r^{l'} s (v) \rangle = \langle \alpha v+\beta_l u_l, \alpha v+\beta_{l'} u_{l'} \rangle = \alpha^2 +\beta_l\beta_{l'} \langle u_l,u_{l'}\rangle \geq \epsilon_1^2 - |\langle u_l,u_{l'}\rangle| .
\end{equation}
\textbf{Fact:} Any finite family of unit vectors $\{u_1,...,u_{l_1}\}$ in a Hilbert space contains two vectors $u_l,u_{l'}$ with $l\neq l'$, such that $\langle u_l,u_{l'} \rangle \geq -\frac{\epsilon_1^2}{2}$, if $l_1$ is large enough depending only on $\epsilon_1$. 

This fact is checked by expanding the inequality $0\leq \langle \sum_{k=1}^{l_1}u_k ,  \sum_{k=1}^{l_1}u_k\rangle$. It proves what we wanted, because applying that fact with (\ref{rvs}), we find $l < l'\in \{1,...,l_1\}$ such that
$$\langle s (v), r^{l'-l} s (v) \rangle  = \langle r^l s (v), r^{l'} s (v) \rangle  \geq \frac{\epsilon_1^2}{2}.$$
Next, (\ref{6jann}) means that $\langle s (v), r^{\tilde{l}} s (v) \rangle \geq \frac{\epsilon_1^2}{2}$, which contradicts the second part of
(\ref{spherical geometry}). 
Similarly, we get a contradiction if $\langle  s (v) , v\rangle \leq -\epsilon_1$. 
Thus $|\langle  s (v) , v\rangle| \leq \epsilon_1$ and since this holds for any $\epsilon_1>0$, we deduce $\langle  s (v) , v\rangle =0$ as in the claim.

Now coming back to our proof, we can applied an approximate version of the above argument to $a_j=g_j^k$, $b_j=h_j^{-1}g^k_j h_j$, $v_j$ and conclude that
$$\lim_{j\to \infty} \langle \lambda_{\Gamma}(h_j) .v_j,v_j \rangle = 0$$
which ends the proof of property (1) of the lemma.

Property (2) the lemma, which is generalizes (1), is shown following a similar logic. In fact, using (1) and $\lim_{j\to \infty} \langle \lambda_{\Gamma}(g_j) .v_j,v_j \rangle = 1$, we find that $\lim_{j\to \infty} \langle \lambda_{\Gamma}(h_j^{N}) .v_j,v_j \rangle = 0$ for any integer $N$. If for some $\epsilon_1>0$, we have $\liminf_{j\to \infty} |\langle v_j,w_j\rangle | \geq \epsilon_1$, then by using again the same argument as above, we get a contradiction with $\lim_{j\to \infty} \langle \lambda_{\Gamma}(h_j) .w_j,w_j \rangle = 1$.

\end{proof}

Let  
$$\lambda_\omega: \Gamma_\omega\to O({H}_\omega)$$ be the 
ultralimit representation of the group $\Gamma_\omega$, defined in (\ref{ultralimit representation}). Here, since $\Gamma$ is assumed to be torsion-free, $\Gamma_\omega$ is also torsion-free. Let $\Omega_\omega,F_\omega$ be as in (\ref{Fomega}).
As we will see, due to the construction of $\Gamma_\omega$ and $\lambda_\omega$, Lemma \ref{degen of reg rep} has immediate implications for $\lambda_\omega$.


Let $\Upsilon$ be the set of maximal abelian subgroups of $\Gamma_\omega$.  Given a subgroup $G$ of $\Gamma_\omega$, let 
$$\Upsilon_0(G):=\{\sigma_z\}_{z\in \mathcal{Z}}$$
be a collection of elements in $\Upsilon$ such that for every $\sigma\in \Upsilon$, there is exactly one $\sigma_z \in \Upsilon_0(G)$ such that  $$\{g^{-1}\sigma_z g\}_{g\in G}=\{g^{-1}\sigma g\}_{g\in G}.$$ 
Given $\sigma\in \Upsilon_0(G)$, set 
$$\mathcal{T}(G,\sigma) :=\text{a left transversal of $\sigma\cap G$ for $G$},$$
in other words $\mathcal{T}(G,\sigma)$ is a subset of $G$ such that $G$ is the disjoint union $\bigcup_{t\in \mathcal{T}(G,\sigma)} t.(\sigma\cap G)$.
The notions of proper free orthogonal representation and induced representation are defined in Subsection \ref{prelim3} in the Appendix. The direct sum of Hilbert spaces, orthogonal representations, and the symbol $\bigoplus$ are defined in \cite[Section A.1, Definition A.1.6]{BDLHV08}.

\begin{coro} \label{ultralimit regular representation}
Consider any subgroup $G$ of $\Gamma_\omega$ and any closed subspace $H_G$ of $H_\omega$ invariant by $G$.
 There are, for every $\sigma\in \Upsilon_0(G)$,
 an orthogonal representation $(\mathcal{K}_\sigma,\eta_\sigma)$ of the group $\sigma \cap G$, and a proper free orthogonal representation $(H_{\mathrm{pf}}, \eta_{\mathrm{pf}})$ of $G$ such that
 if $(H_\sigma, \Induced_{\sigma\cap G}^{G}(\eta_\sigma))$ denotes the induced representation, 
$$\lambda_\omega |_{G,H_G} = \eta_{\mathrm{pf}} \oplus \bigoplus_{\sigma \in \Upsilon_0(G)} \Induced_{\sigma\cap G}^{G}(\eta_\sigma),$$
$$H_G = H_{\mathrm{pf}} \oplus \bigoplus_{\sigma \in \Upsilon_0(G)}  H_\sigma = H_{\mathrm{pf}} \oplus \bigoplus_{\sigma \in \Upsilon_0(G)} \bigoplus_{\gamma \in \mathcal{T}(G,\sigma)} \lambda_\omega(\gamma).\mathcal{K}_{\sigma}.$$
Moreover, if $S_G$ denotes the unit sphere of $H_G$, we have
$$F_\omega \cap S_G \subset  \bigcup_{\sigma\in \Upsilon_0(G)} \bigcup_{\gamma\in \mathcal{T}(G,\sigma)} \lambda_\omega(\gamma).\mathcal{K}_\sigma.$$

\end{coro}
\begin{remarque}
In the decomposition of $H_G$, we allow each term $H_{\mathrm{pf}}$ and $H_\sigma$ ($\sigma\in \Upsilon_0(G)$) to be the trivial vector space $\{0\}$. For instance, when $H_G$ is separable, then only at most countably many $H_\sigma$ are nontrivial in the decomposition. 

\end{remarque}
\begin{proof}
Let $G$ be a subgroup of $\Gamma_\omega$ and $H_G\subset H_\omega$ a $G$-invariant closed linear subspace.

For each $\sigma\in \Upsilon$, let $\mathcal{K}_\sigma$ be the closed subspace of $H_G$ generated by vectors in $H_G$ which are fixed by a nontrivial element of the maximal abelian subgroup $\sigma$ of $\Gamma_\omega$. This subspace is invariant by the subgroup $\sigma\cap G$ of $G$. 

If two elements $\sigma_1,\sigma_2\in \Upsilon$ are different, then $\sigma_1\cap \sigma_2=\{1\}$ and a nontrivial element of $\sigma_1$ never commutes with a nontrivial element of $\sigma_2$, since similar facts holds for the torsion-free hyperbolic group $\Gamma$, and by construction of $\Gamma_\omega$ (see Lemma \ref{limit group}).
Thus by the second bullet in Lemma \ref{degen of reg rep} and by construction of $\Gamma_\omega,\lambda_\omega$, 
\begin{equation}\label{perpendiculaire1}
\mathcal{K}_{\sigma_1}  \perp \mathcal{K}_{\sigma_2}.
\end{equation}
If two different elements $\sigma_1,\sigma_2\in \Upsilon$ are conjugate in $\Gamma_\omega$ by an element of $G$, namely if for some $g\in G$,
$$g\sigma_1g^{-1}= \sigma_2$$
then 
\begin{equation}\label{sigma1sigma2}
\mathcal{K}_{\sigma_2} = \lambda_\omega(g).\mathcal{K}_{\sigma_1}
\end{equation} 
and vice-versa. 
Besides, note that for any $\sigma\in \Upsilon_0(G)$ and for any two distinct elements $\gamma_1,\gamma_2$ in the transversal  $\mathcal{T}(G,\sigma)$, 
$$\gamma_1\sigma\gamma_1^{-1}\neq \gamma_2\sigma\gamma_2^{-1},$$
by Lemma \ref{limit group} (3), which implies with (\ref{perpendiculaire1}), (\ref{sigma1sigma2}) that
\begin{equation}\label{perpendiculaire}
\lambda_\omega(\gamma_1).\mathcal{K}_{\sigma}  \perp \lambda_\omega(\gamma_2).\mathcal{K}_{\sigma}.
\end{equation}

Next, the closed subspace
$$H_\sigma :=\bigoplus_{\gamma \in \mathcal{T}(G,\sigma)} \lambda_\omega(\gamma).\mathcal{K}_{\sigma}$$ 
is invariant by $G$ and the restriction of the representation $\lambda_\omega |_{G,H_G}$ of $G$ to $H_\sigma$ is exactly the induced representation $\Induced_{\sigma}^{G}(\eta_\sigma)$
where we set 
$$\eta_\sigma := \lambda_\omega |_{\sigma\cap G, \mathcal{K}_\sigma}.$$

From (\ref{perpendiculaire1}) and (\ref{sigma1sigma2}), it follows that if $\sigma'\in \Upsilon$ is not contained in the conjugacy class $\{g^{-1}\sigma g\}_{g\in G}$, then 
$$\mathcal{K}_{\sigma'} \perp H_{\sigma}$$
and so if $H_{\sigma'}$ is defined analogously as above,
$$H_{\sigma'} \perp H_{\sigma}.$$

We deduce that indeed 
$$\lambda_\omega |_{G,H_G} = \eta_{\mathrm{pf}} \oplus \bigoplus_{\sigma \in \Upsilon_0(G)} \Induced_{\sigma\cap G}^{G}(\eta_\sigma),$$
$$H_G = H_{\mathrm{pf}} \oplus \bigoplus_{\sigma \in \Upsilon_0(G)}  H_\sigma = H_{\mathrm{pf}} \oplus \bigoplus_{\sigma \in \Upsilon_0(G)} \bigoplus_{\gamma \in \mathcal{T}(G,\sigma)} \lambda_\omega(\gamma).\mathcal{K}_{\sigma},$$
where $\eta_{\mathrm{pf}}$ is the restriction of $\lambda_\omega|_{G}$ to the closed subspace $H_{\mathrm{pf}}$  equal to the orthogonal of the closed subspace $\bigoplus_{\sigma\in \Upsilon_0(G)} H_{\sigma}$ inside $H_G$.  
Each term in the above decomposition of $H_G$ may be the trivial vector space $\{0\}$.

By definition of the subspaces $\{\mathcal{K}_\sigma\}_{\sigma \in \Upsilon}$, any nontrivial vector of $H_G$ fixed by a nontrivial element of $\Gamma_\omega$ belongs to 
$$ \bigcup_{\sigma\in \Upsilon_0(G)} \bigcup_{\gamma\in \mathcal{T}(G,\sigma)} \lambda_\omega(\gamma).\mathcal{K}_\sigma.$$
Thus, the last inclusion in the statement of the corollary holds.

From Subsection \ref{subsection:nota}, we know that $\Gamma_\omega$, and thus $G$, acts freely properly on the complement $\Omega_\omega$ of the set of fixed points $F_\omega\subset S^\infty_\omega$. Since $H_{\mathrm{pf}}$ belongs to this complement, $\eta_{\mathrm{pf}}$ is a proper free orthogonal representation of $G$, as desired.
\end{proof}

Given any subgroup $G'$ of $\Gamma_\omega$, let $\mathbf{1}_{G'}$ denote the trivial $1$-dimensional representation. Given a vector $\mathbf{u}$, let $\mathbb{R}\mathbf{u}$ be the $1$-dimensional vector space generated by $\mathbf{u}$.

\begin{lemme} \label{find induced}
Given a countable subgroup $G$ of $\Gamma_\omega$, a countable subset $\Upsilon'_0(G)$ of $\Upsilon_0(G)$, a separable $G$-invariant closed subspace $H^{(\alpha)}\subset H_\omega$, there exists a separable $G$-invariant closed invariant subspace $H^{(\beta)}$ orthogonal to $H^{(\alpha)}$ in $H_\omega$, with unit sphere $S_{H^{(\beta)}}$, such that 
for some orthogonal family $\{\mathbf{u}_\sigma\}_{\sigma \in \Upsilon'_0(G)} \subset S^\infty_\omega$ where $\mathbf{u}_\sigma$ is fixed by $\sigma\cap G$, we have
$$ \lambda_\omega|_{G,H^{(\beta)}} = \bigoplus_{\sigma\in \Upsilon'_0(G)} \Induced_{\sigma\cap G}^{G}(\mathbf{1}_{\sigma\cap G}),$$
$$H^{(\beta)} = \bigoplus_{\sigma\in \Upsilon'_0(G)}  \bigoplus_{\gamma \in \mathcal{T}(G,\sigma)} \lambda_\omega(\gamma).\mathbb{R}\mathbf{u}_\sigma,$$
$$F_\omega \cap S_{H^{(\beta)}} = S_{H^{(\beta)}} \cap \bigcup_{\sigma\in \Upsilon'_0(G)} \bigcup_{\gamma \in \mathcal{T}(G,\sigma)} \lambda_\omega(\gamma).\mathbb{R}\mathbf{u}_\sigma.$$
Moreover, for any distinct $\sigma\neq \sigma' \in \Upsilon'_0(G)$, and any $\gamma'\in \Gamma_\omega$, we have $\mathbf{u}_\sigma \perp \lambda_\omega(\gamma').\mathbf{u}_{\sigma'}.$
In particular, the projection of $\{\pm \mathbf{u}_\sigma\}_{\sigma \in \Upsilon'_0(G)}$ to $(S^\infty/\Gamma)_\omega=S^\infty_\omega/\lambda_\omega(\Gamma_\omega)$ is a discrete set.

\end{lemme}

\begin{proof}
Let $\{\sigma_0,\sigma_1,\sigma_2,...\}$ be a numbering of the countably many elements of $\Upsilon'_0(G)$. 
For each $k\geq0$, let $\{a_{k,l}\}_{l\geq1}$ be the countably many group elements of the subgroup $ \sigma_k\cap G$. In case 
$$\sigma_k\cap G =\{1\}$$
let us choose an arbitrary element $a_{k,0}\neq 1$ in $\sigma_k$, otherwise just set $a_{k,0}=1$.
By construction of $\Gamma_\omega$, for each $k,l\geq0$, $a_{k,l}$ can be identified with a sequence $\{g_{k,l,i}\}_i \subset \Gamma$. Since each $\sigma_k$ is abelian,  $\{g_{k,l,i}\}_i $ can be chosen so that for any $l_1,l_2$, $g_{k,l_1,i}$ and $g_{k,l_2,i}$ commute  for $i$ large (see Lemma \ref{limit group}), and so given $k$, for $i$ large enough, every $g_{k,l,i}$ is in a cyclic subgroup of $\Gamma$ generated by one element called $z_k\in \Gamma$ depending only on $k$.

For each $k\geq0$, consider the functions $f_{k,i}\in S^\infty\subset \ell^2(\Gamma)$ equal to $\frac{1}{\sqrt{N_{k,i}}}$ on $\{z_k^{K_i},...,z_k^{K_i+N_{k,i}}\}$ and $0$ elsewhere, for some integer $N_{k,i}$ large enough and some arbitrary integer $K_i$. 
Then for any $k,l\geq0$,
\begin{equation} \label{almost fixx}
\lim_{i\to \infty} \| \lambda_\Gamma(g_{k,l,i}). f_{k,i} - f_{k,i}\| =0.
\end{equation}
By imposing that $\lim_{i\to \infty}N_{k+1,i}/N_{k,i} =\infty$ for each $k\geq1$, we can ensure that for any sequence $h_i\in \Gamma$,  
$$\forall k_1\neq k_2,\quad \lim_{i\to \infty} \langle  \lambda_\Gamma(h_i). f_{k_1,i}, f_{k_2,i} \rangle =0.$$
Similarly, given a countable family $\mathcal{F}$ of sequences in $S^\infty$, by making $K_i$ converge to $\infty$ fast enough depending on $\mathcal{F}$ and a diagonal argument, then for any sequence $\{f'_i\}\in \mathcal{F}$, 
$$\lim_{i\to \infty} \langle  f_{k,i}, f'_i \rangle =0.$$

For each $k\geq 0$, let $\mathbf{u}_{\sigma_k}\in S^\infty_\omega$ be the vector in the ultralimit corresponding to $\{f_{k,i}\}_i$.
Translating the previous properties to the ultralimit and using that $H^{(\alpha)}$ is separable, 
we can choose $f_{k,i}$ so that for any distinct $\sigma\neq \sigma' \in \Upsilon'_0(G)$, and any $\gamma'\in \Gamma_\omega$, 
$$\mathbf{u}_\sigma \perp \lambda_\omega(\gamma').\mathbf{u}_{\sigma'}, \quad \mathbf{u}_{\sigma} \perp H^{(\alpha)}.$$
The second orthogonality property crucially relies on the separability of $H^{(\alpha)}$.
By (\ref{almost fixx}), for each $k$, the vector $\mathbf{u}_{\sigma_k}$ is fixed by all the elements of $\sigma_k\cap G$, and also by $a_{k,0} \in \sigma_k$. In particular, it is fixed by a nontrivial element of $\Gamma_\omega$, namely
\begin{equation}\label{ak0}
\mathbf{u}_{\sigma_k}\in F_\omega.
\end{equation}

Applying Lemma \ref{degen of reg rep}, we see that if $H_k$ is the Hilbert space associated with $\Induced_{\sigma_k\cap G}^{G}(\mathbf{1}_{\sigma_k\cap G})$, $(H_k,\Induced_{\sigma_k\cap G}^{G}(\mathbf{1}_{\sigma_k\cap G}))$ is equivalent to a subrepresentation of $(H_\omega,\lambda_\omega|_G)$ that we call $(H_k^{(1)},\lambda_\omega|_{G,H_k^{(1)}})$, and
$$ \lambda_\omega|_{G,H_k^{(1)}} = \Induced_{\sigma_k\cap G}^{G}(\mathbf{1}_{\sigma_k\cap G}),$$
$$H_k^{(1)} =\bigoplus_{\gamma \in \mathcal{T}(G,\sigma_k)} \lambda_\omega(\gamma).\mathbb{R}\mathbf{u}_{\sigma_k}.$$
Since $\mathbf{u}_{\sigma_k}\perp H^{(\alpha)}$ and $ H^{(\alpha)}$ is $G$-invariant, $$H_k^{(1)}\perp H^{(\alpha)}.$$ 
By Lemma \ref{degen of reg rep} (2), for two diffferent $k_1,k_2\geq 0$, 
$$H_{k_1}^{(1)}\perp H_{k_2}^{(1)}.$$
Putting these facts together, we prove the first two equalities of the statement by setting $H^{(\beta)}:=\bigoplus_{k\geq 0} H_{k}^{(1)}$. 

To check the equality for $F_\omega\cap S_{H^{(\beta)}}$, one inclusion follows from (\ref{ak0}), and the other inclusion 
follows by using Lemma \ref{degen of reg rep} (2) as in the proof of Corollary \ref{ultralimit regular representation}.
\end{proof}

\begin{remarque}\label{compatibility}
The decomposition of Corollary \ref{ultralimit regular representation} is compatible with Lemma \ref{find induced} in the following sense. If $G$ is a countable subgroup of $\Gamma_\omega$ and  $H^{(\alpha)}=H_G$ a separable $G$-invariant closed subspace of $H_\omega$, we can apply first  Corollary \ref{ultralimit regular representation} to get a decomposition 
$$H^{(\alpha)} =  H_{\mathrm{pf}} \oplus \bigoplus_{\sigma \in \Upsilon'_0(G)} \bigoplus_{\gamma \in \mathcal{T}(G,\sigma)} \lambda_\omega(\gamma).\mathcal{K}_{\sigma}$$
where $\Upsilon'_0(G)$ is the countable set of $\sigma\in \Upsilon_0(G)$ such that $\mathcal{K}_\sigma \neq \{0\}$. Then we can apply  Lemma \ref{find induced} to get an orthogonal subspace
$$H^{(\beta)}= \bigoplus_{\sigma\in \Upsilon'_0(G)}  \bigoplus_{\gamma \in \mathcal{T}(G,\sigma)} \lambda_\omega(\gamma).\mathbb{R}\mathbf{u}_\sigma.$$
Now by our construction, the decomposition of Corollary \ref{ultralimit regular representation} for $H^{(\alpha)} \oplus H^{(\beta)}$ would give
$$H^{(\alpha)} \oplus H^{(\beta)} = H_{\mathrm{pf}} \oplus \bigoplus_{\sigma \in \Upsilon'_0(G)} \bigoplus_{\gamma \in \mathcal{T}(G,\sigma)} \lambda_\omega(\gamma).(\mathcal{K}_{\sigma} \oplus \mathbb{R}\mathbf{u}_\sigma),$$
and moreover if $S_{H^{(\alpha)} \oplus H^{(\beta)}}$ denotes the unit sphere of $H^{(\alpha)} \oplus H^{(\beta)}$,
$$F_\omega \cap S_{H^{(\alpha)} \oplus H^{(\beta)}} \subset S_{H^{(\alpha)} \oplus H^{(\beta)}} \cap  \bigcup_{\sigma\in \Upsilon'_0(G)} \bigcup_{\gamma\in \mathcal{T}(G,\sigma)} \lambda_\omega(\gamma).(\mathcal{K}_\sigma \oplus \mathbb{R}\mathbf{u}_{\sigma}). $$
\end{remarque}

\subsection{A Margulis type lemma}
The classical Margulis lemma states that thin parts of closed hyperbolic manifolds have cyclic fundamental groups. Below, we show a similar statement for $S^\infty/\Gamma$. 

In general, if $\Gamma$ is a finitely generated group, let $\lambda_\Gamma$ denote the regular representation of $\Gamma$ in $\ell^2(\Gamma)$ and let $S$ be a finite subset of $\Gamma$.  The corresponding \emph{Kazhdan constant} 
is defined (\cite[Section 2]{ABLRSV05}) by
\begin{equation} \label{definitionK}
\mathbf{K}(\lambda_\Gamma, \Gamma,S):= \inf_{u \in \ell^2(\Gamma) \backslash \{0\}} \max_{s\in S} \frac{\| \lambda_\Gamma(s).u - u \|_{\ell^2}}{\| u \|_{\ell^2}}.
\end{equation}
If $S$ is a finite symmetric generating set for $\Gamma$, the Kazhdan constant is directly related to the first eigenvalue of the combinatorial Laplacian $\Delta$ on the Cayley graph of $(\Gamma,S)$, see \cite[(2)]{Valette94}. 
Given  $x\in \Gamma$, we will use the standard notation $ \langle x \rangle$ for the cyclic subgroup generated by $x$ in $\Gamma$.

\begin{lemme}\label{kazhdan lower bound}
Let $\Gamma$ be a non-elementary torsion-free hyperbolic group.
Then there exists $\alpha=\alpha(\Gamma)>0$ such that for any $x\in \Gamma \setminus \{e\}$ and $y\notin \langle x \rangle$, we have
$$\mathbf{K}(\lambda_{\Gamma}, \Gamma, \{x,y\})\geq \alpha.$$
\end{lemme}

\begin{proof}
Consider sequences  $x_i\in \Gamma\setminus \{1\}$ and $y_i\in \Gamma$ with $y_i\notin \langle x_i\rangle$ for any $i$, so that $x_i$ and $y_i$ do not commute.
If we have
$$\lim_{i\to \infty} \mathbf{K}(\lambda_{\Gamma}, \Gamma, \{x_i,y_i\}) = 0,$$
it means that for a sequence of functions $u_i\in \ell^2(\Gamma)$ of $\ell^2$-norm $1$,  
$$\lim_{i\to \infty} \|\lambda_{\Gamma}(x_i).u_i -u_i\| =  \lim_{i\to \infty} \|\lambda_{\Gamma}(y_i).u_i -u_i\| =0.$$
That in turn is equivalent to
$$\lim_{i\to \infty} \langle \lambda_{\Gamma}(x_i).u_i ,u_i\rangle  = \lim_{i\to \infty} \langle \lambda_{\Gamma}(y_i).u_i ,u_i\rangle =1,$$
which of course contradicts Lemma \ref{degen of reg rep} (1).
\end{proof}

Given a path connected subset $R$ of $S^\infty/\Gamma$ and a base point $r_0\in R$, denote by $\iota:R\to S^\infty/\Gamma$ the inclusion map, and $\iota_*(\pi_1(R,r_0))$ the image fundamental group, which is a subgroup of $\pi_1(S^\infty/\Gamma,\iota(r_0)) \approx \Gamma$. This last identification is well-defined and unique only up to conjugacy, but in what follows it does not matter which one we pick, so for each $(R,r_0)$ we fix one identification.
The key result in this subsection is the following Margulis type lemma for $S^\infty/\Gamma$. Without assumptions on $\Gamma$, the conclusion fails to hold; examples can be constructed using \cite{Osin02,ABLRSV05}.
\begin{coro} \label{margulis}
Let $\Gamma$ be a given non-elementary torsion-free hyperbolic group. There exists $\overline{\delta}=\overline{\delta}(\Gamma) >0$ such that the following holds. 
For any path connected component $R$ of $[S^\infty/\Gamma]^{\leq \overline{\delta}}$ and base point $r_0\in R$, 
$\iota_*(\pi_1(R),r_0)$ is a cyclic subgroup of $\Gamma$.
\end{coro}
\begin{proof}
We will fix $\overline{\delta}$ small enough depending on $\alpha = \alpha(\Gamma)$ given by Lemma \ref{kazhdan lower bound}. 
Let $R\subset [S^\infty/\Gamma]^{\leq \overline{\delta}}$ be as in the statement and fix a reference point $r_0\in R$. Let $\sigma:[0,1] \to R$ be a continuous loop in $R$ so that $\sigma(0)=\sigma(1) = r_0$. The map $\sigma$ lifts to a map 
$$\tilde{\sigma}:[0,1] \to S^\infty.$$ 
Each element in the  image of $\tilde{\sigma}$ is a real $\ell^2$-function on $\Gamma$.
Since $R$ is contained inside $[S^\infty/\Gamma]^{\leq \overline{\delta}}$, $\tilde{\sigma}$ has the property that for any $t\in [0,1]$, there is $\gamma_{t} \in \Gamma \setminus \{1\}$ such that 
\begin{equation} \label{alpha sur 2}
\|\lambda_\Gamma(\gamma_t).\tilde{\sigma}(t)- \tilde{\sigma}(t) \|_{\ell^2} < \alpha/2
\end{equation}
provided $ \overline{\delta}$ was chosen small enough.
We choose a fine collection of times $t_0=0, t_1,..., t_J=1$ such that for all $j\in \{0,...,J-1\}$, 
$$\|\lambda_\Gamma(\gamma_{t_j}).\tilde{\sigma}(t_j)- \tilde{\sigma}(t_j) \|_{\ell^2} < \alpha \quad  \text{and} \quad \|\lambda_\Gamma(\gamma_{t_{j+1}}).\tilde{\sigma}(t_j)- \tilde{\sigma}(t_j) \|_{\ell^2} < \alpha,$$
which means in terms of Kazhdan constants that
$$\mathbf{K}(\lambda_\Gamma,  \Gamma,\{\gamma_{t_j},\gamma_{t_{j+1}}\}) < \alpha.$$
By Lemma \ref{kazhdan lower bound}, for any $j\in \{0,...,J-1\}$, the subgroup generated by $\gamma_{t_j}$ and $\gamma_{t_{j+1}}$ is a nontrivial cyclic subgroup of the torsion-free hyperbolic group $\Gamma$.
By induction we conclude that the subgroup 
$H$ generated by all the $\gamma_{t_j}$ is a nontrivial cyclic subgroup of $\Gamma$.

Next let $g$ be the element of $\Gamma$ such that $\lambda_\Gamma(g). \sigma'(0) = \sigma'(1)$. By (\ref{alpha sur 2}) applied at $t=1$, $g$ satisfies 
$$\| \lambda_\Gamma(g^{-1}\gamma_{t_J} g) . \tilde{\sigma}(0) - \tilde{\sigma}(0)\|_{\ell^2} < \alpha/2.$$
Combining that with (\ref{alpha sur 2}) applied at $t=0$, we deduce from Lemma \ref{kazhdan lower bound} that
$g^{-1}\gamma_{t_J} g \in H.$
But $H$ is malnormal inside the hyperbolic group $\Gamma$, therefore, $g\in H$. By applying this argument to any loop $\sigma \subset R$ based at $r_0$, we conclude that  
$\iota_*(\pi_1(R))$
is indeed a subgroup of the cyclic subgroup $H$
and the proof is completed.

\end{proof}


\subsection{The non-vanishing theorem}

We apply the Margulis type lemma Corollary \ref{margulis} in this section to deduce non-vanishing properties of the spherical Plateau problem for hyperbolic groups. 


\begin{lemme}\label{nonempty}
Let $\Gamma$ be a torsion-free hyperbolic group and suppose that  $h\in H_n(\Gamma;\mathbb{Z})$ is a nonzero homology class, for some $n\geq 2$.
Then there is $\overline{\delta}>0$ depending only on $\Gamma$ such that for any $C\in \mathscr{C}(h)$,
$$\spt C\cap [S^\infty/\Gamma]^{> \overline{\delta}} \neq \varnothing.$$
\end{lemme}
\begin{proof}
Let $\overline{\delta}$ be given by Corollary \ref{margulis}.
Consider first a polyhedral chain $P$ which is an element of $\mathscr{C}(h)$, and suppose towards a contradiction that 
\begin{equation} \label{contradictiona}
\spt P \subset [S^\infty/\Gamma]^{\leq \overline{\delta}}.
\end{equation}
If $\iota:\spt P \to S^\infty/\Gamma$ denotes the inclusion map, Corollary \ref{margulis} implies that $\iota_*(\pi_1(\spt P))$ is cyclic. By basic topology, this means that $\spt P$ can be smoothly deformed via a homotopy of the $K(\pi,1)$-space $S^\infty/\Gamma$ inside a $1$-dimensional loop of  $S^\infty/\Gamma$. From this, we obtain that $\spt P$ viewed as a topological cycle representing $h$, is homologous to the trivial homology class  of $S^\infty/\Gamma$, a contradiction with our assumption that $h\neq 0$.
Hence for any polyhedral chain $P \in \mathscr{C}(h)$, we must have
$$\spt P \cap [S^\infty/\Gamma]^{> \overline{\delta}} \neq \varnothing.$$
Integral currents with compact supports in $S^\infty/\Gamma$ can be well approximated by polyhedral chains in the Hausdorff topology as explained in \cite[Lemma 1.6]{Antoine23a}, thus for any $C\in \mathscr{C}(h)$,
$$\spt C\cap [S^\infty/\Gamma]^{> \overline{\delta}} \neq \varnothing$$
(we may need to take $\overline{\delta}$ slightly smaller). This concludes the proof.

\end{proof}

A corollary of Lemma \ref{nonempty} is a strong non-vanishing theorem for the spherical volume, which generalizes the previously known theorem that closed oriented negatively curved manifolds have strictly positive spherical volume (a result which follows from combining \cite[Subsection 0.3]{Gromov82} and \cite[Th\'{e}or\`{e}me 3.16]{BCG91}).
\begin{coro} \label{positive spherevol}
Let $\Gamma$ be a torsion-free hyperbolic group and suppose that  $h\in H_n(\Gamma;\mathbb{Z})$ is a nonzero homology class, for some $n\geq 2$.
Then 
$$\spherevol(h)>0.$$
Furthermore, any spherical Plateau solution $C_\infty$ for $h$ has a non-empty noncollapsed part $C^{>0}_\infty$.

\end{coro}

\begin{proof}
Let $\overline{\delta}$ be given by Lemma \ref{margulis}. Recall from the definitions that, in general, for any group homology class $h$ and a corresponding spherical Plateau solution $C_\infty$, 
$$\spherevol(h)\geq\mathbf{M}(C_\infty),$$
and for any noncollapsed part $C_\infty^{>0}$ of $C_\infty$,
$$\mathbf{M}(C_\infty)\geq\mathbf{M}(C_\infty^{>0}).$$
In particular, in order to prove the first part of the corollary on the strict positivity of the spherical volume of $h$, it is enough to show that any noncollapsed part $C_\infty^{>0}$ of $C_\infty$ is nontrivial.

Applying Theorem \ref{compactness}, we find a pulled-tight minimizing sequence $\{C_i\}\subset  \mathscr{C}(h)$, and by Lemma \ref{nonempty}, the supports $\spt(C_i)$ intersect $[S^\infty/\Gamma]^{> \overline{\delta}}$ nontrivially. In the notations of Subsection \ref{thickdefinition}, $Y_\infty$ is non-empty and so the noncollapsed part $C_\infty^{>0}$ is nontrivial. This finishes the proof.

\end{proof}


\subsection{Subspaces and subrepresentations for a separable subset}
\label{subsubsection:two spheres}

This section paves the way for the definition of the ``deformations maps'', with some defintions and notations. 
As before, let $\Gamma$ be a torsion-free hyperbolic group, and let $h\in H_n(\Gamma;\mathbb{Z})\setminus \{0\}$. Let $(S^\infty/\Gamma)_\omega = S^\infty_\omega/\lambda_\omega(\Gamma_\omega)$ be as in Subsection \ref{subsection:nota}.
 Here is an elementary lemma:


 \begin{lemme} \label{separable embedding}
Let $\mathbf{A}'\subset  (S^\infty/\Gamma)_\omega$ be a separable subset.
There is a countable subgroup $\Gamma_{(1)}$ of $\Gamma_\omega$, and there is a separable $\Gamma_{(1)}$-invariant closed subspace $H_{(1)}$ of  $H_\omega$
 such that if $S^\infty_{(1)}$ denotes the unit sphere in $H_{(1)}$,  then we have two isometric embeddings
 $$(\mathbf{A}',\dist_\omega)\hookrightarrow S^\infty_{(1)}/\lambda_\omega(\Gamma_{(1)}) \hookrightarrow (S^\infty/\Gamma)_\omega$$
 where $\dist_\omega$ is the metric on $(S^\infty/\Gamma)_\omega$, $S^\infty_{(1)}/\lambda_\omega(\Gamma_{(1)})$ is endowed with its quotient metric, and the second embedding is the natural map induced by the inclusion $S^\infty_{(1)}\subset S^\infty_\omega$. 
 The subgroup $\Gamma_{(1)}$ can be chosen to contain $\Gamma$.
 
\end{lemme}
\begin{proof}
Consider the quotient map 
$$\Pi_\omega: S^\infty_\omega \to (S^\infty/\Gamma)_\omega.$$ 
Choose any countable subset $Z_0$ in the preimage $\Pi_\omega^{-1}(\mathbf{A}')$ such that the closure of $\Pi_\omega(Z_0)$ satisfies
$$\overline{\Pi_\omega(Z_0)}=\mathbf{A}'.$$
Let $H_1$ be the smallest closed separable Hilbert subspace of $H_\omega$ which contains $Z_0$. 

Let $\dist_\omega$ be the usual distance on the round sphere $S^\infty_\omega$.
In order to prove the lemma, all we need to do is to find a closed separable Hilbert subspace $H_{(1)}\subset H_\omega$ containing $H_1$, and a countable subgroup $\Gamma_{(1)}$ of $\Gamma_\omega$ which leaves $H_{(1)}$ invariant, such that if 
$S^\infty_{(1)}$ denotes the unit sphere of $H_{(1)}$,
the following holds: for any two points $x,y\in S^\infty_{(1)}$, 
\begin{equation}\label{dist lift}
\inf\{\dist_\omega(x,\lambda_\omega(g).y); g\in \Gamma_{(1)}\} = \inf\{\dist_\omega(x,\lambda_\omega(g).y); g\in \Gamma_\omega\}.
\end{equation}
The meaning of this equality is that the quotient space $S^\infty_{(1)}/\lambda_\omega(\Gamma_{(1)})$, with its quotient metric, embeds isometrically into $(S^\infty/\Gamma)_\omega$. Since in that case, $\overline{\Pi_\omega(Z_0)}=\mathbf{A}'$ embeds isometrically inside  $S^\infty_{(1)}/\lambda_\omega(\Gamma_{(1)})$, the lemma would be checked.

In order to find such $H_{(1)}$ and $\Gamma_{(1)}$, we can argue as follows. Let $S_1$ be the unit sphere of $H_1$. 
By separability of $H_1$, there is a countable subset $Q_2\subset \Gamma_\omega$ such that for 
any two points $x,y\in S_1$, 
$$
\inf\{\dist_\omega(x,\lambda_\omega(g).y); g\in Q_2\} = \inf\{\dist_\omega(x,\lambda_\omega(g).y); g\in \Gamma_\omega\}.
$$
Note that we can enlarge $Q_2$ to contain any countable subset of $\Gamma_\omega$, in particular we can assume that $Q_2$ contains $\Gamma$. Set $\Gamma_2$ to be the countable subgroup of $\Gamma_\omega$ generated by $Q_2$, and let $H_2$ be the smallest closed separable Hilbert subspace of $H_\omega$ containing $H_1$ and invariant by $\Gamma_2$. Let $S_2$ be the unit sphere in $H_2$. We repeat the above process and get a sequence of separable Hilbert subspaces $H_1\subset H_2\subset H_3...$ and corresponding countable subgroups $\Gamma_2\subset \Gamma_3...$. Now take the closure $H_{(1)}$ of the union of the subspaces $\{H_j\}_{j\geq 1}$, and consider the union $\Gamma_{(1)}$ of the subgroups $\{\Gamma_j\}_{j\geq 2}$. The subspace $H_{(1)}$ is closed separable, the subgroup $\Gamma_{(1)}$ is countable and leaves $H_{(1)}$ invariant. Moreover, any two points $x,y\in S^\infty_{(1)}$ can be approximated by points $x_j,y_j$ belonging to the unit sphere $S_j$ of some $H_j$, and so (\ref{dist lift}) is indeed true by construction.

\end{proof}

Let $\Omega_\omega\subset S^\infty_\omega$, $F_\omega\subset S^\infty_\omega$ and $\Sing (S^\infty_\omega/\Gamma)$ be as in Subsection \ref{subsection:nota}.
Given a separable subset 
$$\mathbf{A}\subset  (S^\infty/\Gamma)_\omega,$$ 
let $\mathbf{B}\subset (S^\infty/\Gamma)_\omega$ be another separable subset such that
\begin{equation}\label{mathbfB}
\begin{split}
&\mathbf{B}\subset  \Sing  (S^\infty/\Gamma)_\omega=F_\omega/\lambda_\omega(\Gamma_\omega)\\& \text{and for any $a\in \mathbf{A}$,}\quad \dist_\omega(a,\Sing  (S^\infty/\Gamma)_\omega) = \dist_\omega(a,\mathbf{B}).
\end{split}
\end{equation}
Such $\mathbf{B}$ always exists. Set $\mathbf{A}':=\mathbf{A}\cup \mathbf{B}$.
Let $\Gamma_{(1)}$, $H_{(1)}$, 
be as in Lemma \ref{separable embedding} applied to $\mathbf{A}'$, and let $S^\infty_{(1)}$ be the unit sphere of $H_{(1)}$. In particular $\Gamma\subset \Gamma_{(1)}$. Recall from the elementary Lemma \ref{limit group} (2) that since $\Gamma$ is torsion-free, so are $\Gamma_\omega$ and $\Gamma_{(1)}$. 

By Lemma \ref{ultralimit regular representation}, there are an orthogonal representation $(\mathcal{K}_\sigma,\eta_\sigma)$ of $\sigma \cap\Gamma_{(1)}$  for every $\sigma\in \Upsilon_0(\Gamma_{(1)})$, and a proper free orthogonal representation $(H_{\mathrm{pf}}, \eta_{\mathrm{pf}})$ of $\Gamma_{(1)}$ such that
 \begin{equation}\label{h_(1)}
\begin{split}
\lambda_\omega|_{\Gamma_{(1)},H_{(1)}} & = \eta_{\mathrm{pf}} \oplus \bigoplus_{\sigma \in \Upsilon_0(\Gamma_{(1)})} \Induced_{\sigma \cap\Gamma_{(1)}}^{\Gamma_{(1)}}(\eta_\sigma),\\
H_{(1)} & = H_{\mathrm{pf}} \oplus \bigoplus_{\sigma \in \Upsilon_0(\Gamma_{(1)})} H_\sigma = H_{\mathrm{pf}} \oplus \bigoplus_{\sigma \in \Upsilon_0(\Gamma_{(1)})} \bigoplus_{\gamma\in \mathcal{T}(\Gamma_{(1)},\sigma)} \lambda_\omega(\gamma).\mathcal{K}_{\sigma}.
\end{split}
\end{equation}
Set also 
$$F_{(1)}:= F_\omega \cap S^\infty_{(1)}.$$
Recall that $\mathbf{B}$ is the subset introduced in (\ref{mathbfB}).
Let $\pi: S^\infty_{(1)}\to S^\infty_{(1)}/\lambda_\omega(\Gamma_{(1)})$ be the natural projection. 
By Lemma \ref{separable embedding} and the inclusion in Corollary \ref{ultralimit regular representation}, the following holds:
\begin{equation}\label{pi-1b}
\pi^{-1}(\mathbf{B}) \subset F_{(1)} \subset S^\infty_{(1)} \cap \bigcup_{\sigma\in \Upsilon_0(\Gamma_{(1)})} \bigcup_{\gamma\in \mathcal{T}(\Gamma_{(1)},\sigma)} \lambda_\omega(\gamma).\mathcal{K}_\sigma.
\end{equation}
Next, set 
\begin{align*}
\Upsilon'_0(\Gamma_{(1)}) := \{& \text{the countable set of $\sigma \in \Upsilon_0(\Gamma_{(1)})$ such that}\\
& \text{ $H_\sigma\neq \{0\}$ in the decomposition (\ref{h_(1)}) of $H_{(1)}$}\}.
\end{align*}
By Lemma \ref{find induced} applied to $H^{(\alpha)}:=H_{(1)}$, we can find some unit vectors $\{\mathbf{u}_\sigma\}_{\sigma\in \Upsilon'_0(\Gamma_{(1)})}$ and a $\Gamma_{(1)}$-invariant closed subspace $H^{(\beta)}$ orthogonal to $H_{(1)}$ in $H_\omega$, such that 
\begin{equation}\label{gamma'h(1)}
\begin{split}
\lambda_\omega |_{\Gamma_{(1)},H^{(\beta)}} & =  \bigoplus_{\sigma\in \Upsilon'_0(\Gamma_{(1)})} \Induced_{\sigma \cap\Gamma_{(1)}}^{\Gamma_{(1)}}(\mathbf{1}_{\sigma\cap \Gamma_{(1)}}),\\
H^{(\beta)} & = \bigoplus_{\sigma\in \Upsilon'_0(\Gamma_{(1)})} \bigoplus_{\gamma\in \mathcal{T}(\Gamma_{(1)},\sigma)} \lambda_\omega(\gamma). \mathbb{R}\mathbf{u}_\sigma
\end{split}
\end{equation}
and $\mathbf{u}_\sigma$ is fixed by $\sigma\cap \Gamma_{(1)}$.
Moreover, Lemma \ref{find induced} ensures that for any distinct  $\sigma\neq\sigma'\in \Upsilon'_0(\Gamma_{(1)})$ and any $\gamma'\in \Gamma_\omega$,
\begin{equation}\label{u perp u'}
\mathbf{u}_\sigma \perp \lambda_\omega(\gamma').\mathbf{u}_{\sigma'}.
\end{equation} 
Consider  the infinite-dimensional separable Hilbert space 
\begin{equation}\label{ploooouc2}
H_{(3)} := H_{(1)}  \oplus H^{(\beta)}
\end{equation}
and let $S^\infty_{(3)}$ be the unit sphere in $H_{(3)}$ and set $F_{(3)}:= F_\omega \cap S^\infty_{(3)}$. By Remark \ref{compatibility}, we have
$$F_{(3)} \subset S^\infty_{(3)} \cap \bigcup_{\sigma\in \Upsilon'_0(\Gamma_{(1)})} \bigcup_{\gamma\in \mathcal{T}(\Gamma_{(1)},\sigma)} \lambda_\omega(\gamma).(\mathcal{K}_\sigma \oplus \mathbb{R}\mathbf{u}_{\sigma}).$$
We also set
\begin{equation}\label{ploooouc}
H_{(2)}:=H_{\mathrm{pf}}\oplus H^{(\beta)},
\end{equation}
where $H_{\mathrm{pf}}$ appears in (\ref{h_(1)}). 
Let $S^\infty_{(2)}$ be the unit sphere in $H_{(2)}$ and set $F_{(2)}:= F_\omega \cap S^\infty_{(2)}$. Then  by Corollary \ref{ultralimit regular representation}, Remark \ref{compatibility}, and Lemma \ref{find induced}:
\begin{equation} \label{F_{(2)}}
F_{(2)} = \{\pm \lambda_\omega(\gamma).\mathbf{u}_{\sigma}\}_{\sigma\in \Upsilon'_0(\Gamma_{(1)}), \gamma\in \mathcal{T}(\Gamma_{(1)},\sigma)} = S^\infty_{(2)} \cap \bigcup_{\sigma\in \Upsilon'_0(\Gamma_{(1)})} \bigcup_{\gamma\in \mathcal{T}(\Gamma_{(1)},\sigma)} \lambda_\omega(\gamma).\mathbb{R}\mathbf{u}_{\sigma}.
\end{equation}
That is the key property of the orthogonal representation $(H_{(2)},\lambda_\omega|_{\Gamma_{(1)},H_{(2)}})$: the set of vectors in $S^\infty_{(2)}$ fixed by at least one nontrivial element of $\Gamma_\omega$ is the discrete countable set of points $F_{(2)}$.
Adding (\ref{u perp u'}), we deduce the following fact: the Hilbert spherical quotient  $S^\infty_{(2)}/\lambda_\omega(\Gamma_{(1)})$ 
is smooth outside the discrete countable set of points 
$$F_{(2)}/\lambda_\omega(\Gamma_{(1)}).$$
Around each such point, locally the spherical quotient looks like a cone over a smooth Hilbert spherical quotient. Note that a priori, a point of $F_{(2)}/\lambda_\omega(\Gamma_{(1)})$ is not necessarily a singular point inside $S^\infty_{(2)}/\lambda_\omega(\Gamma_{(1)})$, but its image in $(S^\infty/\Gamma)_\omega$ is always a singular point there.

\vspace{1em}
\textbf{Summary of notations:}

At this point, given a separable subset $\mathbf{A}\subset (S^\infty/\Gamma)_\omega$, we have a countable group $\Gamma_{(1)}\subset \Gamma_\omega$, and in (\ref{h_(1)}), (\ref{gamma'h(1)}), (\ref{ploooouc}), (\ref{ploooouc2}), we defined three separable Hilbert spaces:
\begin{equation} \label{decomp Hsigma}
\begin{split}
H_{(1)} & = H_{\mathrm{pf}} \oplus \bigoplus_{\sigma \in \Upsilon'_0(\Gamma_{(1)})}  \bigoplus_{\gamma\in \mathcal{T}(\Gamma_{(1)},\sigma)} \lambda_\omega(\gamma).\mathcal{K}_{\sigma},\\
H_{(2)} & = H_{\mathrm{pf}} \oplus \bigoplus_{\sigma\in \Upsilon'_0(\Gamma_{(1)})} \bigoplus_{\gamma\in \mathcal{T}(\Gamma_{(1)},\sigma)} \lambda_\omega(\gamma). \mathbb{R}\mathbf{u}_\sigma,\\
H_{(3)} & = H_{\mathrm{pf}} \oplus \bigoplus_{\sigma \in \Upsilon'_0(\Gamma_{(1)})}  \bigoplus_{\gamma\in \mathcal{T}(\Gamma_{(1)},\sigma)} \lambda_\omega(\gamma).(\mathcal{K}_{\sigma}\oplus \mathbb{R} \mathbf{u}_\sigma).
\end{split}
\end{equation}
where $\mathcal{K}_\sigma$ is the Hilbert space associated to the representation $\eta_\sigma$.
We also considered the corresponding orthogonal representations $\lambda_\omega|_{\Gamma_{(1)}, H_{(j)}}$ and unit spheres $S^\infty_{(j)}$ where $j\in \{1,2,3\}$.
By Lemma \ref{separable embedding}, we have isometric embeddings:
$$\mathbf{A}\subset S^\infty_{(1)}/\lambda_\omega(\Gamma_{(1)}) \to  (S^\infty/\Gamma)_\omega.$$
It is elementary to see that both natural maps 
$$S^\infty_{(1)}/\lambda_\omega(\Gamma_{(1)}) \to  S^\infty_{(3)}/\lambda_\omega(\Gamma_{(1)}) \quad \text{and}\quad  S^\infty_{(2)}/\lambda_\omega(\Gamma_{(1)}) \to S^\infty_{(3)}/\lambda_\omega(\Gamma_{(1)})$$ are isometric embeddings. 
We also define for $j\in \{1,2,3\}$:
$$F_j := F_\omega\cap S^\infty_{(j)},$$
namely the set of points in $S^\infty_{(j)}$ fixed by some nontrivial element of $\Gamma_\omega$. We saw that:
\begin{equation}\label{F_j subset}
\begin{split}
F_{(1)} &\subset S^\infty_{(1)} \cap \bigcup_{\sigma\in \Upsilon'_0(\Gamma_{(1)})} \bigcup_{\gamma\in \mathcal{T}(\Gamma_{(1)},\sigma)} \lambda_\omega(\gamma).\mathcal{K}_\sigma,\\
F_{(2)}&=  S^\infty_{(2)} \cap \bigcup_{\sigma\in \Upsilon'_0(\Gamma_{(1)})} \bigcup_{\gamma\in \mathcal{T}(\Gamma_{(1)},\sigma)} \lambda_\omega(\gamma).\mathbb{R}\mathbf{u}_{\sigma},\\
F_{(3)} &\subset S^\infty_{(3)} \cap \bigcup_{\sigma\in \Upsilon'_0(\Gamma_{(1)})} \bigcup_{\gamma\in \mathcal{T}(\Gamma_{(1)},\sigma)} \lambda_\omega(\gamma).(\mathcal{K}_\sigma \oplus \mathbb{R}\mathbf{u}_{\sigma}).
\end{split}
\end{equation}

\subsection{The $1$-Lipschitz deformation maps} \label{subsection:projection}

We can now define the deformation maps, which will be crucial later.
All the constructions depend on the separable subset $\mathbf{A}\subset (S^\infty/\Gamma)_\omega$.

\textbf{A first deformation map:}

We will use the notations from the previous subsection.
According to the first decomposition in (\ref{decomp Hsigma}), any vector of $S^\infty_{(1)}$ can be written as
$$v_{\mathrm{pf}} + \sum_{\sigma\in \Upsilon'_0(\Gamma_{(1)})} \sum_{\gamma\in \mathcal{T}(\Gamma_{(1)},\sigma)} v_{\gamma,\sigma}$$
for some $v_{\mathrm{pf}} \in H_{\mathrm{pf}}$ and $v_{\gamma,\sigma} \in \lambda_\omega(\gamma).\mathcal{K}_{\sigma}$ such that
$$\|v_{\mathrm{pf}}\|^2 + \sum_{\sigma\in \Upsilon'_0(\Gamma_{(1)})}\sum_{\gamma\in \mathcal{T}(\Gamma_{(1)},\sigma)} \|v_{\gamma,\sigma}\|^2=1.$$ 
Let 
$$\overline{\mathscr{D}}_1: S^\infty_{(1)}\to S^\infty_{(2)} \subset S^\infty_{(3)}$$
be the map defined by
\begin{align} \label{define overline proj1}
v_{\mathrm{pf}} + \sum_{\sigma\in \Upsilon'_0(\Gamma_{(1)})} \sum_{\gamma\in \mathcal{T}(\Gamma_{(1)},\sigma)} v_{\gamma,\sigma}
\quad \mapsto \quad  v_{\mathrm{pf}} + \sum_{\sigma\in \Upsilon'_0(\Gamma_{(1)})} \sum_{\gamma\in \mathcal{T}(\Gamma_{(1)},\sigma)} \|v_{\gamma,\sigma}\| \lambda_\omega(\gamma).\mathbf{u}_{\sigma}.
\end{align}
Recall from (\ref{F_{(2)}}) that $F_{(2)} = S^\infty_{(2)} \cap \bigcup_{\sigma\in \Upsilon'_0(\Gamma_{(1)})} \bigcup_{\gamma\in \mathcal{T}(\Gamma_{(1)},\sigma)} \lambda_\omega(\gamma).\mathbb{R} \mathbf{u}_\sigma$.
The first main lemma of this subsection is the following.
\begin{lemme} \label{projection main 1}
The map $\overline{\mathscr{D}}_1: S^\infty_{(1)}\to S^\infty_{(2)}$ is $\Gamma_{(1)}$-equivariant $1$-Lipschitz, and
\begin{equation}\label{proj_1-1}
\begin{split}
\overline{\mathscr{D}}_1\Big( S^\infty_{(1)}  \cap \bigcup_{\sigma\in \Upsilon'_0(\Gamma_{(1)})} \bigcup_{\gamma\in \mathcal{T}(\Gamma_{(1)},\sigma)} \lambda_\omega(\gamma).\mathcal{K}_\sigma  \Big)  & \subset   F_{(2)},\\
\overline{\mathscr{D}}_1\Big( S^\infty_{(1)}  \setminus \bigcup_{\sigma\in \Upsilon'_0(\Gamma_{(1)})} \bigcup_{\gamma\in \mathcal{T}(\Gamma_{(1)},\sigma)} \lambda_\omega(\gamma).\mathcal{K}_\sigma\Big)
& \subset  S^\infty_{(2)} \setminus F_{(2)}.
\end{split}
\end{equation}
Moreover, the quotient map
$\overline{\mathscr{D}}_1: S^\infty_{(1)}/\lambda_\omega(\Gamma_{(1)})\to S^\infty_{(2)}/\lambda_\omega(\Gamma_{(1)})$ satisfies the following:
for each $d> 0$, if $q\in \mathbf{A}$ and $\dist_\omega(q,\Sing (S^\infty/\Gamma)_\omega) < d$, then 
$$\overline{\mathscr{D}}_1(q)\text{ belongs to the $d$-neighborhood of  $F_{(2)}/\lambda_\omega(\Gamma_{(1)})$ in $S^\infty_{(2)}/\lambda_\omega(\Gamma_{(1)})$}.$$

\end{lemme}
\begin{proof}
The map is clearly $\Gamma_{(1)}$-equivariant. The vectors 
$$\{\lambda_\omega(\gamma).\mathbf{u}_{\sigma}\}_{\sigma\in \Upsilon'_0(\Gamma_{(1)}), \gamma\in \mathcal{T}(\Gamma_{(1)},\sigma)}$$ are of norm $1$ and pairwise orthogonal. From Cauchy-Schwarz inequality, we deduce that the map is $1$-Lipschitz.
The inclusions (\ref{proj_1-1}) follow from the definition of $\overline{\mathscr{D}}_1$.

There is a point $q'$ in $\mathbf{B}$, defined in (\ref{mathbfB}), such that $\dist_\omega(q,q')<d$.
From the inclusions (\ref{pi-1b}), the first inclusion in (\ref{proj_1-1}), and the fact that $\overline{\mathscr{D}}_1$ is $1$-Lipschitz, we deduce that $\overline{\mathscr{D}}_1(q)$ is at distance less than $d$ from $F_{(2)}/\lambda_\omega(\Gamma_{(1)})$. The last part of the lemma is proved. 
\end{proof}

\textbf{Interpolation:}

We can interpolate between the identity map and $\overline{\mathscr{D}}_1$ as follows. 
Let $\theta\in [0,1]$. For each $\sigma\in \Upsilon'_0(\Gamma_{(1)})$ and $\gamma\in \mathcal{T}(\Gamma_{(1)},\sigma)$, for every unit vector $v\in \lambda_\omega(\gamma).\mathcal{K}_\sigma \subset  H_{(1)}$, recall 
that $v$ and $\lambda_\omega(\gamma). \mathbf{u}_\sigma\in H_{(2)}$ are orthogonal inside $H_{(3)}$, namely they are at distance $\frac{\pi}{2}$ from each other in $S^\infty_{(3)}$.
Define
$$\overline{\mathscr{D}}_\theta(v)$$
to be the unique point $q$ on the minimizing geodesic segment between $v$ and $\lambda_\omega(\gamma). \mathbf{u}_\sigma$  inside $S^\infty_{(3)}$ such that $$\dist_\omega(q,v) = \frac{\theta\pi}{2}.$$ 
Now if, with respect to the previous decomposition (\ref{decomp Hsigma}), 
$$v:=v_{\mathrm{pf}} + \sum_{\sigma\in \Upsilon'_0(\Gamma_{(1)})} \sum_{\gamma\in \mathcal{T}(\Gamma_{(1)},\sigma)} v_{\gamma,\sigma}
$$
then set
$$\overline{\mathscr{D}}_\theta(v) := v_{\mathrm{pf}} + \sum_{\sigma\in \Upsilon'_0(\Gamma_{(1)})} \sum_{\gamma\in \mathcal{T}(\Gamma_{(1)},\sigma)} \|v_{\gamma,\sigma}\|\overline{\mathscr{D}}_\theta(\frac{v_{\gamma,\sigma}}{\|v_{\gamma,\sigma}\|})
$$
where by convention $\|v_{\gamma,\sigma}\|\overline{\mathscr{D}}_\theta(\frac{v_{\gamma,\sigma}}{\|v_{\gamma,\sigma}\|})$ means the $0$ vector if  $\|v_{\gamma,\sigma}\|=0$.
This yields a well-defined map
$$\overline{\mathscr{D}}_\theta:S^\infty_{(1)} \to S^\infty_{(3)}$$
such that $\overline{\mathscr{D}}_0=\Id$ and $\overline{\mathscr{D}}_1$ is indeed the map defined in (\ref{define overline proj1}).

We state the second main lemma of this subsection, whose proof is similar to Lemma \ref{projection main 1}:
\begin{lemme}\label{projection main 2}
For any $\theta\in [0,1]$, the map $\overline{\mathscr{D}}_\theta: S^\infty_{(1)} \to S^\infty_{(3)}$ is $\Gamma_{(1)}$-equivariant and $1$-Lipschitz.
Moreover, 
\begin{equation}\label{proj_1-1bis}
\begin{split}
\overline{\mathscr{D}}_\theta\Big( S^\infty_{(1)} \cap \bigcup_{\sigma\in \Upsilon'_0(\Gamma_{(1)})} \bigcup_{\gamma\in \mathcal{T}(\Gamma_{(1)},\sigma)} \lambda_\omega(\gamma).\mathcal{K}_\sigma \Big) & \subset S^\infty_{(3)} \cap \bigcup_{\sigma\in \Upsilon'_0(\Gamma_{(1)})} \bigcup_{\gamma\in \mathcal{T}(\Gamma_{(1)},\sigma)} \lambda_\omega(\gamma).(\mathcal{K}_{\sigma}\oplus \mathbb{R} \mathbf{u}_\sigma)
,\\
 \overline{\mathscr{D}}_\theta \Big(S^\infty_{(1)} \setminus \bigcup_{\sigma\in \Upsilon'_0(\Gamma_{(1)})} \bigcup_{\gamma\in \mathcal{T}(\Gamma_{(1)},\sigma)} \lambda_\omega(\gamma).\mathcal{K}_\sigma\Big) & \subset S^\infty_{(3)} \setminus \bigcup_{\sigma\in \Upsilon'_0(\Gamma_{(1)})} \bigcup_{\gamma\in \mathcal{T}(\Gamma_{(1)},\sigma)} \lambda_\omega(\gamma).(\mathcal{K}_{\sigma}\oplus \mathbb{R} \mathbf{u}_\sigma).
\end{split}
\end{equation}
\end{lemme}

\begin{remarque} \label{point of 36}
We chose $\Gamma_{(1)}$ so that it contains $\Gamma$, so $\mathcal{T}(\Gamma_{(1)},\sigma)$ is infinite for any $\sigma\in \Upsilon'_0(\Gamma_{(1)})$, or equivalently there is no maximal abelian subgroup of $\Gamma_{(1)}$ with finite index in $\Gamma_{(1)}$. 
The point of the inclusions (\ref{proj_1-1bis}) is then that, due to (\ref{decomp Hsigma}), 
$$S^\infty_{(1)} \cap \bigcup_{\sigma\in \Upsilon'_0(\Gamma_{(1)})} \bigcup_{\gamma\in \mathcal{T}(\Gamma_{(1)},\sigma)} \lambda_\omega(\gamma).\mathcal{K}_\sigma$$ is a union of closed disjoint subspheres  of infinite codimension inside $S^\infty_{(1)}$. 
Thus by (\ref{F_j subset}), given an integral current with compact support $S$ in $ S^\infty_{(1)}\cap \Omega_\omega$, we can easily approximate it by a polyhedral chain $P'$ in 
$$ S^\infty_{(1)} \setminus \bigcup_{\sigma\in \Upsilon'_0(\Gamma_{(1)})} \bigcup_{\gamma\in \mathcal{T}(\Gamma_{(1)},\sigma)} \lambda_\omega(\gamma).\mathcal{K}_\sigma\subset S^\infty_{(1)}\cap \Omega_\omega.$$
From Lemma \ref{projection main 2}, for any $\theta\in [0,1]$, $(\overline{\mathscr{D}}_\theta)_\sharp P'$ is then supported inside 
 $$ S^\infty_{(3)} \setminus \bigcup_{\sigma\in \Upsilon'_0(\Gamma_{(1)})} \bigcup_{\gamma\in \mathcal{T}(\Gamma_{(1)},\sigma)} \lambda_\omega(\gamma).(\mathcal{K}_{\sigma}\oplus \mathbb{R} \mathbf{u}_\sigma).$$
 
 \end{remarque}

\textbf{The quotient map $\mathscr{J}:$}

Consider the natural map
\begin{equation}\label{map J}
\mathscr{J}: S^\infty_{(3)}/\lambda_\omega(\Gamma_{(1)}) \to (S^\infty/\Gamma)_\omega=S^\infty_\omega/\lambda_\omega(\Gamma_\omega).
\end{equation}
A priori this map is not an isometric embedding, however we only need something weaker. 
Recall from (\ref{F_j subset}) that $\big(S^\infty_{(3)} \setminus \bigcup_{\sigma\in \Upsilon'_0(\Gamma_{(1)})} \bigcup_{\gamma\in \mathcal{T}(\Gamma_{(1)},\sigma)} \lambda_\omega(\gamma).(\mathcal{K}_{\sigma}\oplus \mathbb{R} \mathbf{u}_\sigma) \big) /\lambda_\omega(\Gamma_{(1)})$ is smooth. 
We check without difficulty, using (\ref{F_j subset}), (\ref{u perp u'}) and (\ref{F_{(2)}}), the following:
\begin{lemme}\label{J easy}
The map $\mathscr{J}$ is $1$-Lipschitz, and $\mathscr{J}(F_{(2)}/\lambda_\omega(\Gamma_{(1)}))$ is a countable discrete subset of $(S^\infty/\Gamma)_\omega$. Moreover the  restriction of $\mathscr{J}$ to 
$$\big(S^\infty_{(3)} \setminus \bigcup_{\sigma\in \Upsilon'_0(\Gamma_{(1)})} \bigcup_{\gamma\in \mathcal{T}(\Gamma_{(1)},\sigma)} \lambda_\omega(\gamma).(\mathcal{K}_{\sigma}\oplus \mathbb{R} \mathbf{u}_\sigma) \big) /\lambda_\omega(\Gamma_{(1)})$$
is a locally isometric immersion into the smooth part of $(S^\infty/\Gamma)_\omega$, namely  $\Reg (S^\infty/\Gamma)_\omega$.

\end{lemme}


\textbf{The final deformation maps:}

Consider the compositions of quotient maps
\begin{equation} \label{notation proj theta}
\begin{split}
\mathscr{D}_1 & : S^\infty_{(1)}/\lambda_\omega(\Gamma_{(1)}) \xrightarrow{\overline{\mathscr{D}}_1}
S^\infty_{(2)}/\lambda_\omega(\Gamma_{(1)}) \hookrightarrow S^\infty_{(3)}/\lambda_\omega(\Gamma_{(1)}) \xrightarrow{\mathscr{J}} (S^\infty/\Gamma)_\omega,\\
\mathscr{D}_\theta & :S^\infty_{(1)}/\lambda_\omega(\Gamma_{(1)}) \xrightarrow{\overline{\mathscr{D}}_\theta}
S^\infty_{(3)}/\lambda_\omega(\Gamma_{(1)}) \xrightarrow{\mathscr{J}} 
 (S^\infty/\Gamma)_\omega.
\end{split}
\end{equation}
They are well-defined by $\Gamma_{(1)}$-equivariance. Here $\mathscr{J}$ is the $1$-Lipschitz natural map (\ref{map J}), and $\overline{\mathscr{D}}_\theta$ is the quotient map induced by $\overline{\mathscr{D}}_\theta$. All the maps above are $1$-Lipschitz by Lemma \ref{projection main 1}, Lemma \ref{projection main 2}, Lemma \ref{J easy}. Besides, $\mathscr{D}_\theta$ is a continuous family of maps interpolating between $\mathscr{D}_0=\Id$ and $\mathscr{D}_1$. Finally, by (\ref{F_j subset}), Lemma \ref{projection main 1} and Lemma \ref{J easy},
\begin{equation}\label{3 janvier}
{\mathscr{D}}_1(\mathbf{A}) \cap \Sing(S^\infty/\Gamma)_\omega 
\end{equation}
is a discrete countable subset of $(S^\infty/\Gamma)_\omega$.


\begin{remarque} \label{point of 36 bis}
By Lemmas \ref{projection main 1} and  \ref{J easy}, given an integral current with compact support $S$ in the quotient of  $S^\infty_{(1)} \setminus \bigcup_{\sigma\in \Upsilon_0(\Gamma_{(1)})} \bigcup_{\gamma\in \mathcal{T}(\Gamma_{(1)},\sigma)} \lambda_\omega(\gamma).\mathcal{K}_\sigma$ by $\lambda_\omega(\Gamma_{(1)})$, we can approximate it with a polyhedral chain $P'$ with the following property: for any $\theta\in [0,1]$, 
$$\spt (\mathscr{D}_\theta)_\sharp P'\subset \Omega_\omega/\lambda_\omega(\Gamma_\omega)=\Reg(S^\infty/\Gamma)_\omega .$$
\end{remarque}

\section{The spherical Plateau problem for hyperbolic groups} \label{section:final}


\subsection{Preparations}

From now on, $\Gamma$ always denotes a non-elementary  torsion-free hyperbolic group. 
We continue to use the notations of Section \ref{subsection:projection}. 
Let $\mathbf{A}$ be any separable subset of $(S^\infty/\Gamma)_\omega$.
Recall from (\ref{F_j subset}) that
$$F_{(2)}=  S^\infty_{(2)} \cap \bigcup_{\sigma\in \Upsilon'_0(\Gamma_{(1)})} \bigcup_{\gamma\in \mathcal{T}(\Gamma_{(1)},\sigma)} \lambda_\omega(\gamma).\mathbb{R}\mathbf{u}_{\sigma}.$$
As we saw in  Subsection \ref{subsubsection:two spheres},
one of the key properties of the Hilbert spherical quotient  
$$S^\infty_{(2)}/\lambda_\omega(\Gamma_{(1)})$$ 
is that it is smooth  outside the discrete set of points $F_{(2)}/\lambda_\omega(\Gamma_{(1)})$. 
By (\ref{u perp u'}) and (\ref{F_{(2)}}), two distinct points of $F_{(2)}/\lambda_\omega(\Gamma_{(1)})$ are at distance at least $\frac{\pi}{2}$ from each other.
Around each point of $F_{(2)}/\lambda_\omega(\Gamma_{(1)})$, the spherical quotient  locally looks like a cone over a smooth Hilbert spherical quotient. 
Such a simple local geometry allow us to construct almost mass-non-increasing conical deformations of the type below. 

\begin{lemme}\label{conical retraction}
There is $\tilde{d}>0$ such that for any $\eta_1>0$ , there is $d(\eta_1)\in(0,\tilde{d})$ with the following property: for any $\eta_2\in(0,1)$,  there is a diffeomorphism 
$$\varphi : (S^\infty_{(2)}\setminus F_{(2)})/\lambda_\omega(\Gamma_{(1)})\to (S^\infty_{(2)}\setminus F_{(2)})/\lambda_\omega(\Gamma_{(1)})$$ 
 such that the following holds:
 \begin{enumerate}
 \item $\varphi$ is smoothly isotopic to the identity, 
 \item $\varphi = \Id$ outside of the $\tilde{d}$-neighborhood of 
 $F_{(2)})/\lambda_\omega(\Gamma_{(1)})$, 
 \item $\varphi$ is $(1+\eta_1)$-Lipschitz,
 \item  the restriction of $\varphi$ to the $d(\eta_1)$-neighborhood of 
$F_{(2)})/\lambda_\omega(\Gamma_{(1)})$ is $\eta_2$-Lipschitz.
\end{enumerate}
Moreover $\tilde{d}$ and $d(\eta_1)$ do not depend on $\Gamma_{(1)}$ or $S^\infty_{(2)}$.
\end{lemme}

\begin{proof}

First consider the toy case of $[0,\infty)$. Consider the Lipschitz bijection $\psi_0$ of $[0,\infty)$ defined as follows. Fix $\tilde{d}>0$, $\eta_1>0$, $\eta_2>0$ and $d(\eta_1)\in(0,\tilde{d})$. Set $\psi_0$ to be equal to $0$ at $0$, $\frac{\eta_2}{2}x$ at $d(\eta_1)$, $x$ for any $x\geq \tilde{d}$, and interpolate linearly between $0$, $d(\eta_1)$, and $d(\eta_1)$, $\tilde{d}$. We can smooth $\psi_0$ to a diffeomorphism $\psi_1$ of $[0,\infty)$ with the following properties whenever $d(\eta_1)$ is smaller than say $\frac{\eta_1 \tilde{d}}{4(1+\eta_1)}$: $\psi_1=\Id$ outside of $[0,\tilde{d}]$, $\psi_1$ is $(1+\frac{\eta_1}{2})$-Lipschitz and for any $\eta_2\in(0,1)$, the restriction of $\psi_1$ to $[0,d(\eta_1)]$ is $\eta_2$-Lipschitz. Note that here $d(\eta_1)$ only depends on $\eta_1$, not on $\eta_2$.

Now consider a point $p\in F_{(2)}/\lambda_\omega(\Gamma_{(1)})$. By the description of $F_{(2)}$ in (\ref{gamma'h(1)}), (\ref{F_{(2)}}), there is a neighborhood of $p$ which can be written as a cone over a smooth spherical quotient $Q$. More concretely, this neighborhood can be described as the quotient space
$$Q\times [0,2\tilde{d})/\sim$$ 
where $\tilde{d}>0$ is some positive constant and two points of $Q\times [0,2\tilde{d})$ are identified if and only if they both belong to $Q\times \{0\}$. The point corresponding to $Q\times \{0\}$ is $p$. 
Next, $Q\times [0,2\tilde{d})/\sim$ is endowed with a metric of the form $g := h^2(r) g_{\mathrm{round}}\oplus dr^2$ for the round unit Hilbert-Riemannian metric $g_{\mathrm{round}}$ on $Q$ and an explicit, increasing,  warping function $h(r)$ depending on $r$ (here $r$ denotes the distance to $p$). We have the first order expansion $h(r) = r+o(r)$.

We can apply the discussion for the toy case, and define for $\tilde{d}>0$, $\eta_1>0$, $\eta_2>0$ and $d(\eta_1)>0$, the following diffeomorphism  $\phi$ of the smooth part $Q\times (0,2\tilde{d})$:
$$\phi : Q\times (0,2\tilde{d})\to Q\times (0,2\tilde{d}),$$
$$\phi((q,r)):=(q,\psi_1(r))$$
where $\psi_1$ was defined in the first paragraph. As long as $\tilde{d}$ is chosen small enough so that $h(r)/r$ is close enough to $1$ when $r\in (0,\tilde{d})$, we obtain $\varphi$ as in the lemma by defining such a diffeomorphism around disjoint neighborhoods of all points $p \in F_{(2)}/\lambda_\omega(\Gamma_{(1)})$.

Finally, $\tilde{d}$, $d(\eta_1)$ do not depend on $\Gamma_{(1)}$ or $S^\infty_{(2)}$ since the size $\tilde{d}$ of the neighborhoods of $p \in F_{(2)}/\lambda_\omega(\Gamma_{(1)})$ and the warping function $h(r)$ ($r\in [0,\tilde{d}]$) do not depend on those parameters.

\end{proof}

The next tool is a ''closing lemma'' which enables to construct controlled fillings for integral currents in the thin part of $S^\infty/\Gamma$.
Let $\overline{\delta}$ be the positive constant found in Corollary \ref{margulis}.

\begin{lemme} \label{controlled filling}
Let $n\geq 2$ 
and let $T$ be an $(n-1)$-dimensional integral current with compact support inside $[S^\infty/\Gamma]^{\leq \overline{\delta}/4}$. 
Suppose that $ T=\partial W_0$ for some $n$-dimensional integral current $W_0$ with compact support inside $[S^\infty/\Gamma]^{\leq \overline{\delta}/4}$. 
Then for some $n$-dimensional integral current with compact support $W$ inside $S^\infty/\Gamma$, $T=\partial W$ and
$$\mathbf{M}(W) \leq \bar{c} \mathbf{M}(T)$$
for some positive constant $\bar{c}$ depending only on the dimension $n$,
and
$$W_0= W+\partial V$$
for some integral current $V$ with compact support in $S^\infty/\Gamma$.
\end{lemme}

\begin{proof}
Recall that by \cite[Lemma 1.6]{Antoine23a}, any integral current with compact support in $S^\infty/\Gamma$ is well approximated by polyhedral chains. Thus we only need to show the lemma for polyhedral chains. So let $P$ be an $(n-1)$-dimensional polyhedral chain in $[S^\infty/\Gamma]^{\leq \overline{\delta}/3}$, and assume that it bounds an $n$-dimensional polyhedral chain $Q$ with compact support inside $[S^\infty/\Gamma]^{\leq \overline{\delta}/3}$: $P = \partial Q.$
We can additionally assume that $\spt Q$ is connected since if not, we can apply the conclusion to each of the finitely many connected components of $\spt Q$.
We can also assume that each element of $\spt Q$ lifts to a function with finite support inside $S^\infty\subset \ell^2(\Gamma)$.

With those preparatory assumptions in place, we continue the proof of the lemma. By the Margulis type lemma, Corollary \ref{margulis}, there is a connected closed embedded curve $\gamma_0\subset S^\infty/\Gamma$ and a Lipschitz map 
$$\psi : \spt Q\to \gamma_0,$$
which is homotopic to the identity via a Lipschitz map 
$$\Phi : [0,1]\times \spt Q \to S^\infty/\Gamma$$
with $\Phi(0,.)=\Id$, $\Phi(1,.)=\psi$.
The push-forward integral current $\psi_\sharp P$ bounds the integral current $\psi_\sharp Q$ supported inside the curve $\gamma_0$. Of course, since $n\geq2$, the push-forward current $\psi_\sharp Q$ is trivial: $\psi_\sharp Q=0$. Thus $\psi_\sharp P = 0.$

The idea of the proof is then to choose a $\gamma_0$ of small length, and straighten the homotopy $\Phi$ without changing the end-maps, so that the new homotopy between $P$ and $\psi_\sharp P=0$ induces an integral current $W$ with mass controlled linearly by that of $P$.

We briefly recall the simple argument to choose $\gamma_0$ of length at most $1$. For any integer $N>0$, for any $g\in \Gamma$, the function $f$ in $S^\infty\subset \ell^2(\Gamma)$ equal to $\frac{1}{\sqrt{N}}$ on $\{g,...,g^N\}$ and equal to $0$ everywhere else is such that
$$\|\lambda_\Gamma(g).f -f\|^2 =\frac{2}{N}.$$  
Thus we can find a curve $\gamma_0$ as above with small length, in such a way that $\psi:\spt Q \to \gamma_0$ can be chosen to be $1$-Lipschitz.
Furthermore we can assume, by taking $N$ very large, that for any $x\in \spt Q$, the curve $\{\Phi(t,x)\}_{t\in [0,1]}$ is homotopic to a length minimizing geodesic segment joining $x$ to $ \Phi(1,x)=\psi(x)$ inside $S^\infty/\Gamma$ of length less than $\frac{2\pi}{3}$. This length condition ensures that the length minimizing geodesic segment is unique (for fixed endpoints and fixed homotopy class) and varies in a Lipschitz way with respect to $x\in \spt Q$. 

Next, we can suppose that the map $\psi$ is simplicial and ``linear'' on $\spt Q$, meaning that it maps $0$-dimensional faces in $\spt Q$ to one point $p_0\in \Gamma_0$, and it maps geodesics segments in $\spt Q$ to geodesic segments inside the curve $\gamma_0$ with respect to the metric induced on $\gamma_0$, in a linear way. 
We straighten the map
$$\Phi : [0,1]\times\spt Q \to S^\infty/\Gamma,$$
namely we replace, for each $x\in \spt Q$, the curve $\{\Phi(t,x)\}_{t\in [0,1]}$ by the unique length minimizing geodesic segment joining $x$ to $ \Phi(1,x)=\psi(x)$ inside $S^\infty/\Gamma$ (which has length less than $\frac{2\pi}{3}$ by choice of $\gamma_0$). This straightening yields a new Lipschitz map
$$\Phi' : [0,1]\times\spt Q \to S^\infty/\Gamma$$
such that $\Phi'(0,.)=\Id$, $\Phi'(1,.)=\psi$.
Set 
$$W:=-\Phi'_\sharp \llbracket 1_{[0,1]\times\spt P}\rrbracket + \psi_\sharp Q=-\Phi'_\sharp \llbracket 1_{[0,1]\times\spt P}\rrbracket,$$
and 
$$V:= - \Phi'_\sharp \llbracket 1_{[0,1]\times\spt Q}\rrbracket.$$
Here we put the  orientations on $[0,1]\times\spt Q$ and $[0,1]\times\spt P$ so that 
$$\partial \Phi'_\sharp \llbracket 1_{[0,1]\times\spt P} \rrbracket = \psi_\sharp P- P = -P,$$
$$\partial \Phi'_\sharp \llbracket 1_{[0,1]\times\spt Q}\rrbracket  = \psi_\sharp Q -Q - \Phi'_\sharp \llbracket 1_{[0,1]\times\spt P}\rrbracket .$$
Then 
$$\partial W=P\quad \text{and}\quad  Q=W+\partial V.$$ By elementary spherical geometry, and since $\psi:\spt Q\to \gamma_0$ was chosen to be $1$-Lipschitz, there is a positive constant $\bar{c}$ depending only on the dimension $n$ such that 
$$\mathbf{M}(W ) \leq \bar{c} \mathbf{M}(P).$$

\end{proof}

As in Section \ref{hilbert spherical quotients}, let $h\in H_n(\Gamma;\mathbb{Z})$,  let $C_i \in \mathscr{C}(h)$ be a minimizing sequence of cycles converging to a Plateau solution 
$$C_\infty = (X_\infty,d_\infty,S_\infty)$$
in the intrinsic flat topology. We consider the integral current $S_\infty$ as an integral current inside $ (S^\infty/\Gamma)_\omega$, see Lemma \ref{ultraembedding1}. We use the notation $U^{>\delta}$ introduced in Subsection \ref{subsection:nota}. Recall that $\mathbf{d}_\mathcal{F}$ denotes the intrinsic flat distance.
The following refinement of Lemma \ref{approx before limit} is needed: 
\begin{lemme} \label{approx before limit bis}
Let $0<\delta_1<\delta_2$.
Consider an $n$-dimensional integral current $U_0$ in $ (S^\infty/\Gamma)_\omega$ without boundary and an $(n+1)$-dimensional integral current $V_0$ in $ (S^\infty/\Gamma)_\omega$ such that
$$S_\infty=U_0+\partial V_0.$$
Suppose that $\spt V_0$ is compact and contained inside $\Reg(S^\infty/\Gamma)_\omega $. 
Suppose that $U_0=E_0+F_0$ for some integral currents $E_0,F_0$ such that $\spt E_0$  is compact and
$$\spt E_0\subset U^{>\delta_1},\quad \spt F_0\subset U^{\leq\delta_2}.$$
Let $\eta>0$. Then there exists integral currents $A,B$ with compact supports in $S^\infty/\Gamma$ such that 
$$A+B\in  \mathscr{C}(h)\quad \text{with}\quad \spt B\subset [S^\infty/\Gamma]^{\leq 2\delta_2},$$
$$\mathbf{d}_\mathcal{F}(A+B,U_0) \leq \eta \quad \text{and}\quad\mathbf{d}_\mathcal{F}(A,E_0) \leq \eta,$$
$$\mathbf{M}(A) \leq \mathbf{M}(E_0) +\eta \quad \text{and}\quad \mathbf{M}(\partial A) \leq \mathbf{M}(\partial E_0) +\eta.$$
\end{lemme}
\begin{proof}
The proof is very similar to that of Lemma \ref{approx before limit}. 
Roughly speaking, we consider polyhedral chains $P_1,P_2,R$ like in the proof of Lemma \ref{approx before limit}, and also another polyhedral chain $Z$ approximating $E_0$, such that $\partial Z$ approximates $\partial E_0$. By using the definition of ultralimits and applying the same lifting argument to those polyhedral chains, we construct a new element $\tilde{D}\in \mathscr{C}(h)$ close to $U_0$ in the intrinsic flat topology, and a decomposition $\tilde{D}=A+B$ satisfying the conclusion of the lemma. 

\end{proof}

\subsection{Volume continuity and avoidance of singularities}

Let $n\geq 2$, let $\Gamma$ be a torsion-free hyperbolic group, let $h\in H_n(\Gamma;\mathbb{Z})\setminus \{0\}$, and
 let $C'_i \in \mathscr{C}(h)$ be a minimizing sequence of cycles, 
 $$\lim_{i\to \infty} \mathbf{M}(C'_i) = \spherevol(h),$$
which converges to a spherical Plateau solution 
$$C_\infty = (X_\infty,d_\infty,S_\infty)$$
in the intrinsic flat topology. We view $S_\infty$ as an integral current inside $ (S^\infty/\Gamma)_\omega$.  



The volume continuity question asks: is it always true that $\lim_{i\to \infty} \mathbf{M}(C'_i) = \mathbf{M}(C_\infty)$?
The following theorem settles it, and its proof makes crucial use of the $1$-Lipschitz deformation maps. Consider the separable subset
$$\mathbf{A}:= \spt S_\infty \subset (S^\infty/\Gamma)_\omega$$
or more generally any separable subset containing $\spt S_\infty$, and let $\{\mathscr{D}_\theta\}_{\theta\in [0,1]}$ be the deformation maps of Subsection \ref{subsection:projection}. Recall that $\mathscr{D}_1:= \mathscr{J}\circ \overline{\mathscr{D}}_1$ and all those maps are $1$-Lipschitz.

\begin{theo} \label{theorem:continuity}
We have
$$\mathbf{M}({C}_\infty) =\spherevol(h).$$
Moreover, the integral current $({\mathscr{D}}_1)_\sharp S_\infty$ in $(S^\infty/\Gamma)_\omega$, viewed as an integral current space, is a spherical Plateau solution for $h$.
\end{theo}

The above result will be a corollary of the following proposition:

\begin{prop} \label{proposition:plateau solution}
The integral current $({\mathscr{D}}_1)_\sharp S_\infty$ in $(S^\infty/\Gamma)_\omega$, viewed as an integral current space, is the intrinsic flat limit of a sequence $\{C_i\}_{i\geq 0} \subset \mathscr{C}(h)$ such that
$$\limsup_{i\to \infty}\mathbf{M}(C_i) \leq  \mathbf{M}(({\mathscr{D}}_1)_\sharp S_\infty).$$
\end{prop}

\begin{proof}

As in Subsection \ref{subsection:nota}, we denote by $N_\omega(d)$ and $N^c_\omega(d)$ the closed $d$-neighborhood of $\Sing (S^\infty/\Gamma)_\omega$ and its complement.
By the slicing theorem, for any $d_0>0$, there is $d\in (0,d_0)$ such that $S_\infty\llcorner N^c_\omega(d) $ is an integral current. We will fix $d_0$ depending on small constants called ${\delta}, \eta,\eta_1$ later in the proof.

Let $\delta\in (0,\frac{\overline{\delta}}{100})$ where $\overline{\delta}$ is given by Corollary \ref{margulis}, and let $\eta>0$. Set
$$W:=S_\infty\llcorner N^c_\omega(d).$$
By $1$-Lipschitzness of $\mathscr{D}_1$, if $d_0$ is small enough depending on $\eta$ and $S_\infty$, we have the following estimate for the intrinsic flat distance $\mathbf{d}_\mathcal{F}$:
\begin{equation}\label{close D1 W S}
\mathbf{d}_\mathcal{F}((\mathscr{D}_1)_\sharp S_\infty, ({\mathscr{D}}_1)_\sharp  W) \leq \mathbf{M}(S_\infty) - \mathbf{M}(W)\leq \eta.
\end{equation}

By Lemma  \ref{lemma:yinfinity}, $\spt(W \llcorner U^{>\delta/2})$ is compact. Applying Lemma \ref{apply coarea} to $W$, we get a corresponding open set $O_W\subset N^c_\omega(d/2)$ such that $W\llcorner O_W$ is a compactly supported integral current with: 
$$\mathbf{M}(W-W\llcorner O_W)\leq \eta,$$
$$\mathbf{M}(\partial (W-W\llcorner O_W))\leq \eta,$$
\begin{equation} \label{wudeltaow}
\spt (W-W\llcorner O_W) \subset U^{\leq \delta/2},
\end{equation}
and in particular from the above,
\begin{equation} \label{dftwow}
\mathbf{d}_\mathcal{F}(W, W\llcorner O_W)\leq 2\eta.
\end{equation}

By a standard approximation argument by polyhedral chains, there is an $n$-dimensional polyhedral chain $P$ in $O_W \cap S^\infty_{(1)}/\lambda_\omega(\Gamma_{(1)})$, and some integral currents with compact supports $I,J$ such that  
$$W\llcorner O_W = P + I+\partial J$$
where $$\mathbf{M}(I)\leq \eta,\quad \mathbf{M}(J)\leq \eta.$$
By (\ref{dftwow}), we get
\begin{equation} \label{dftp}
\mathbf{d}_\mathcal{F}(W, P)\leq 4\eta.
\end{equation}
Moreover, if $d_0$ is small enough with respect to $\delta$, we can ensure that 
\begin{equation} \label{udeltaa}
\spt ((W-W\llcorner O_W) + I) \subset U^{\leq \delta},
\end{equation}
\begin{equation}\label{spt P u delta}
\spt \partial P \subset U^{\leq \delta},
\end{equation}
\begin{equation}\label{mpeta}
\mathbf{M}(P)\leq \mathbf{M}(W) +\eta,
\end{equation}
\begin{equation}\label{partialmpeta}
\mathbf{M}((\partial P)\llcorner N^c_\omega(2d)) \leq 2\eta,
\end{equation}
where the first inclusion above comes from (\ref{wudeltaow}), and the last property comes from the fact that $(\partial W)\llcorner N^c_\omega(2d)=0$.

Next, by using the deformation maps $\{\mathscr{D}_\theta\}_{\theta\in [0,1]}$, by (\ref{proj_1-1bis}) and Remark \ref{point of 36 bis},  after perturbing $P$ if necessary, it is possible to construct an $(n+1)$-dimensional integral current $Q$ with compact support in $\Reg(S^\infty/\Gamma)_\omega $ interpolating between $P$ and $ (\mathscr{D}_1)_\sharp P$ in the following sense:   
$$P= (\mathscr{D}_1)_\sharp P + R+ \partial Q$$
for some $n$-dimensional integral current with compact support $R$ in $\Reg(S^\infty/\Gamma)_\omega $. 
Moreover, by (\ref{spt P u delta}) and since each $\mathscr{D}_\theta$ is distance nonincreasing, we can ensure that
\begin{equation} \label{udeltab}
 \spt R \subset U^{\leq \delta}.
 \end{equation}

Recall from Subsection \ref{subsection:projection} the two 1-Lipschitz maps  $\mathscr{J}$ and $\overline{\mathscr{D}}_1$ such  that 
$$\mathscr{D}_1 := \mathscr{J} \circ \overline{\mathscr{D}}_1.$$ 
By the last part of Lemma \ref{projection main 1},
\begin{equation}\label{2dneighborhood}
\text{$\overline{\mathscr{D}}_1(\spt((\partial P)\llcorner N_\omega(2d)))$ is contained  in the $3d$-neighborhood of  $F_{(2)}/\lambda_\omega(\Gamma_{(1)})$.}
\end{equation}
Consider $\eta_1,\eta_2>0$ two small constants which will be fixed below.  Now we choose $d_0$ so that $d_0< \frac{d(\eta_1)}{100}$ where $d(\eta_1)$ is given by Lemma \ref{conical retraction} (recall that $d\in (0,d_0)$ and $W:=S_\infty\llcorner N^c_\omega(d)$).
Using the $(1+\eta_1)$-Lipschitz diffeomorphism 
$$\varphi : (S^\infty_{(2)}\setminus F_{(2)})/\lambda_\omega(\Gamma_{(1)})\to (S^\infty_{(2)}\setminus F_{(2)})/\lambda_\omega(\Gamma_{(1)})$$ 
constructed in Lemma \ref{conical retraction}, we can squeeze $\spt((\partial P)\llcorner N_\omega(2d))$ towards $F_{(2)}/\lambda_\omega(\Gamma_{(1)})$, and find  other integral currents $Q'$, $R'$ with compact supports in $\Reg(S^\infty/\Gamma)_\omega $ such that 
$$(\mathscr{D}_1)_\sharp P = (\mathscr{J}\circ \varphi\circ \overline{\mathscr{D}}_1)_\sharp P +R' +\partial Q' $$ 
such that
\begin{equation} \label{udeltac}
 \spt R' \subset U^{\leq \delta},
 \end{equation}
$$ \partial (\mathscr{J}\circ \varphi\circ \overline{\mathscr{D}}_1)_\sharp P \subset U^{\leq \delta},$$
$$\mathbf{M}((\mathscr{J}\circ \varphi\circ \overline{\mathscr{D}}_1)_\sharp P) \leq (1+\eta_1)(\mathbf{M}((\mathscr{D}_1)_\sharp W) +\eta),$$
$$\mathbf{M}( \partial (\mathscr{J}\circ \varphi\circ \overline{\mathscr{D}}_1)_\sharp P) \leq \eta_2 \mathbf{M}(\partial P) +(1+\eta_1)\mathbf{M}((\partial P)\llcorner N^c_\omega(2d)).$$
In particular, since $\mathbf{M}(W)\leq \mathbf{M}(S_\infty)$ and since we have (\ref{partialmpeta}), if we first choose $\eta_1$ small enough depending on $\mathbf{M}(S_\infty)$ and $\eta$, and then choose $\eta_2$ small enough depending on $\mathbf{M}(\partial P)$, we obtain 
\begin{equation} \label{masseta1}
\mathbf{M}((\mathscr{J}\circ \varphi\circ \overline{\mathscr{D}}_1)_\sharp P) \leq  \mathbf{M}((\mathscr{D}_1)_\sharp W) +2\eta,
 \end{equation}
\begin{equation} \label{partialeta1}
\mathbf{M}( \partial  (\mathscr{J}\circ \varphi\circ \overline{\mathscr{D}}_1)_\sharp P) \leq 3\eta.
 \end{equation}
Besides, when $\eta_1$ is small enough depending on $\mathbf{M}(S_\infty)$, 
\begin{equation} \label{dftprojvarphi}
\mathbf{d}_\mathcal{F}((\mathscr{D}_1)_\sharp P, (\mathscr{J}\circ \varphi\circ \overline{\mathscr{D}}_1)_\sharp P)\leq \eta.
\end{equation}
For clarity, note that $\eta_2$ depends on $\eta_1$, then $\eta_1$ depends on $\eta$, and $d_0 \in (0,\frac{d(\eta_1)}{100})$ only depends on $\eta$ and $\eta_1$, but not on $\eta_2$.

By definition,
$$S_\infty =  S_\infty \llcorner N_\omega(d) +S_\infty\llcorner N^c_\omega(d)= S_\infty \llcorner N_\omega(d) +W.$$
Next set 
$$V_0 := J+Q+Q',$$
\begin{align*}
U_0& := S_\infty- \partial V_0 \\
& = S_\infty \llcorner N_\omega(d) + (W-W\llcorner O_W) +W\llcorner O_W - \partial V_0 ,\\
& = S_\infty \llcorner N_\omega(d) + (W-W\llcorner O_W) + I+R+R'+ (\mathscr{J}\circ \varphi\circ \overline{\mathscr{D}}_1)_\sharp P,
\end{align*} 
By compactness of $Q,Q',R,R'$, by perturbing a bit $R,R'$ if necessary, we can find $\delta' \in (0,\delta)$ such that 
 after setting
 \begin{equation}\label{wxcvb}
E_0 := (\mathscr{J}\circ \varphi\circ \overline{\mathscr{D}}_1)_\sharp P,\quad F_0=U_0-E_0,
 \end{equation}
 we have by (\ref{udeltaa}) (\ref{udeltab}) (\ref{udeltac}):
 $$\spt E_0 \subset U^{>\delta'},$$
 $$\spt F_0\subset U^{\leq \delta}.$$

We can now apply Lemma \ref{approx before limit bis} to $S_\infty = U_0+\partial V_0$, $U_0=E_0+F_0$, $\delta_1=\delta'$, $\delta_2=\delta$, and we get integral currents $A,B$ with compact support in $S^\infty/\Gamma$ such that
\begin{equation} \label{dfau0O}
\begin{split}
&A+B\in \mathscr{C}(h)\quad  \text{with}\quad\spt B\subset [S^\infty/\Gamma]^{\leq \overline{\delta}/4},\\
&\mathbf{d}_\mathcal{F}(A+B,U_0) \leq \eta,\quad \mathbf{d}_\mathcal{F}(A,E_0) \leq \eta,\\
&\mathbf{M}(A) \leq \mathbf{M}(E_0) +\eta, \quad \text{and}\quad\mathbf{M}(\partial A) \leq \mathbf{M}(\partial E_0) +\eta.
\end{split}
\end{equation}

Lemma \ref{controlled filling} enables us to fill-in the integral current $\partial A$, namely we obtain an integral current $\hat{B}$  with compact support in $S^\infty/\Gamma$ and ``homologous'' to $B$, such that 
$$A+\hat{B}\in \mathscr{C}(h),$$
$$\mathbf{M}(\hat{B}) \leq \bar{c} \mathbf{M}(\partial A).$$

Note that by  (\ref{dfau0O}), (\ref{wxcvb}) and (\ref{partialeta1}), 
\begin{align*}
\mathbf{M}(\partial A) &\leq \mathbf{M}(\partial E_0) +\eta\\
&=\mathbf{M}(\partial (\mathscr{J}\circ \varphi\circ \overline{\mathscr{D}}_1)_\sharp P) +\eta\\
& \leq 4\eta.
\end{align*}
Hence, we find that 
\begin{equation}\label{bhat3ceta}
\mathbf{M}(\hat{B})\leq 4\bar{c} \eta,
\end{equation}
 and so by (\ref{bhat3ceta}), (\ref{dfau0O}), (\ref{wxcvb}), (\ref{dftprojvarphi}), (\ref{dftp}), (\ref{close D1 W S}), we get
\begin{align} \label{dfproj1tab}
\begin{split}
\mathbf{d}_\mathcal{F}((\mathscr{D}_1)_\sharp S_\infty, A+\hat{B}) 
& \leq \mathbf{d}_\mathcal{F}((\mathscr{D}_1)_\sharp S_\infty,A) + 4\bar{c} \eta \\
& \leq \mathbf{d}_\mathcal{F}((\mathscr{D}_1)_\sharp S_\infty, E_0) + 4\bar{c} \eta +\eta \\
& \leq \mathbf{d}_\mathcal{F}((\mathscr{D}_1)_\sharp S_\infty, (\mathscr{D}_1)_\sharp P) + 4\bar{c} \eta +2\eta \\
& \leq \mathbf{d}_\mathcal{F}((\mathscr{D}_1)_\sharp S_\infty, (\mathscr{D}_1)_\sharp W) + 4\bar{c} \eta +6\eta \\
&\leq  4\bar{c} \eta +7\eta.
\end{split}
\end{align}
Next, by the  
(\ref{dfau0O}),  (\ref{bhat3ceta}) and (\ref{masseta1}), 
\begin{align*}
\mathbf{M}(A+\hat{B}) &\leq \mathbf{M}(A)+4\bar{c} \eta\\
&\leq \mathbf{M}((\mathscr{J}\circ \varphi\circ \overline{\mathscr{D}}_1)_\sharp P)+\eta +4\bar{c} \eta\\
& \leq \mathbf{M}((\mathscr{D}_1)_\sharp W) +3\eta +4\bar{c} \eta\\
& \leq \mathbf{M}((\mathscr{D}_1)_\sharp S_\infty) + 3\eta +4\bar{c}\eta.
\end{align*}
Since $\eta$ is arbitrarily small, this finishes the proof.

\end{proof}

\begin{proof}[Proof of Theorem \ref{theorem:continuity}]
By Proposition \ref{proposition:plateau solution}, for any spherical Plateau solution $C_\infty=(X_\infty,d_\infty,S_\infty)$, there is a sequence $\{C_i\}\subset  \mathscr{C}(h)$ converging in the intrinsic flat topology to $({\mathscr{D}}_1)_\sharp S_\infty$ such that 
$$\limsup_{i\to \infty}\mathbf{M}(C_i) \leq  \mathbf{M}(({\mathscr{D}}_1)_\sharp S_\infty).$$
In particular, 
$$\spherevol(h)\leq \mathbf{M}(({\mathscr{D}}_1)_\sharp S_\infty).$$
On the other hand, since ${\mathscr{D}}_1$ is $1$-Lipschitz and by lower semicontinuity of the mass under intrinsic flat convergence, 
$$\mathbf{M}(({\mathscr{D}}_1)_\sharp S_\infty)\leq \mathbf{M}(S_\infty)=\mathbf{M}(C_\infty) \leq \spherevol(h).$$
Thus in fact
\begin{equation}\label{===}
\mathbf{M}(({\mathscr{D}}_1)_\sharp S_\infty) = \mathbf{M}(S_\infty)=\mathbf{M}(C_\infty) = \spherevol(h)=\limsup_{i\to \infty}\mathbf{M}(C_i).
\end{equation}

\end{proof}

Next, we prove the following avoidance phenomenon: any spherical Plateau solution mostly avoid the singular set of $(S^\infty/\Gamma)_\omega$. So around almost every point, a spherical Plateau solution is a mass-minimizing current in a spherical domain.
\begin{theo} \label{theorem:avoidance}
The integral current $S_\infty$ satisfies
$$\mathbf{M}\big(S_\infty \llcorner \Reg(S^\infty/\Gamma)_\omega \big) =\mathbf{M}(S_\infty).$$
Furthermore, $({\mathscr{D}}_1)_\sharp S_\infty$ is a spherical Plateau solution for $h$ and  
$$\spt ({\mathscr{D}}_1)_\sharp S_\infty \cap \Sing(S^\infty/\Gamma)_\omega $$ is a discrete subset.
\end{theo}

\begin{proof}
From (\ref{===}) in the proof of Theorem \ref{theorem:continuity}, we know that
$$\mathbf{M}(S_\infty) = \mathbf{M}((\mathscr{D}_1)_\sharp S_\infty).$$
By Lemmas \ref{projection main 1}, \ref{J easy},
\begin{align}\label{<<>>}
\begin{split}
&\spt ({\mathscr{D}}_1)_\sharp S_\infty \cap \Sing(S^\infty/\Gamma)_\omega  \subset \mathscr{J}(F_{(2)}/\lambda_\omega(\Gamma_\omega)),\\
&\text{$\mathscr{J}(F_{(2)}/\lambda_\omega(\Gamma_\omega))$ is a countable discrete set.}\\
\end{split}
\end{align}
In particular,
$$\mathbf{M}\Big(({\mathscr{D}}_1)_\sharp S_\infty\Big) = \mathbf{M}\Big(({\mathscr{D}}_1)_\sharp S_\infty \llcorner \Reg(S^\infty/\Gamma)_\omega  \Big).$$
But since by Lemma \ref{J easy} and the first inclusion in (\ref{F_j subset}),
$$\mathscr{D}_1^{-1}(\Reg(S^\infty/\Gamma)_\omega )
\subset 
\Reg(S^\infty/\Gamma)_\omega ,$$
we conclude as desired that 
$$\mathbf{M}(S_\infty \llcorner \Reg(S^\infty/\Gamma)_\omega ) =\mathbf{M}(S_\infty).$$
Now, by Theorem \ref{theorem:continuity}, 
$({\mathscr{D}}_1)_\sharp S_\infty $ is a spherical Plateau solution (viewed as an integral current space of course). From (\ref{<<>>}), we conclude the proof.

\end{proof}

\subsection{Global minimality}

We keep the assumptions and notations of the previous subsection. The following improves Theorem \ref{theorem:locally mass minimizing} by showing that $S_\infty$ is globally mass-minimizing in $(S^\infty/\Gamma)_\omega$, even around the singular part.
\begin{theo} \label{theorem:main statement}
The integral current $S_\infty$ is a mass-minimizing $n$-cycle in $(S^\infty/\Gamma)_\omega$.

\end{theo}

Before starting the proof of Theorem \ref{theorem:main statement}, we state an extension of Lemma \ref{reduce to compact}: 
\begin{lemme} \label{reduce to compact bis}
Consider some $n$-dimensional integral currents $S_0$, $U_0$  without boundary, and an $(n+1)$-dimensional integral current $V_0$ in $(S^\infty/\Gamma)_\omega$ such that
$$S_0=U_0+\partial V_0.$$
Let $d>0$ and $\eta>0$. Suppose that the closure of $\spt S_0 \cap U^{>\delta}$ is compact for any $\delta>0$.
Then there exists integral currents $U_3$, $V_3$ in $(S^\infty/\Gamma)_\omega$ such that 
$\spt V_3$ is compact, and
$$S_0=U_3+\partial V_3,$$
$$\spt V_3\subset \Reg(S^\infty/\Gamma)_\omega .$$
Moreover, $U_3 = E_3+F_3$ where $E_3,F_3$ are integral currents such that $\spt E_3$ is locally compact, 
$$\spt F_3 \subset N_\omega(d),\quad  \spt E_3\subset N^c_\omega(d/2)$$
$$\mathbf{M}(E_3 )\leq \mathbf{M}(U_0) +\eta,.$$
\end{lemme}

\begin{proof}
The difference with Lemma \ref{reduce to compact} is that here 
$\spt V_0$ is not necessarily contained in $N_\omega^c(\hat{d})$ for some $\hat{d}>0$. 
Roughly speaking, we first apply Ekeland's variational principle as in Lemma \ref{reduce to compact} (without the need to rescale the metric) and find  integral currents $U_1,V_1$ in $(S^\infty/\Gamma)_\omega$ such that 
$$S_0=U_1+\partial V_1,$$
 $\spt U_1\cap  \Reg(S^\infty/\Gamma)_\omega $, $\spt V_1 \cap  \Reg(S^\infty/\Gamma)_\omega $ are locally compact, and the mass of $U_1$ is close to that of $U_0$. 
By the slicing theorem, we find a generic $d'\in (d/2,d)$ such that 
$$V'_1:=V_1\llcorner N^c_\omega(d')$$ is an integral current.
We then apply Lemma \ref{apply coarea} to $V'_1$, find the corresponding open subset $O_{V'_1}$ and set
$$V_3:= V'_1\llcorner O_{V'_1}, \quad U_3 = S_0-\partial V_3.$$
Then set 
$$E_3:= U_3 \llcorner N^c_\omega(d''), \quad F_3:=U_3-E_3$$
for some generic $d''\in (d',d)$. Local compactness of $\spt E_3$ readily follows from compactness of $V_3$ and compactness of the closure of $\spt S_0 \cap U^{>\delta}$ for any $\delta>0$.

\end{proof}

\begin{proof}[Proof of Theorem \ref{theorem:main statement}]

We only sketch the proof since it is technically very similar to that of Proposition \ref{proposition:plateau solution}.  We start with two integral currents $U_0,V_0$ inside $(S^\infty/\Gamma)_\omega$ such that
$$S_\infty = U_0+\partial V_0.$$
Suppose towards a contradiction that 
\begin{equation}\label{contrad <}
\mathbf{M}(U_0) < \mathbf{M}(S_\infty)\leq \spherevol(h).
\end{equation}

Let $\eta>0$ such that $(1+100\eta)\mathbf{M}(U_0) <\spherevol(h)$. Let $d>0$ be smaller than $\frac{\bar{\delta}}{100}$ where $\bar{\delta}$ is from  Corollary \ref{margulis} and smaller than $\frac{d(\eta_1)}{100}$ given by Lemma \ref{conical retraction} with $\eta_1:=\eta$. Recall from Lemma \ref{lemma:yinfinity} that the closure of $\spt S_\infty \cap U^{>\delta}$ is compact for all $\delta>0$. 
By (\ref{contrad <}) and Lemma \ref{reduce to compact bis}, we can find integral currents $U_3=E_3+F_3$ and  $V_3$  
such that $\spt V_3$ is compact, $\spt E_3$ is locally compact
$$S_\infty=U_3+\partial V_3,$$
$$\spt V_3 \subset  \Reg(S^\infty/\Gamma)_\omega ,$$
$$\spt F_3 \subset N_\omega(d), \quad \spt E_3\subset N_\omega^c(d/2),$$
$$
\mathbf{M}(E_3) < \spherevol(h).
$$

Consider the $1$-Lipschitz deformation maps $\{\mathscr{D}_\theta\}_{\theta\in [0,1]}$ associated this time with the separable subset
$$\mathbf{A} := \spt S_\infty\cup \spt V_3.$$
Let $U'_3:=(\mathscr{D}_1)_\sharp U_3$, $V'_3:= (\mathscr{D}_1)_\sharp V_3$, $E'_3:=(\mathscr{D}_1)_\sharp E_3$, $F'_3 :=(\mathscr{D}_1)_\sharp F_3$ so that 
$$(\mathscr{D}_1)_\sharp S_\infty = U'_3+\partial V'_3,$$
$$U'_3 = E'_3+F'_3,$$
and from the previous paragraph:
\begin{equation}\label{contrad < bis}
\mathbf{M}(E'_3) < \spherevol(h).
\end{equation}

By Remark \ref{point of 36 bis}, after perturbing $U'_3,V'_3$ a bit if necessary, we can assume that 
$$\spt V'_3 \subset  \Reg(S^\infty/\Gamma)_\omega ,$$
that for some $d'\in(0,d)$,
$$\spt E'_3 \subset N_\omega^c(d'),$$
and since $\mathscr{D}_\theta$ are $1$-Lipschitz,
\begin{equation}\label{f'_3 subset}
\spt F'_3 \subset N_\omega(2d).
\end{equation}

Recall from Subsection \ref{subsection:projection}  that $\mathscr{D}_1 := \mathscr{J} \circ \overline{\mathscr{D}}_1.$
Since $\spt \partial E_3 \subset N_\omega(d)$, by the last part of Lemma \ref{projection main 1}, $\overline{\mathscr{D}}_1( \spt \partial E_3)$ is contained  in the $2d$-neighborhood of  $F_{(2)}/\lambda_\omega(\Gamma_{(1)})$.
With the help of the $(1+\eta)$-Lipschitz diffeomorphism 
$$\varphi : (S^\infty_{(2)}\setminus F_{(2)})/\lambda_\omega(\Gamma_{(1)})\to (S^\infty_{(2)}\setminus F_{(2)})/\lambda_\omega(\Gamma_{(1)})$$ 
constructed in Lemma \ref{conical retraction}, as in the proof of Proposition \ref{proposition:plateau solution}, 
we can additionally ensure without loss of generality that
\begin{equation}\label{petit bord}
\mathbf{M}(\partial E'_3) \leq \eta.
\end{equation}

Applying the compact approximation lemma, Lemma \ref{apply coarea}, to $W:=E'_3$, we find  an open  subset $O\subset N_\omega^c(d')$ such that $E'_3 \llcorner O$ has compact support  and
\begin{equation}\label{petit bord bis}
\begin{split}
\mathbf{M}(E'_3-E'_3\llcorner O)& \leq \eta,\\
\mathbf{M}(\partial(E'_3-E'_3 \llcorner O)) & \leq \eta.
\end{split}
\end{equation}
In particular 
\begin{equation}\label{E'_3 subset}
\spt E'_3 \llcorner O\subset U^{>\delta_1}
\end{equation}
for some $\delta_1>$. Since the closure of $\spt S_\infty \cap U^{>\delta}$ is compact for all $\delta>0$, and since $\mathscr{D}_1$ is $1$-Lipschitz,  the closure of $(\mathscr{D}_1)_\sharp S_\infty \cap U^{>\delta}$ is also compact for all $\delta>0$. Next $(\mathscr{D}_1)_\sharp \spt V_3$ is of course compact. 
The last sentence in Lemma \ref{apply coarea} then ensures that we can impose
\begin{equation}\label{E'_3 complement subset}
\spt (E'_3-E'_3 \llcorner O) \subset U^{\leq \delta_2}
\end{equation}
for some $\delta_2>\delta_1$ which  can be chosen to be equal to $d$.

Summarizing the previous manipulations, we have
$$(\mathscr{D}_1)_\sharp S_\infty = U'_3+\partial V'_3,$$
$$U'_3 = E'_3+F'_3 = E''+F'',$$
where 
$$E'':= E'_3\llcorner O,\quad F'' = (E'_3-E'_3\llcorner O) +F'_3.$$
Here $\spt E''$ is compact, and by (\ref{f'_3 subset}), (\ref{E'_3 subset}), (\ref{E'_3 complement subset}),
$$\spt E'' \subset U^{>\delta_1}, \quad \spt F'' \subset U^{\leq \delta_2}.$$
By (\ref{petit bord}) and (\ref{petit bord bis}),
$$\mathbf{M}(\partial E'') \leq 2 \eta.$$

At this point, note that Proposition \ref{proposition:plateau solution} and Theorem \ref{theorem:continuity}  (see (\ref{===}))  already imply that $(\mathscr{D}_1)_\sharp S_\infty$ is the intrinsic flat limit of a sequence $\{\hat{C}_i\}_i \subset \mathscr{C}(h)$ such that 
$$\lim_{i\to \infty}\mathbf{M}(\hat{C}_i) =\mathbf{M}((\mathscr{D}_1)_\sharp S_\infty)=\mathbf{M}(S_\infty)=\spherevol(h),$$
namely $(\mathscr{D}_1)_\sharp S_\infty$ is also a spherical Plateau solution for $h$.

Applying the approximation lemma, Lemma \ref{approx before limit bis}, to the spherical Plateau solution $(\mathscr{D}_1)_\sharp S_\infty $ instead of $S_\infty$, together with the closing lemma, Lemma \ref{controlled filling}, we obtain the existence of a cycle 
$C\in \mathscr{C}(h)$ such that 
$$\mathbf{d}_\mathcal{F}(C,E'')\leq \hat{\eta},$$
$$\mathbf{M}(C) \leq \mathbf{M}(E'') +\hat{\eta},$$
where $\hat{\eta}$ is a number converging to $0$ as $d, \eta$ converge to $0$. When $\eta$ and $\hat{\eta}$ are small enough, (\ref{contrad < bis}) and (\ref{petit bord bis}) lead to 
$$\mathbf{M}(C)  < \spherevol(h)$$
which contradicts the definition of $\spherevol(h)$. This ends the proof.

\end{proof}

To summarize, from Corollary \ref{positive spherevol}, Theorem \ref{theorem:continuity}, Theorem \ref{theorem:avoidance} and Theorem \ref{theorem:main statement}, we readily obtain the following general result, of which Theorem \ref{theorem:plateau solution} is a special case. 
\begin{theo}
Given a torsion-free hyperbolic group $\Gamma$ and $h\in H_n(\Gamma;\mathbb{Z})\setminus \{0\}$, 
there is a spherical Plateau solution for $h$, which is a mass-minimizing $n$-cycle $T$ in $(S^\infty/\Gamma)_\omega$, with mass equal to $\spherevol(h)>0$. Moreover, $\spt T\cap \Sing(S^\infty/\Gamma)_\omega $ is discrete.
\end{theo}
\vspace{1em}

\begin{remarque}\label{plausible}
We end this section with a discussion of Conjecture \ref{conjecture}:

\emph{Let $M$ be a closed oriented negatively curved Riemannian $n$-manifold with fundamental group $\Gamma$. 
Then, there exist an orthogonal representation $\rho: \Gamma\to \End(H)$ weakly equivalent to the regular representation $\lambda_\Gamma$, and a $\Gamma$-equivariant, indecomposable, $n$-dimensional minimal surface $\Sigma$ in  the unit sphere of $H$.}

Here, we informally use the term ``$n$-dimensional minimal surface'' to denote a closed subset $\Sigma$ such that any $x\in \Sigma$ is contained in an open neighborhood $\Omega$ with $\Sigma \cap \Omega=\spt S \cap \Omega$ for some $n$-dimensional mass-minimizing integral current $S$ without boundary in $\Omega$. Equivalently, we could also define it to be the support of a boundaryless, locally mass-minimizing integral $n$-current in the sense of Lang \cite{Lang11} (contrarily to \cite{AK00}, \cite{Lang11} allows integral currents to have infinite mass).
A minimal surface is said to be indecomposable if it is not the union of two strictly smaller minimal surfaces. 
We say that  $\rho$ is weakly equivalent to $\lambda_\Gamma$ if their complexifications are weakly equivalent as unitary representations in the sense of \cite[Definition F.1.1]{BDLHV08}.

In Theorem \ref{theorem:plateau solution}, it seems reasonable to expect that the mass-minimizing cycle $T$ in fact avoids completely the singular set of $(S^\infty/\Gamma)_\omega=S^\infty_\omega/\lambda_\omega(\Gamma_\omega)$. 
Its lift $\tilde{T}$ to the unit sphere $S^\infty_\omega$ would be an $n$-dimensional  minimal surface equivariant under $\Gamma_\omega$.   
The locally symmetric case \cite{BCG95} and the fact that negatively curved manifolds are aspherical suggest that one connected component $\Sigma$ of $\tilde{T}$ should be indecomposable, and invariant under $\Gamma$ viewed as a subgroup of $\Gamma_\omega$. 
The restriction $\lambda_\omega\vert_\Gamma$ is weakly equivalent to $\lambda_\Gamma$, see Remark \ref{derniere rem}.
What already follows from our results is that there is a closed subset $\Sigma\subset S^\infty_\omega$ which projects to $\spt T$, which is an $n$-dimensional minimal surface outside a discrete subset.

Conjecture \ref{conjecture} is partly motivated by the hyperbolization theorem, which states that closed, oriented manifolds of dimension $2$ or $3$, which admit a metric of negative curvature, in fact have a hyperbolic metric.  The naive analogue of this statement fails to hold in higher dimensions \cite{GT87}. In the conjecture, we replace hyperbolic metrics by minimal surfaces in spheres. Note that, if $M$ is a priori known to admit a hyperbolic metric, then the conjecture is true and the natural choice of minimal surface $\Sigma$ is intrinsically isometric to the hyperbolic n-plane up to scaling \cite{BCG95}\cite[Section 4.2]{Antoine23a}.

\end{remarque}

\section{Appendix}

\subsection{Integral currents in metric spaces}

The theory of metric currents initiated by Ambrosio-Kirchheim \cite{AK00} extends the classical theory of currents in $\mathbb{R}^n$ of de Giorgi, Federer-Fleming \cite{FF60}. The reader is referred to \cite[Section 1]{Antoine23a}, where we review the main notions used in this paper: integral currents in complete metric spaces, mass, weak and flat topologies, area and coarea formulas, slicing theorem, integral current spaces, intrinsic flat topology of Sormani-Wenger and the intrinsic flat distance $\mathbf{d}_\mathcal{F}$. Below is an additional  technical lemma, which embeds different converging sequences in a common Banach space.
\begin{lemme} \label{in a common Banach}
For $i\geq 0$ and $k\geq 0$, let $C_i$ be an $n$-dimensional integral current in a complete metric space $(E_i,d_i)$ of uniformly bounded diameter, and   consider 
$$\text{a Borel set }U_{k,i} \subset E_i,\quad  \text{a compact set }K_{k,i}\subset E_i,$$ such that $C_i\llcorner U_{k,i}$ is an integral current space. Suppose that 
$$\mathbf{M}(C_i)+\mathbf{M}(\partial C_i)\quad \text{ and } \quad \mathbf{M}(C_i\llcorner U_{k,i})+\mathbf{M}(\partial(C_i\llcorner U_{k,i}))$$
are uniformly bounded as $i\to \infty$.
Suppose that $\{C_i\}_{i\geq 0}$ and $\{C_i\llcorner U_{k,i}\}_{i\geq 0}$ converge in the intrinsic flat topology respectively to $(X_\infty,d_\infty,S_\infty)$ and $(W_k,d_k,T_k)$, and that $\{K_{k,i}\}_{i\geq 0}$ converges in the Gromov-Hausdorff topology to $K_k$. 

Then after extracting a subsequence, there exist a Banach space $\mathbf{Z}$ and isometric embeddings
$$\spt(S_\infty) \hookrightarrow \mathbf{Z}, \quad  \spt(T_k) \hookrightarrow \mathbf{Z}, \quad K_k \hookrightarrow \mathbf{Z},\quad  j_i : E_i \to \mathbf{Z},$$
such that inside $\mathbf{Z}$:
\begin{itemize}
\item
$(j_i)_\sharp(C_i)$ converges in the flat topology to $S_\infty$,
\item
 for all $k$,
$(j_i)_\sharp(C_i\llcorner U_{k,i})$ converges in the flat topology to $T_k$,  
\item
 for all $k$, $j_i(K_{k,i})$ converges in the Hausdorff topology to $K_k$.
 \end{itemize}
\end{lemme} 

\begin{proof}
This lemma is a corollary of the arguments in \cite{Wenger11}. Going through the proof of \cite[Theorem 1.2]{Wenger11}, one checks that it implies the following decomposition: for each $i\geq 0$,
$$C_i = F_i^1+...+F_i^i +R_i^i$$
where $F_i^q$ and $R_i ^i$ are integral currents such that for each fixed $q$, the sequence $\{\spt(F_i^q)\}_i$ is equi-compact with uniformly bounded diameter (see \cite[Definition 6.5]{AK00}). Moreover for each integer $L$,
\begin{equation} \label{mepj}
\limsup_{i\to \infty} \sum_{q=1}^L \mathbf{M}(F_i^q)+\mathbf{M}(\partial F_i^q) <\infty,
\end{equation}
and for any $\eta>0$, there is $L_\eta$ so that for all $L>L_\eta$ and $i>L$:
\begin{equation} \label{eer}
\mathbf{d}_\mathcal{F}(F_i^{L+1}+...+F_i^{i} +R^i_i,\mathbf{0}) \leq \eta
\end{equation}
where $\mathbf{d}_\mathcal{F}$ is the intrinsic flat distance and $\mathbf{0}$ is the zero integral current.
Similarly, for each $k\geq 0$ and $i\geq 0$, there are analogous decompositions
$$C_i \llcorner U_{k,i}= F_{k,i}^1+...+F_{k,i}^i +R_{k,i}^i.$$

For each $m\geq 0$, set
$$B_i^m :=  \big(\bigcup_{k=0}^m K_{k,i}\big) \cup \big(\bigcup_{q=0}^{\min\{m,i\}}\spt(F_i^q)\big) \cup \big(\bigcup_{k=0}^m \bigcup_{q=0}^{\min\{m,i\}}\spt(F_{k,i}^q)\big).$$
The sequence $\{B_i^m\}_{i\geq 0}$ is still equi-compact for each $m\geq 0$.
We can now apply \cite[Proposition 5.2 and Section 5]{Wenger11}: by extracting a subsequence with a diagonal argument and renumbering, we  conclude that there is a metric space $\mathbf{Z}$ (which can be assumed to be a Banach space by embedding into an $\ell^\infty$ space) and there are isometric embeddings 
$$j_i : (E_i,d_i) \hookrightarrow \mathbf{Z}$$
such that the following is true.
\begin{itemize}
\item $(j_i)_\sharp(C_i)$ converges in the flat topology to a limit integral current of $\mathbf{Z}$; this follows from (\ref{mepj}), (\ref{eer}), the compactness and closure theorems of \cite{AK00}, and \cite{Wenger07} as in the proof of \cite[Theorem 1.2]{Wenger11}. We identify that limit with $S_\infty$ by uniqueness of intrinsic flat limits \cite[Theorem 1.3]{Wenger11}.
\item Similarly for each $k>0$, $(j_i)_\sharp(C_i\llcorner U_{k,i})$ converges in the flat topology to a limit integral current of $\mathbf{Z}$ that we identify with $T_k$.
\item For each $k>0$, $j_i(K_{k,i})$ converges in the Hausdorff topology to a compact subset of $\mathbf{Z}$ which we identify with $K_k$ by uniqueness of compact Gromov-Hausdorff limits. 
\end{itemize}
\end{proof}

\subsection{Hyperbolic groups}\label{hyp groups}

The references \cite{GDLH90}\cite[Chapter 11]{DK18} provide introductions to hyperbolic groups, a fundamental class of groups in geometric group theory. 
Recall that a finitely generated group $\Gamma$ is hyperbolic if for some finite generating set $S\subset \Gamma$, the corresponding Cayley graph $\mathrm{Cay}(\Gamma,S)$ endowed with the word metric is a $\delta$-hyperbolic metric space for some $\delta\geq0$, in the sense that for any geodesic triangle $T$ in $\mathrm{Cay}(\Gamma,S)$ with sides $p,q,r$ and any point $ x\in p$, there exists $y\in q \cup r$ such that the word distance between $x$ and $y$ is at most $\delta$.

Let us assume that $\Gamma$ is a torsion-free hyperbolic group. Given a finite generating set $S$, there is a well-defined notion of boundary $\partial \mathrm{Cay}(\Gamma,S)$, defined by certain equivalence classes of geodesic rays in $\mathrm{Cay}(\Gamma,S)$. There is a standard topology on $\partial \mathrm{Cay}(\Gamma,S)$ which makes both $\partial \mathrm{Cay}(\Gamma,S)$ and $\mathrm{Cay}(\Gamma,S) \cup \partial \mathrm{Cay}(\Gamma,S)$ compact. This standard topology on $\partial \mathrm{Cay}(\Gamma,S)$ is induced by a metric (there are many possible choices for such a metric, but we arbitrarily choose one). The topology on the boundary does not depend on $S$, and will be denoted by $\partial \Gamma$. Thus 
$$\overline{\Gamma} :=\Gamma \cup \partial \Gamma $$ endowed with the natural topology is a compactification of $\Gamma$, which is identified with the vertices of the Cayley graph $\mathrm{Cay}(\Gamma,S) $ for an arbitrarily chosen finite generating set $S$.

Any group element $g\in \Gamma$ induces an isometric action on the boundary $\partial \Gamma$, and also on $\overline{\Gamma}$. Since $\Gamma$ is torsion-free and hyperbolic, it is known that when $g$ is not the identity element, $g$ has exactly two fixed points $a,b\in \partial \Gamma$. Moreover, if $\Gamma_a, \Gamma_b$ denote the stabilizers of $a,b$, then $\Gamma_a=\Gamma_b$ and they are equal to an infinite cyclic subgroup of $\Gamma$ containing $g$. 
In particular, if $g'$ does not commute with $g$, 
then $g'^N$ does not commute with $g$ for any positive integer $N$, and
\begin{equation}\label{disjoint pairs}
\{a,b\}\cap \{g'.a,g'.b\} =\varnothing.
\end{equation}
A related fact is that any nontrivial cyclic subgroup $H$ of $\Gamma$ are malnormal in th sense that for any $g\notin H$, $g^{1}Hg\cap H=\{1\}$.
The torsion-free hyperbolic group $\Gamma$ is called elementary if $\Gamma$ is the infinite cyclic group.

A theorem of Delzant \cite[Th\'{e}or\`{e}me I]{Delzant96}, (see also Chaynikov \cite{Chaynikov11} for a more complete proof) is important in our arguments:
\begin{theo}\label{theorem:delzant}
Let $\Gamma$ be a torsion-free hyperbolic group. There exists $k=k(\Gamma)$ such that for any two elements $x,y\in \Gamma$ which do not commute, the group generated by $x^k$ and $y^{-1}x^ky$ is a nonabelian free group of rank $2$.

\end{theo}

In this paper, we are not only interested in negatively curved closed manifolds, but more generally in torsion-free hyperbolic groups with nontrivial group homology over $\mathbb{Z}$. Recall that the group homology of $\Gamma$ is simply the homology of any classifying space for $\Gamma$. Various constructions of nontrivial homology classes of torsion-free hyperbolic groups are known. Of course, for any oriented closed $n$-manifold $M$ with negative curvature, the fundamental group is torsion-free hyperbolic and so the fundamental class $[M]\in H_n(M;\mathbb{Z})$. Examples of negatively curved non-locally symmetric manifolds include the metrics of Gromov-Thurston \cite{GT87} and the large family of examples due to Ontaneda  \cite{Ontaneda20} which are smoothings of the older general $\mathrm{CAT}(-1)$ topological manifolds of Charney-Davis \cite{CD95}. There are additionally closed oriented aspherical manifolds with hyperbolic fundamental group, but which cannot carry a negatively curved metric \cite{DJ91} (using the correction in  \cite{CD95}). There are several additional examples of homology classes of torsion-free hyperbolic groups which are not given by manifolds: Mosher-Sageev \cite{MS97}, and more generally Fujiwara-Manning \cite{FM10,FM11} studied the so-called ``cusp closing'' construction. 
Other general examples of homology classes of torsion-free hyperbolic groups which are not given by manifolds appear in 
the work of Januszkiewicz-Swiatkowski \cite[Corollary 19.3]{JS06}, Haglund \cite{Haglund03}, Osajda \cite{Osajda13}.
More examples can be constructed via the combination theorem in \cite{BF92}. 

\subsection{Regular representation and ultralimits} \label{prelim3}

In this paper we will need to study ``limits of isometric group actions on the Hilbert sphere''. To clarify such a concept, the notions of ultralimit  comes in handy. We also point out how $\Gamma$-limit groups come into play.

The definition of an ultralimit of a sequence of metric spaces is presented in details in Dru\c{t}u-Kapovich \cite[Chapter 10]{DK18}, see also Bridson-Haefliger \cite[Chapter I.5]{BH99}, \cite[Chapter 19]{DK18}. We recall some basic elements.  
By \cite[Definition 10.20]{DK18}, a \emph{non-principal ultrafilter} on the set of  nonnegative integers $\mathbf{N}$ is a finite additive probability measure $\omega$ such that all subsets $S\subset \mathbb{N}$ are  $\omega$-measurable, $\omega(S)\in \{0,1\}$, and $\omega(S)=0$ if $S$ is finite. 
We fix a non-principal ultrafilter $\omega$ on $\mathbb{N}$ once and for all.
Any bounded sequence of real numbers $a_n\in \mathbb{R}$ converges to a unique limit $l\in \mathbb{R}$ with respect to $\omega$, in the sense that $\omega\{n; |a_n-l| <\epsilon\} =1$ for any $\epsilon>0$. If $S^\infty$ denotes the unit sphere in the real separable Hilbert space $H$ endowed with the standard round Hilbert-Riemannian metric $\mathbf{g}_{\mathrm{Hil}}$, then one defines the \emph{ultralimit} of the constant sequence $S^\infty$ with respect to $\omega$ as follows, see \cite[Section 10.4]{DK18}. We consider the space of all possible sequences of points in $S^\infty$, and we take the quotient of that space by the equivalence relation which says that two sequences of points are equivalent if the limit of the distances between the corresponding points converges to $0$ with respect to $\omega$. Let $S^\infty_\omega$ denote that ultralimit. By \cite[Corollary 19.3]{DK18}, $S^\infty_\omega$ is the unit sphere of a (non-separable) Hilbert space $H_\omega$ which is the ultralimit of $H$. 

The notion of \emph{$\Gamma$-limit group} is also useful, see Sela \cite{Sela01}.
Let $\Gamma$ be an arbitrary group with identity element $1$. Given a group $G$ and a sequence of homomorphisms $\varphi_i:G\to \Gamma$, suppose that for any $g\in G$, either $\varphi_i(g)=1$ for almost all $i$ or $\varphi_i(g)\neq 1$ for almost all $i$. Then the quotient $G/\underset{\rightarrow}{\ker}(\varphi_i)$ is called a $\Gamma$-limit group, where 
$$ \underset{\rightarrow}{\ker}(\varphi_i):=\{g\in G; \quad \varphi_i(g)=1 \text{ for almost all $i$}\}.$$

To see how these notions play a role in our discussion, consider a countable group $\Gamma$.  
Let $\ell^2(\Gamma)$ be the space of $\ell^2$ real functions on $\Gamma$. Identify $S^\infty$ with
$$\{f:\Gamma\to \mathbb{R};\quad \| f \|_{\ell^2} =1\}.$$
The group $\Gamma$ acts isometrically on $S^\infty$ by the \emph{left regular representation} $\lambda_\Gamma$: for all $\gamma\in \Gamma, x\in \Gamma, f\in S^\infty$,
\begin{equation}\label{reg rep}
\lambda_\Gamma(\gamma).f(x) := f(\gamma^{-1}x).
\end{equation}
The isometric action $\lambda_\Gamma$ gives rise to a discrete group $\Gamma_\omega$ acting non-freely by isometries on the ultralimit $S^\infty_\omega$. The group $\Gamma_\omega$ is defined as follows. First, any sequence $\{\gamma_j\}_{j\geq 1} \subset \Gamma$ yields an isometry of $S^\infty_\omega$, by \cite[Lemma 10.48]{DK18}. Next, we say that two sequences are equivalent if they yield the same isometry of $S^\infty_\omega$, and $\Gamma_\omega$ is defined as the quotient of the space of all sequences in $\Gamma$ by this equivalence relation. 
In fact, we will see in the next lemma that $\Gamma_\omega$ is the ultraproduct of $\Gamma$ with respect to $\omega$, see  \cite[Section 10.3]{DK18}.
The isometric action of 
$\Gamma_\omega$ on $S^\infty_\omega$ determines an ``ultralimit'' orthogonal representation
$$\lambda_\omega : \Gamma_\omega\to O(H_\omega)$$
where $O(H_\omega) $ denotes the orthogonal group of $H_\omega$.

We observe the following simple properties: 
\begin{lemme} \label{limit group}
\begin{enumerate}
\item A sequence $\{g_j\}_{j\geq 1} \subset \Gamma$ is identified with the identity element $1$ of $\Gamma_\omega$ if and only if eventually $g_j=1$ with respect to the non-principal ultrafilter $\omega$, in the sense that $\omega(\{j; g_j=1\})=1$. Equivalently, $\Gamma_\omega$ is the ultraproduct of $\Gamma$ with respect to $\omega$.
\item If $\Gamma$ is torsion-free then so is $\Gamma_\omega$.
\item If maximal abelian subgroups of $\Gamma$ are malnormal, then so are the maximal abelian subgroups of $\Gamma_\omega$.
\item Any finitely generated subgroup of $\Gamma_\omega$ is a $\Gamma$-limit group. Conversely, any $\Gamma$-limit group is isomorphic to a subgroup of $\Gamma_\omega$.

\end{enumerate}
\end{lemme}
\begin{proof}

For (1): If  $g_j$ is not eventually equal to $1$ with respect to $\omega$, then consider the delta function $\delta_{g_j}\in S^\infty\subset \ell^2(\Gamma)$ and note that $\|\lambda_\Gamma(g_j). \delta_{g_j}-\delta_{g_j}\|$ is a constant sequence not equal to $0$ and so $\{g_j\}_j$ is not identified with the identity $1\in \Gamma_\omega$. 

Bullets (2) and (3) are elementary given (1).

For (4): Let $L$ be a finitely generated subgroup of $\Gamma_\omega$. Let $\langle a_1,...,a_{K}| r_1,r_2... \rangle$ be a presentation of $L$. Assume that $a_1,...,a_{K}$ are identified with respectively $\{g_{1,j}\}_{j\geq 1},...,\{g_{K,j}\}_{j\geq 1}$. Let $F_{K}$ be the free group of rank $K$, generated by $a_1,...,a_{K}$, and let $\varphi_j : F_{K} \to \Gamma$ be the homomorphism sending $a_k$ to $g_{k,j}$. Bullet (1) and a diagonal argument imply that after taking a subsequence of indices $j_i$, the direct limit $\underset{\rightarrow}{\ker}(\varphi_{j_i})$
exists and the $\Gamma$-limit group $L':=F_{K}/\underset{\rightarrow}{\ker}(\varphi_{j_i})$ is isomorphic to $L$. 

Conversely, let $L$' be a $\Gamma$-limit group equal to $F_{K}/\underset{\rightarrow}{\ker}(\varphi_{j})$, where $F_K$ is the free group generated by $a_1,...,a_K$, for some sequence of homomorphisms $\varphi_j : F_K \to \Gamma$. Let  $L$ be the subgroup of $\Gamma_\omega$ generated by the $K$ elements represented by the sequences $\{\varphi_j(a_1)\}_{j\geq 1},...,\{\varphi_j(a_K)\}_{j\geq 1}$. By (1) again, it follows that $L$ is isomorphic to $L'$.

\end{proof}

\subsection{Representation theory} \label{rep theory}
We review some elementary notions from the theory of orthogonal representations of discrete groups. Most of the literature concerns unitary representations on complex Hilbert spaces, while we only deal with orthogonal representations on real Hilbert spaces. Since an orthogonal representation can be complexified to yield a unitary representation, that distinction does not really matter for us. Good references are the books of Bekka-de la Harpe \cite{BDLH19} and Bekka-de la Harpe-Valette \cite{BDLHV08}.

Let $H$ be a (real) Hilbert space, $\End(H)$ the space of bounded linear operators on $H$, $O(H)$ the group of orthogonal transformations in $\End(H)$. Given a countable group $\Gamma$, an orthogonal representation $\rho$ of $\Gamma$ on $H$ is a homomorphism $\rho:\Gamma\to O(H)$. Equivalently, an orthogonal representation $\rho$ of $\Gamma$  is an isometric $\Gamma$-action on the unit sphere $S_H$ of $H$. We write either $\rho$ or $(\rho,H)$ to denote the representation. Given two orthogonal representations $(\rho,H), (\rho',H')$, we sometimes write $\rho=\rho'$ when they are equivalent, meaning that there is an orthogonal transformation $V:H\to H'$ such that $\rho'(g) = V\rho(g)V^{-1}$ for all $g\in G$. 

A subrepresentation of $(\rho,H)$ is given by the restriction of $\rho(G)$ to any $\rho$-invariant closed subspace $H_0$ of $H$. In that case, the orthogonal subspace of $H_0$ also yields  a subrepresentation. A direct sum of a family $\{(\rho_j,H_j)\}_j$ of orthogonal representations of $\Gamma$ is the representation $\bigoplus_j \rho_j$ on the Hilbert direct sum $\bigoplus_j H_j$. We always denote by $\mathbf{1}_\Gamma$ the trivial representation of $\Gamma$, which sends all of $\Gamma$ on the identity of the $1$-dimensional space $\mathbb{R}$ (sometimes written $\mathbb{R}\mathbf{u}$ where $\mathbf{u}$ is an abstract unit vector). 

An orthogonal representation $(\rho,H)$ will be called ``proper free'', if the action of $\rho(\Gamma)$ on the unit sphere $S_H$ is proper and free.

Given a subgroup $G$ of $\Gamma$, and a representation $(\rho,H)$, we denote by $\rho\vert_G$ the restriction of $\rho$ to $G$. If moreover $H_1$ is $\rho\vert_G$-invariant closed subspace of $H$, then $\rho\vert_{G,H_1}$ denotes the representation $(\rho\vert_G, H_1)$.
We will also need the notion of induced representations. If $G$ is a subgroup of $\Gamma$, a left transversal for $G$ is a subset $T$ of $\Gamma$ such that $\Gamma$ is the disjoint union $\bigcup_{t\in T} t.G$. Given an orthogonal representation $(\sigma,K_\sigma)$ of $G$,  an induced representation of $\sigma$ from $G$ to $\Gamma$ is an orthogonal representation $\rho$ of  $\Gamma$ on a Hilbert space $H$ statisfying:
\begin{enumerate}
\item there exists a closed subspace $K$ of $H$, invariant by $\rho(G)$, 
\item the restricted representation $(\rho\vert_H,K)$ is equivalent to $(\sigma,K_\sigma)$,
\item $H$ is the Hilbert direct sum $\bigoplus_{t\in T} \rho(t).K$ for some left transversal $T$ of $G$ in $\Gamma$.
\end{enumerate} 
Such an induced representation is denoted by $\Induced_G^\Gamma \sigma$.

\begin{remarque} \label{derniere rem}
The orthogonal representation $\lambda_\omega:\Gamma_\omega\to O(H_\omega)$ constructed in Subsection \ref{prelim3} satisfies the following properties:
$\Gamma$ naturally embeds as a subgroup of $\Gamma_\omega$ and if $\Gamma$ is torsion-free hyperbolic, the complexification of the restriction $\lambda_\omega\vert_\Gamma$ is weakly contained in the complex regular representation $\lambda_\Gamma:\Gamma\to \End(\ell^2(\Gamma,\mathbb{C}))$, in the sense of \cite[Definition F.1.1]{BDLHV08}. Indeed, for each $\gamma\in \Gamma$, the constant sequence $\gamma$ corresponds to an element of $\Gamma_\omega$, and by Lemma \ref{limit group}, this yields an  injective morphism $\Gamma\to \Gamma_\omega$. 
The second property is checked using the definition of $\lambda_\omega$ and \cite[Lemma F.1.3]{BDLHV08}. Note that if $\Gamma$ is torsion-free hyperbolic, then by $C^*$-simplicity of torsion-free hyperbolic groups \cite[Section 1 and Theorem 7]{DLH07}, the complexification of $\lambda_\omega\vert_\Gamma$ is actually weakly equivalent to $\lambda_\Gamma$.
\end{remarque}

\bibliographystyle{alpha}
\bibliography{biblio_24_01_08}

\end{document}